\documentclass[leqno,11pt]{amsart}
\usepackage{amssymb, amsmath}
\usepackage{amsthm, amsfonts,mathrsfs}
\usepackage{color}
\def\refer#1{~\ref{#1}}

\def\ccite#1{~\cite{#1}}

\def\inte#1{
\displaystyle\mathop{#1\kern0pt}^\circ }

\let\pa=\partial

\let\d=\delta
\let\e=\varepsilon

\let\r=\rho
\let\s=\sigma
\let\f=\frac

\let\p=\psi

\let\D=\Delta

\let\Om=\Omega

\def\cC{{\mathcal C}}

\def\cF{{\mathcal F}}

\def\cS{{\mathcal S}}

\def\pa{\partial}
\def\grad{\nabla}

\def\virgp{\raise 2pt\hbox{,}}
\def\cdotpv{\raise 2pt\hbox{;}}

\def\eqdefa{\buildrel\hbox{\footnotesize def}\over =}

\def\C{\mathop{\mathbb C\kern 0pt}\nolimits}
\def\DD{\mathop{\mathbb D\kern 0pt}\nolimits}
\def\EE{\mathop{{\mathbb E \kern 0pt}}\nolimits}
\def\K{\mathop{\mathbb K\kern 0pt}\nolimits}
\def\N{\mathop{\mathbb N\kern 0pt}\nolimits}
\def\Q{\mathop{\mathbb Q\kern 0pt}\nolimits}
\def\R{\mathop{\mathbb R\kern 0pt}\nolimits}
\def\SS{\mathop{\mathbb S\kern 0pt}\nolimits}
\def\ZZ{\mathop{\mathbb Z\kern 0pt}\nolimits}
\def\TT{\mathop{\mathbb T\kern 0pt}\nolimits}
\def\P{\mathop{\mathbb P\kern 0pt}\nolimits}

\newcommand{\Z}{{\ZZ}}

\def\dv{\mbox{div}}

\def\dive{\mathop{\rm div}\nolimits}
\def\curl{\mathop{\rm curl}\nolimits}

\def\Supp{\mathop{\rm Supp}\nolimits\ }

\def\no{\noindent}
\def\na{\nabla}
\def\p{\partial}

\newcommand{\beq}{\begin{equation}}
\newcommand{\eeq}{\end{equation}}
\newcommand{\ben}{\begin{eqnarray}}
\newcommand{\een}{\end{eqnarray}}
\newcommand{\beno}{\begin{eqnarray*}}
\newcommand{\eeno}{\end{eqnarray*}}
\newcommand{\andf}{\quad\hbox{and}\quad}
\newcommand{\with}{\quad\hbox{with}\quad}
\newtheorem{defi}{Definition}[section]
\newtheorem{thm}{Theorem}[section]
\newtheorem{lem}{Lemma}[section]
\newtheorem{rmk}{Remark}[section]
\newtheorem{col}{Corollary}[section]
\newtheorem{prop}{Proposition}[section]
\renewcommand{\theequation}{\thesection.\arabic{equation}}
\begin{document}
\title[Well-posedness of $3-$D MHD equations]
{Global well-posedness of $3-$D  density-dependent incompressible MHD equations with variable resistivity}

\bigbreak\medbreak
\author[H.  Abidi]{Hammadi Abidi}
\address[H.  Abidi]{D\'epartement de Math\'ematiques
Facult\'e des Sciences de Tunis
Universit\'e de Tunis EI Manar
2092
Tunis
Tunisia}\email{hammadi.abidi@fst.utm.tn}
\author[G. Gui]{Guilong Gui}
\address[G. Gui]{ School of Mathematics and Computational Science, Xiangtan University,  Xiangtan 411105,  China}\email{glgui@amss.ac.cn}
\author[P. Zhang]{Ping Zhang}
\address[P. Zhang]{Academy of Mathematics $\&$ Systems Science, The Chinese Academy of Sciences, Beijing 100190, China;\\
School of Mathematical Sciences, University of Chinese Academy of Sciences,
Beijing 100049, China} \email{zp@amss.ac.cn}

\setcounter{equation}{0}
\date{}

\maketitle
\begin{abstract}
In this paper, we investigate the global existence of weak solutions to 3-D inhomogeneous incompressible MHD equations with variable viscosity and resistivity, which is sufficiently close to $1$ in $L^\infty(\mathbb{R}^3),$  provided that the initial density is bounded from above and below by positive constants,  and both the initial velocity and  magnetic field are small enough in the critical space $\dot{H}^{\frac{1}{2}}(\mathbb{R}^3).$ Furthermore, if we assume in addition that the  kinematic viscosity equals $1,$ and both the initial velocity and  magnetic field belong to $\dot{B}^{\frac{1}{2}}_{2,1}(\mathbb{R}^3),$ we can also prove the uniqueness of such solution.
\end{abstract}


\vskip 0.2cm


\renewcommand{\theequation}{\thesection.\arabic{equation}}
\setcounter{equation}{0}

\renewcommand{\theequation}{\thesection.\arabic{equation}}
\setcounter{equation}{0}

\section{Introduction}
In this paper, we investigate  the global well-posedness of the following $3$-D inhomogeneous incompressible magnetohydrodynamics (MHD) equations:
\begin{equation}\label{1.2}
\begin{cases}
\displaystyle\partial_t\rho+\dv(\rho u)=0 \quad \mbox{in}\, \ \mathbb{R}^+\times \mathbb{R}^3,\\
\displaystyle\partial_t (\rho u)
+{\mathop{\rm div}}(\rho u\otimes u)
-\,2\dv\big(\mu(\rho)d\big)
+\nabla\Pi=(B\cdot\na)B,\\
\displaystyle\partial_t B+u\cdot\nabla B
+\curl\bigl(\sigma(\rho)\curl\, B\bigr)
=(B\cdot\nabla) u,\\
\displaystyle\dv\, u=\dv\, B=0,\\
\displaystyle(\rho,u,B)|_{t=0}=(\rho_0,u_0,B_0),
\end{cases}
\end{equation}
where $\rho$ and $u$ denote the density and  velocity  of the fluid, $B$ the magnetic field, and $\Pi=\pi +\frac{|B|^2}{2}$ with $\pi$ standing for
 the scalar pressure function of the fluid,  $d=d(u)=(d_{ij})_{3\times3},$ with $d_{ij}\eqdefa \frac{1}{2}(\partial_i u^j+\partial_ju^i),$
 designates the stress tensor, $\mu(\rho)>0$ and $\sigma(\rho)>0$ are the kinematic viscosity   and  the resistivity (the reciprocal of conductivity) of the fluid.

When the magnetic field $B=0$ in \eqref{1.2}, the system is reduced to the classical inhomogeneous incompressible Navier-Stokes equations
with variable viscosity \cite{LP1996}, which we denote by $(INS)$ below. The system \eqref{1.2} is a coupled system of the incompressible inhomogeneous Navier-Stokes equations with Maxwell's equations of electromagnetism, where the displacement current can be neglected \cite{ku-ly, lau-li}. The dependence of $\mu$ and $\sigma $ on $\rho$ in \eqref{1.2} enables us to consider the density-dependent equations as a model of a multi-phase flow consisting of several immiscible fluids with various viscosities and conductivities and without surface tension in presence of a magnetic field \cite{ger-br}.

Just as  $(INS),$  the system \eqref{1.2}  has
the following scaling-invariant property: if $(\rho, u, B)$ solves \eqref{1.2} with initial data $(\rho_0, u_0, B_0)$, then for any $\ell>0$,
\begin{equation}\label{S1eq1}
(\rho, u, B)_{\ell}(t, x) \eqdefa (\rho(\ell^2\cdot, \ell\cdot), \ell
u(\ell^2 \cdot, \ell\cdot), \ell
B(\ell^2 \cdot, \ell\cdot))
\end{equation}
is also a solution of \eqref{1.2} with initial data $(\rho_0(\ell\cdot),\ell
u_0(\ell\cdot), \ell
B_0(\ell\cdot))$. We call such functional spaces as critical spaces if the norms of
which are invariant under the scaling transformation \eqref{S1eq1}.

For the system $(INS)$ with constant viscosity, Lady\v zenskaja and Solonnikov  \ccite{LS}
first   considered the system
in bounded domain $\Om$ with homogeneous Dirichlet boundary
condition for $u.$ Under the assumption that $u_0$ belongs to~$
W^{2-\frac2p,p}(\Om)$ with $p$ greater than~$d$,  is divergence free and vanishes on
$\p\Om$ and that $\r_0$ is~$C^1(\Om)$,  bounded and  away from zero, then
they  proved
\begin{itemize}
\item Global well-posedness in dimension $d=2;$
\item Local well-posedness in dimension $d=3.$ If in addition $u_0$ is small in $W^{2-\frac2p,p}(\Om),$
then global well-posedness holds true.
\end{itemize}
Based on the energy law, Kazhikov \cite{KA1974} proved  that this system has a global weak solution in the energy space provided that the initial density is bounded from above and away from vacuum. For the system $(INS)$ with variable viscosity, Lions \cite{LP1996} proved the global existence of weak solutions with finite energy. Yet the uniqueness and regularities of such weak solutions are big open questions even in two space dimensions, as was mentioned by Lions in \cite{LP1996} (see pages 31-32 of \cite{LP1996}).

 In the critical framework for the system $(INS)$ with constant viscosity, under the smallness assumptions of $\rho_0-1$ and $u_0$,
  after the works \cite{A,A-P,DAN-03}, Danchin and Mucha \cite{DM1} eventually proved the global well-posedness of the system with initial density being close enough to a positive constant in the multiplier space of $\dot B^{-1+\f{d}p}_{p,1}(\R^d)$  and initial velocity being small enough in
  $\dot B^{-1+\f{d}p}_{p,1}(\R^d)$  for $1<p<2d.$ The work of \cite{A-G-Z-2} is the first to investigate the global well-posedness of the system with initial data in the critical spaces and yet without the size restriction on the initial density. The third author of this paper \cite{Zhang2020} proved the global existence of strong solutions to the system $(INS)$ with initial density being bounded from above and below by positive constants, and with initial velocity being sufficiently small in the critical Besov space $\dot{B}^{\frac{1}{2}}_{2, 1}(\mathbb{R}^3)$. This solution corresponds to the Fujita-Kato solution of the classical Navier-Stokes equations. The uniqueness of such solution was proved lately by Danchin and Wang \cite{DW2023}. Based on the improved uniqueness theorem and motivated by \cite{Zhang2020}, Hao et al. \cite{HSWZ2024} proved the global  existence of unique solution to the system $(INS)$ with  bounded initial density and initial velocity being sufficiently small in $\dot{H}^{\frac{1}{2}}(\mathbb{R}^3)$. More recently, we \cite{AGZ6-2024} proved that  3-D inhomogeneous incompressible Navier-Stokes equations has a unique global Fujita-Kato solution  if the initial velocity is sufficiently small in $\dot{B}^{\frac{1}{2}}_{2, \infty}(\R^3)$. One may check \cite{AGZ6-2024,HSWZ2024} and references therein for the most recent  progresses in this direction.

For the system $(INS)$ with variable viscosity, the problem turns out to be very difficult, there are only a few well-posedness results.
Under the additional assumptions that $\|\mu(\rho_0)-1\|_{L^\infty(\mathbb{T}^2)}\leq \varepsilon$  and  $u_0\in H^1(\mathbb{T}^2)$ for small $\varepsilon>0$, Desjardins \cite{des1997} proved that the global weak solution $(\rho, u, \nabla\Pi)$ constructed  in \cite{LP1996} satisfies $u \in L^{\infty}([0, T]; H^1(\mathbb{T}^2))$ for any $T>0.$ The first and third authors of this paper \cite{AZ1} improved the regularities of the solutions in \cite{des1997} and proved the uniqueness of such solution under additional regularity assumption on the initial density. They \cite{AZ2015-1} also established the global well-posedness of the 3-D incompressible inhomogeneous Navier-Stokes system with variable viscosity provided that the initial data $(\rho_0, u_0)$ satisfies $0<c_0\leq \rho_0\leq C_0$, $\rho_0-1 \in L^2\cap \dot{W}^{1, r}$, $u_0 \in \dot{H}^{-2\delta}$, $\d\in ]{1}/{4}, {1}/{2}[$, $r \in \bigl]6, \frac{3}{1-2\delta}\big[$, and $\|\mu(\rho_0)-1\|_{L^\infty}+\|u_0\|_{L^2}\|\nabla{u}_0\|_{L^2}$ small enough.

On the other hand,
there have been a lot of studies on magnetohydrodynamics by physicists and mathematicians because of their prominent roles in modeling many phenomena in astrophysics, geophysics and plasma physics, see  for instance \cite{abidi-h, cao-wu, des-br, du-lions, ger-br, huang-w, lin-zh, ser-te} and the references therein.

For the inhomogeneous incompressible MHD equations \eqref{1.2},  Gerbeau and Le Bris \cite{ger-br} (see also   Desjardins and  Le Bris \cite{des-br}) established the global existence of weak solutions to this system with  finite energy  in the whole
space $\mathbb{R}^3$ or in the torus $\mathbb{T}^3.$ Under the assumptions that both conductivity and viscosity are constants, Huang and Wang \cite{huang-w} demonstrated  the global existence of strong solutions to the 2-D  inhomogeneous incompressible MHD equations \eqref{1.2} with smooth initial data. The second author of this paper \cite{gui2014} proved that  2-D incompressible inhomogeneous MHD system \eqref{1.2} with  constant viscosity is globally well-posed for a generic family of the variations of the initial data and an inhomogeneous electrical conductivity.

 Motivated by \cite{DW2023, des1997,Zhang2020}, here we shall focus on the different effect of variations of the viscosity and resistivity to the existence of global unique solution to  the 3-D incompressible inhomogeneous MHD system provided that the initial density is bounded from above and below by positive constants and the initial velocity is sufficiently small in critical space.

In what follows, we shall always assume that
\begin{equation}\label{t.1}
0< m\leq \rho_0(x)\leq M  \quad \forall\,\, x \in \mathbb{R}^3,
\end{equation}
and
\begin{equation}\label{viscosite-conductivite}
0<\underline{\sigma}\leq \sigma(\rho_0),\quad 0<\underline{\mu}\leq\mu(\rho_0), \quad \sigma(\cdot),\,\,\mu(\cdot) \in W^{2, \infty}(\mathbb{R}^+)
\end{equation}
for some positive constants $m, \, M,\,\underline{\sigma},\, \underline{\mu}$.

The main results of this paper state as follows:

\begin{thm}\label{thm-GWS-MHD}
{\sl  Let $\rho_0$ and $\s(\rho_0), \mu(\rho_0)$ satisfy \eqref{t.1}-\eqref{viscosite-conductivite}, let
$(u_0,B_0)\in \dot{H}^{\frac{1}{2}}\times \dot{H}^{\frac{1}{2}}$ with $\dive\,u_0=\dv\,B_0=0.$
Then there exist positive constants $\frak{c}$ and $\varepsilon_0$ depending only on $m,\,M$ such that if
\begin{equation}\label{small-data-1}
\|(u_0,B_0)\|_{\dot{H}^{\frac{1}{2}}}\leq \frak{c}, \quad\|\mu(\rho_0)-1\|_{L^\infty}+\|\sigma(\rho_0)-1\|_{L^\infty}
\le
\varepsilon_0,
\end{equation}
the system \eqref{1.2} has a global  solution $(\rho, \, u,\, B,\, \nabla\Pi)$ with $\rho\in C_{\rm w}([0,\infty); L^{\infty})$ and $(u, B) \in (C([0, +\infty); \dot{H}^{\frac{1}{2}})\cap L^4(\mathbb{R}^+; \dot{H}^{1}))^2,$ which satisfies
\begin{equation}\label{bdd-density-visc-1}
0< m\leq \rho(t, x)\leq M, \quad 0<\underline{\sigma}\leq \sigma(\rho),\quad 0<\underline{\mu}\leq\mu(\rho) \quad \forall\, (t, x) \in \mathbb{R}^+\times \mathbb{R}^3,
\end{equation}
and
\begin{equation}\label{est-variable-2}
\begin{split}
&\|(u,B)\|_{\widetilde L^\infty_T(\dot H^{\frac{1}{2}})}+\|(u,B)\|_{L^\infty_T(L^3)}
+\|(\nabla u,\nabla B)\|_{L^4_T(L^2)}
+\|{t}^{-\frac{1}{4}}(\nabla u,\nabla B)\|_{L^2_T(L^2)}
\\&+\|{t}^{\frac{1}{4}}(\nabla u,\nabla B)\|_{L^\infty_T(L^2)}
+\bigl\|{t}^{\frac{1}{4}}
\bigl(u_t, B_t, \nabla\Pi-2\dv(\mu(\rho)d),\curl\left(\sigma(\rho)\curl\, B)\right)\bigr\|_{L^2_T(L^2)}
\\&
+\|{t}^{\frac{1}{4}}(\nabla u,\nabla B)\|_{L^2_T(L^6)}+\|(\nabla u,\nabla B)\|_{L^2_T(L^3)} +\|(u,B)\|_{L^2_T(L^\infty)}
\lesssim
\|(u_0,B_0)\|_{\dot H^{\frac{1}{2}}}.
\end{split}
\end{equation}
}
\end{thm}

We remark that Theorem \ref{thm-GWS-MHD} in particular improves the existence result for 3-D inhomogeneous incompressible Navier-Stokes equations
with variable viscosity in \cite{AZ2015-1} to the critical framework. As far as we know, Theorem \ref{thm-GWS-MHD}  maybe the first existence result for \eqref{1.2} with variable viscosity and resistivity and with initial data being in the critical spaces.

For the case when the kinematic  viscosity  $\mu$ is a positive constant, we can also prove the  uniqueness of such solution
 constructed in Theorem \ref{thm-GWS-MHD} provided that $(u_0,B_0)\in \dot B^{\frac{1}{2}}_{2,1}\times \dot B^{\frac{1}{2}}_{2,1}$.

\begin{thm}\label{mainthm-GWP}
{\sl Under the assumptions of Theorem \ref{thm-GWS-MHD},
the system \eqref{1.2} with $\mu(\rho)=1$  has a global  solution $(\rho, \, u,\, B, \,\nabla\Pi)$ with $\rho\in C_{\rm w}([0,\infty); L^{\infty})$ and $(u, B) \in (C([0, +\infty); \dot{H}^{\frac{1}{2}})\cap L^4(\mathbb{R}^+; \dot{H}^{1}))^2,$ which satisfies \eqref{bdd-density-visc-1}   and
\begin{equation}\label{est-const-basic-0}
\begin{split}
&\|(u,B)\|_{\widetilde{L}^\infty_T(\dot{H}^{\frac{1}{2}})}
+\|{t}^{\frac{1}{4}}\nabla({u},{B})\|_{L^{\infty}_T(L^2)}
+\|(\nabla{u},\nabla{B})\|_{L^4_T(L^2)}+\|{t}^{-\frac{1}{4}}(\nabla{u},\nabla{B})\|_{L^2_T(L^2)}\\
&+\bigl\|{t}^{\frac{1}{4}}\left(u_t, B_t, \nabla^2u,
\nabla\Pi, \curl(\sigma(\rho)\curl{B})\right)\bigr\|_{L^2_T(L^2)}+\|{t}^{\frac{1}{4}}(\nabla u,\nabla B)\|_{L^2_T(L^6)}\\
&+\|(\nabla u,\nabla B)\|_{L^2_T(L^3)} +\|(u,B)\|_{L^2_T(L^\infty)}+\|t^{\frac{3}{4}}\,(u_t, B_t)\|_{L^\infty_t(L^2)}+\|{t}^{\frac{3}{4}}(\nabla u,\nabla B)\|_{L^{\infty}_T(L^6)}\\
&+\|t^{\frac{3}{4}}\,(\nabla u_t,\nabla D_tu,\nabla D_tB)\|_{L^2_T(L^2)}+\|t^{\frac{1}{2}}(u,B)\|_{L^\infty_T(L^\infty)}+\|t^{\frac{1}{2}}(\nabla u, \nabla B)\|_{L^\infty_T(L^3)}\\
&+\|t^{\frac{1}{2}}\,u_t\|_{L^{2}_T(L^3)} +\|t^{\frac{3}{4}}\,(\nabla^2 u, \nabla\Pi)\|_{L^2_T(L^6)}
+\|t^{\frac{1}{2}}\,\nabla u\|_{L^2_T(L^\infty)} \lesssim
\|(u_0,B_0)\|_{\dot H^{\frac{1}{2}}}.
\end{split}
\end{equation}
 Furthermore, if in addition, $(u_0,B_0)\in \dot B^{\frac{1}{2}}_{2,1}\times \dot B^{\frac{1}{2}}_{2,1}$, then the solution is unique  and there holds
\begin{equation}\label{est-const-B21-0}
\begin{split}
&\|(u,B)\|_{\widetilde L^\infty_T(\dot B^{\frac{1}{2}}_{2,1})}
+\|u\|_{\widetilde L^2_T(\dot B^{\frac{3}{2}}_{2,1})}
+\|B\|_{\widetilde L^2_T(\dot B^{\frac{1}{2}}_{6,1})}+\|t^{\frac{1}{2}}(u_t,\,\nabla^2\,u,\nabla\Pi)\|_{L^{4, 1}_T(L^2)}\\
&+\|t\bigl(\nabla^2u,\nabla\Pi\bigr)\|_{L^{4,1}_T(L^6)}+\|(\nabla^2u,\nabla\Pi)\|_{L^1_T(L^3)}
+\|\nabla u\|_{L^1_T(L^\infty)}
\lesssim
\|(u_0,B_0)\|_{\dot B^{\frac{1}{2}}_{2,1}},
\end{split}
\end{equation}
where $L^{4, 1}_T$ denotes Lorentz norm with respect to the time variable (see Definition \ref{espace_lorentz}).
}
\end{thm}

\begin{rmk}\label{rmk-to-thm2-1}
We emphasize that different from \cite{HSWZ2024}, the estimate of $\|t^{\frac{1}{2}} \nabla{u}\|_{ L^2_T(L^{\infty})}$ in \eqref{est-const-basic-0} is not sufficient for us  to prove the uniqueness of the solution  to the system \eqref{1.2} constructed  in Theorem \ref{mainthm-GWP} because the resistivity $\sigma(\rho)$ varies with respect to the density function. In order to overcome this difficulty, here we derive the estimate of  $ \|\nabla{u}\|_{L^1_T(L^{\infty})}$ by using the maximal estimates $\|t^{\frac{1}{2}}\nabla^2{u}\|_{L^{4, 1}_T(L^2)}$ and
$\|t \nabla^2u\|_{ L^{4,1}_T(L^6)},$ which can derived under the additional assumption: $(u_0,B_0)\in \dot B^{\frac{1}{2}}_{2,1}\times \dot B^{\frac{1}{2}}_{2,1}$ (see \eqref{est-const-B21-4} in Sect. \ref{sect-apriori}). Moreover, the estimate of $ \|\nabla{u}\|_{L^1_T(L^{\infty})}$ makes it possible for us to solve the density patch problem of the 3-D inhomogeneous Navier-Stokes (or MHD) system in the future. 
\end{rmk}

The organization of this paper is as follows:

In the second section, we shall derive the {\it a priori} estimates of the system \eqref{1.2}, which are necessary to prove the existence part of
both Theorems \ref{thm-GWS-MHD} and \ref{mainthm-GWP}.

In the third section, we shall prove Theorem \ref{thm-GWS-MHD}.

In the fourth section, we shall prove Theorem\refer{mainthm-GWP}.

Finally in the appendix, we shall present a toolbox on basics of Littlewood-Paley theory
and Lorentz spaces.\\

Let us complete this section by the notations of the paper:

Let $A, B$ be two operators, we denote $[A;B]=AB-BA,$ the commutator
between $A$ and $B$. For $a\lesssim b$, we mean that there is a
uniform constant $C,$ which may be different on different lines,
such that $a\leq Cb$.  We denote by $(a|b)$ the $L^2(\R^3)$ inner
product of $a$ and $b,$ $(d_j)_{j\in\Z}$ (resp.
$(c_{j})_{j\in\Z}$) will be a generic element of $\ell^1(\Z)$
(resp. $\ell^2(\Z)$) so that $\sum_{j\in\Z}d_j=1$ (resp.
$\sum_{j\in\Z}c^2_{j}=1$).

For $X$ a Banach space and $I$ an interval of $\R,$ we denote by
${\cC}(I;\,X)$ the set of continuous functions on $I$ with values in
$X.$   For $q\in[1,+\infty],$ the
notation $L^q(I;\,X)$ stands for the set of measurable functions on
$I$ with values in $X,$ such that $t\longmapsto\|f(t)\|_{X}$ belongs
to $L^q(I).$ If in particular, $I=]0,T[,$ we denote the norm $\|f\|_{L^q(]0,T[;X)}$ by $L^q_T(X).$


\renewcommand{\theequation}{\thesection.\arabic{equation}}
\setcounter{equation}{0}

\setcounter{equation}{0}

\renewcommand{\theequation}{\thesection.\arabic{equation}}
\setcounter{equation}{0}

\section{{\it A priori} estimates}\label{sect-apriori}

In this section, we shall derive the {\it a priori} estimates for smooth enough solutions of \eqref{1.2}, which are necessary to prove the existence part of both Theorems \ref{thm-GWS-MHD} and  \ref{mainthm-GWP}. Motivated by \cite{Zhang2020}, we build the following scheme for the construction of the solutions to \eqref{1.2}.
 Let $(\rho, u, B, \nabla\Pi)$ be a smooth enough solution of \eqref{1.2} on $[0, T^{\ast}[$. Then it is easy to observe  from the continuity equation of \eqref{1.2} that for any $ T \in [0, T^{\ast}[$, there hold \eqref{bdd-density-visc-1}  and
\begin{equation}\label{visco-resis-1}
\begin{split}
\|\mu(\rho)-1\|_{L^{\infty}_T(L^\infty)}=\|\mu(\rho_0)-1\|_{L^\infty},\quad \|\sigma(\rho)-1\|_{L^{\infty}_T(L^\infty)}=\|\sigma(\rho_0)-1\|_{L^\infty}.
\end{split}
\end{equation}
 For any $j\in\Z,$ we define $(u_j, \nabla\Pi_j, B_j)$ through
 the following system:
\begin{equation}\label{model-3d-freq-1}
\begin{cases}
&\rho\bigl(\partial_t u_j
+(u\cdot\nabla)u_j\bigr)
-2\dv\big(\mu(\rho)d_j\big)
+\nabla\Pi_j=(B\cdot\na)B_j\quad \mbox{for} \ (t,x)\in [0,T^\ast[\times\R^3,\\
&\partial_t B_j+(u\cdot\nabla) B_j
+\curl\bigl(\sigma(\rho)\curl\, B_j\bigr)
=(B\cdot\nabla) u_j,\\
&\dv\, u_j=\dv\, B_j=0,\\
&(u_j,B_j)|_{t=0}=(\dot\Delta_ju_0,\dot\Delta_jB_0),
\end{cases}
\end{equation}
where $d_j\eqdefa \frac{1}{2}\left(\nabla{u}_j+(\nabla{u}_j)^T\right),$ the dyadic operator $\dot{\D}_j$ is recalled in Appendix A.

Then  we deduce from the uniqueness of local smooth solution to \eqref{1.2} that for any $t\in [0,T^\ast[$
\begin{equation}\label{identity-1}
\begin{split}
u(t)=\sum_{j \in \mathbb{Z}}u_j(t),\quad \nabla\,
\Pi(t)=\sum_{j \in \mathbb{Z}}\nabla\,\Pi_j(t) \quad
\mbox{and}\quad B(t)=\sum_{j\in\Z}B_j(t)\quad \mbox{in} \quad \mathcal{S}'_h.
\end{split}
\end{equation}

In what follows, we separate the {\it a priori} estimates of $(u,B,\na\Pi)$  into  $\mu(\rho)$ being variable  and constant cases.


\subsection{Variable viscosity case}

In view of \eqref{identity-1}, to establish the {\it a priori} estimates for $(u,B,\na\Pi),$ we need
first to derive the related estimates for $(u_j,B_j,\na\Pi_j),$  which we state as follows:

\begin{lem}\label{S2lem1} {\sl For asufficiently small positive constant $ \mathfrak{c}_2,$
we denote
\begin{equation}\label{small-assump-u-1}
T^\star\eqdefa \sup\bigl\{\ T\in ]0, T^{\ast}[: \ \|(u,B)\|_{L^\infty_{T}(L^3)}+\|\nabla u\|_{L^4_{T}(L^2)}
\leq \mathfrak{c}_2\ \bigr\}.
\end{equation}
Then under the assumptions of Theorem \ref{thm-GWS-MHD}, and for any $T\in [0, T^{\star}[$ and any $j \in \mathbb{Z}$,
one has
\begin{equation}\label{S2eq1}
\begin{split}
&\|(u_j,B_j)\|_{L^\infty_T(L^2)}+
\|(\nabla\,u_j,\nabla\,B_j)\|_{L^2_T(L^2)}
\lesssim
\|(\dot\Delta_ju_0,\dot\Delta_jB_0)\|_{L^2},\\
&\|(\nabla\,u_j,\nabla\,B_j)\|_{L^\infty_T(L^2)}+\|(\partial_tu_j,\partial_tB_j)\|_{L^2_T(L^2)}
\lesssim
2^j \|(\dot\Delta_ju_0,\dot\Delta_jB_0)\|_{L^2},
\end{split}
\end{equation}
and
\begin{equation}\label{est-variable-j-1}
\begin{split}
 &\|\sqrt{t}(\partial_tu_j,\partial_tB_j) \|_{L^2_T(L^2)}
+\|\sqrt{t}(\nabla\,u_j,\nabla\,B_j)\|_{L^\infty_T(L^2)}
\lesssim
\|(\dot\Delta_ju_0,\dot\Delta_jB_0\|_{L^2},\\
&\|{t}^{-\alpha}(\nabla u_j,\nabla B_j)\|_{L^2_T(L^2)}
\lesssim
2^{2j\alpha }\|(\dot\Delta_ju_0,\dot\Delta_jB_0)\|_{L^2} \quad \mbox{for any} \ \alpha\in[0,{1}/{2}[.
\end{split}
\end{equation}
}
\end{lem}
\begin{proof} As a convention in the proof of this lemma, we always assume that $t, T\leq T^\star.$ We divide the
proof of this lemma into the following steps:

\no{\bf Step 1.} The basic energy estimate of $(u_j,B_j).$

 We first get, by
taking the $L^2$ inner-product of the $u_j$ (resp. $B_j$) equations in \eqref{model-3d-freq-1} with $u_j$ (resp. $B_j$)
 and using integration by parts, that
\begin{equation*}\label{L2-ujBj-1}
\begin{split}
&\frac{1}{2}\frac{d}{dt}\|\sqrt{\rho}\,u_j(t)\|_{L^2}^2+2\int_{\mathbb{R}^3}  \mu(\rho)d_j : \nabla u_j\,dx=
\int_{\R^3} (B\cdot\nabla) B_j\cdot u_j\,dx,\\
&\frac{1}{2}\frac{d}{dt}\|B_j(t)\|_{L^2}^2+\int_{\mathbb{R}^3} \sigma(\rho)|\curl{B}_j|^2\,dx=
\int_{\R^3} (B\cdot\nabla) u_j\cdot B_j\,dx.
\end{split}
\end{equation*}
Due to $\dv\,B=0$, one has
\begin{equation*}\label{L2-uj-2}
\begin{split}
&\int_{\R^3} (B\cdot\nabla) B_j\cdot u_j\,dx+\int_{\R^3} (B\cdot\nabla) u_j\cdot B_j\,dx=0,\\
&2\int_{\mathbb{R}^3}  \mu(\rho)d_j : \nabla u_j\,dx=2\|\sqrt{\mu(\rho)}\,d_j\|_{L^2}^2,\end{split}
\end{equation*}
as a consequence, we obtain
\begin{equation*}\label{L2-uj-3}
\begin{split}
&\frac{1}{2}\frac{d}{dt}\|(\sqrt{\rho}\,u_j,\,B_j)(t)\|_{L^2}^2+2\|\sqrt{\mu(\rho)}\,d_j\|_{L^2}^2+\|\sqrt{\sigma(\rho)}\curl{B}_j\|_{L^2}^2=0,
\end{split}
\end{equation*}
which together with \eqref{bdd-density-visc-1} implies
\begin{equation*}\label{L2-uj-4}
\begin{split}
&\frac{1}{2}\frac{d}{dt}\|(\sqrt{\rho}\,u_j,\,B_j)(t)\|_{L^2}^2+2\underline{\mu}\|d_j\|_{L^2}^2+\underline{\sigma}\|\curl{B}_j\|_{L^2}^2\leq 0.
\end{split}
\end{equation*}
Due to $\dv\, u_j=\dv\,B_j=0$, we deduce that there exists a positive constant $c_0$ so that
\begin{equation*}\label{est-basic-1-1}
\begin{split}
\frac{d}{dt}\|(\sqrt{\rho} u_j,B_j)(t)\|_{L^2}^2
+c_0\|(\nabla\,u_j,\nabla\,B_j)\|_{L^2}^2
\leq
0.
\end{split}
\end{equation*}
By integrating the above equation over $[0,t]$ for $t\leq T$ and using \eqref{t.1} and \eqref{bdd-density-visc-1}, we obtain the first inequality of \eqref{S2eq1}.

\no{\bf Step 2.} The energy estimate of $\na u_j.$

Motivated by the derivation of (29) in \cite{des1997}, we get, by taking $L^2$ inner product of the momentum equation of \eqref{model-3d-freq-1} with $\partial_tu_j,$ that
\begin{equation*}\label{est-basic-2-1}
\begin{split}
\|\sqrt{\rho}\,\partial_tu_j\|_{L^2}^2&-\int_{\R^3}\dive\bigl(2\mu(\rho)
d_j\bigr)|\p_tu_j\,dx\\
& = -\int_{\R^3}\partial_tu_j|\bigl(\rho
u\cdot\nabla u_j\bigr)\,dx
+\int_{\R^3}(B\cdot\na)B_j| \partial_tu_j\,dx,
\end{split}
\end{equation*}
which follows that
\begin{equation}\label{est-basic-2-2}
\begin{split}
\|\sqrt{\rho}\,\partial_tu_j\|_{L^2}^2
-\int_{\R^3}\dive\bigl(2\mu(\rho)d_j\bigr)|\p_tu_j\,dx
\lesssim
\|(u,B)\|_{L^3}
\|(\nabla u_j, \nabla B_j)\|_{L^6}\|\sqrt{\rho}\,\partial_tu_j\|_{L^2}.
\end{split}
\end{equation}
By using integration by parts, we have
\begin{equation}\label{est-basic-2-3}
\begin{split}
-\int_{\R^3}\dive\bigl(2\mu(\rho) d_j\bigr)|\p_tu_j\,dx&=\int_{\R^3}2\mu(\rho) d_j:\p_td_j\,dx\\
&=\frac{d}{dt}\|\sqrt{\mu(\rho)}\, d_j(t)\|_{L^2}^2-\int_{\R^3}\p_t\bigl(\mu(\rho)\bigr) |d_j|^2\,dx.
\end{split}
\end{equation}
Yet by using  integration by parts, one has
\begin{equation*}
\begin{split}
&-\int_{\R^3}\p_t\bigl(\mu(\rho)\bigr) |d_j|^2\,dx=\int_{\R^3} u\cdot\na \mu(\rho) \,|d_j|^2\,dx\\
&=-\int_{\R^3} u\cdot\na ([\mu(\rho)]^{-1}) \,|\mu(\rho)\,d_j|^2\,dx=\sum_{i=1}^3\int_{\R^3}u^i
d_j:\p_i\bigl(2\mu(\r)d_j\bigr)\,dx.
\end{split}
\end{equation*}
Thanks to the symmetry of the matrix $d_j=(d^{k\ell}_{j})$, we obtain
\begin{equation*}
\begin{split}
&\sum_{i=1}^3\int_{\R^3}u^i
d_j:\p_i\bigl(2\mu(\r)d_j\bigr)\,dx
=\sum_{1\leq i,k,\ell\leq3}
\int_{\R^3}u^i\p_ku^\ell_j\p_i\bigl(2\mu(\r)d^{k\ell}_{j}\bigr)\,dx\\
&
=-\sum_{1\leq i,k,\ell\leq 3}\Bigl(\int_{\R^3}u^i
u^\ell_j\p_i\p_k\bigl(2\mu(\r)d_j^{k\ell}\bigr)\,dx+\int_{\R^3}\p_ku^i
u^\ell_j\p_i\bigl(2\mu(\r)d_j^{k\ell}\bigr)\,dx\Bigr),
\end{split}
\end{equation*}
which along with the fact $\dive u=0$ leads to
\begin{equation}\label{est-basic-2-5}
\begin{split}
&-\int_{\R^3}\p_t\bigl(\mu(\rho)\bigr) |d_j|^2\,dx\\
&
=\sum_{1\leq i,k,\ell\leq3}
\Bigl(\int_{\R^3}u^i\,\partial_iu^\ell_j\,\partial_k\bigl(2\mu(\rho)
d^{k\ell}_{j}\bigr)dx+ \int_{\R^3}2\mu(\rho)
\partial_ku^i\,\partial_iu^\ell_j\,d^{k\ell}_{j}\,dx\Bigr).
\end{split}
\end{equation}
By inserting \eqref{est-basic-2-5} into \eqref{est-basic-2-3}, we find
\begin{equation*}\label{est-basic-2-6}
\begin{split}
-\int_{\R^3}\dive\bigl(2\mu(\rho) d_j\bigr)|\p_tu_j\,dx&=\frac{d}{dt}\|\sqrt{\mu(\rho)}\, d_j(t)\|_{L^2}^2+\int_{\R^3} (u\cdot \nabla){u}_j\cdot\dv\,(2\mu(\rho)
d_{j})\,dx\\
&\quad
+\sum_{1\leq i,k,\ell\leq3} 2\int_{\R^3}\mu(\rho)\partial_ku^i\,\partial_iu^\ell_j\,d^{k\ell}_{j}\,dx,
\end{split}
\end{equation*}
from which and  the momentum equation of \eqref{model-3d-freq-1}, we infer
\begin{equation}\label{est-basic-2-7}
\begin{split}
-\int_{\R^3}\dive\bigl(2\mu(\rho) d_j\bigr)|\p_tu_j\,dx=&\frac{d}{dt}\|\sqrt{\mu(\rho)}\, d_j(t)\|_{L^2}^2+\int_{\R^3} (u\cdot \nabla){u}_j\cdot \nabla\Pi_j\,dx\\
&+\int_{\R^3} \bigl(u\cdot \nabla){u}_j\cdot(\rho\partial_tu_j+\rho
u\cdot\nabla u_j-(B\cdot\na)B_j\bigr)\,dx\\
&
+\sum_{1\leq i,k,\ell\leq3} 2\int_{\R^3}\mu(\rho)\partial_ku^i\,\partial_iu^\ell_j\,d^{k\ell}_{j}\,dx.
\end{split}
\end{equation}

By substituting  \eqref{est-basic-2-7} into \eqref{est-basic-2-2}, we achieve
\begin{equation}\label{est-basic-2-8}
\begin{split}
&\frac{d}{dt}\|\sqrt{\mu(\rho)}\, d_j(t)\|_{L^2}^2+\|\sqrt{\rho}\,\partial_tu_j\|_{L^2}^2
+\int_{\R^3} (u\cdot \nabla){u}_j\cdot \nabla\Pi_j\,dx\\
&\lesssim \|u\cdot \nabla {u}_j\|_{L^2}\|\sqrt{\rho}\partial_tu_j\|_{L^2}+\|u\cdot \nabla {u}_j\|_{L^2}^2+\|B\cdot \nabla {B}_j\|_{L^2}^2\\
&\quad+\|\nabla{u}\|_{L^2}\|\nabla{u}_j\|_{L^4}^2+
\|(u,B)\|_{L^3}
\|\nabla(u_j,B_j)\|_{L^6}\|\sqrt{\rho}\,\partial_tu_j\|_{L^2}.
\end{split}
\end{equation}
To deal with the term $\int_{\R^3} (u\cdot \nabla){u}_j\cdot \nabla\Pi_j\,dx$, we get, by taking
space divergence to the momentum equation of \eqref{model-3d-freq-1}, that
\begin{equation*}\label{b.4a}
\begin{split}
\Pi_j& =(-\Delta)^{-1}\dv\otimes\dv\bigl(2\mu(\rho)d_j\bigr)
-(-\Delta)^{-1}\dv\bigl(\rho\partial_tu_j+\rho u\cdot\nabla
u_j-B\cdot\nabla B_j\bigr),
\end{split}
\end{equation*}
from which, we infer
\begin{equation*}\label{est-basic-2-10}
\begin{split}
&\bigl|\int_{\R^3} (u\cdot \nabla){u}_j\cdot \nabla\Pi_j\,dx\bigr|=\bigl|\sum_{i,k=1}^3\int_{\R^3}
\Pi_j\partial_iu^k\partial_ku^i_j\,dx\bigr| \lesssim
\|\nabla u\|_{L^2}\|\nabla u_j\|_{L^4}^2\\
&\quad
+\bigl\|(-\Delta)^{-1}\dv\bigl(\rho\partial_tu_j
+\rho(u\cdot\nabla)u_j-B\cdot\nabla B_j\bigr)\bigr\|_{\dot H^1}
\bigl\|\sum_{i,k=1}^3\partial_i(u^k\partial_ku^i_j)\bigr\|_{\dot H^{-1}},
\end{split}
\end{equation*}
we thus obtain
\begin{equation*}\label{est-basic-2-11}
\begin{split}
|\int_{\R^3} (u\cdot \nabla){u}_j\cdot \nabla\Pi_j\,dx|
\lesssim
\|\nabla u\|_{L^2}\|\nabla u_j\|_{L^4}^2&+\|u\cdot \nabla {u}_j\|_{L^2}\|\sqrt{\rho}\partial_tu_j\|_{L^2}\\
&+\|u\cdot \nabla {u}_j\|_{L^2}^2+\|B\cdot \nabla {B}_j\|_{L^2}^2.
\end{split}
\end{equation*}
By substituting the above inequality into \eqref{est-basic-2-8} and using Young's inequality, we find
\begin{equation}\label{est-basic-2-13}
\begin{split}
&\frac{d}{dt}\|\sqrt{\mu(\rho)}\, d_j(t)\|_{L^2}^2+\frac{1}{2}\|\sqrt{\rho}\,\partial_tu_j\|_{L^2}^2
\lesssim
\|\nabla u\|_{L^2}\|\nabla u_j\|_{L^4}^2 +\|(u,B)\|_{L^3}^2
\|\nabla(u_j,B_j)\|_{L^6}^2.
\end{split}
\end{equation}

\no{\bf Step 3.} The energy estimate of $ \curl B_j.$

We  first get, by taking
$L^2$ inner product of the magnetic equation of \eqref{model-3d-freq-1} with
$\p_tB_j,$ that
\begin{equation*}\label{est-basic-2-14}
\begin{split}
\int_{\R^3}|\partial_tB_j|^2\,dx +\int_{\R^3}\curl\bigl(\sigma(\rho)
\curl\,B_j\bigr)| \p_tB_j\,dx =\int_{\R^3}\bigl(-
u\cdot\nabla{B}_j+B\cdot\nabla{u}_j\bigr)|\partial_tB_j\,dx.
\end{split}
\end{equation*}
By using integrating by parts and transport equation of $\s(\r),$ we  obtain
\begin{equation}\label{est-basic-2-15}
\begin{split}
&\f12\f{d}{dt}\|\sqrt{\sigma(\rho)}\,\curl{B}_j(t)\|_{L^2}^2+\|\partial_tB_j\|_{L^2}^2\\&
=-\f12\int_{\R^3}(u\cdot\nabla)\sigma(\rho)|\curl{B}_j|^2\,dx+\int_{\R^3}\bigl(-
u\cdot\nabla{B}_j+B\cdot\nabla{u}_j\bigr)|\partial_tB_j\,dx.
\end{split}
\end{equation}
Notice that
\begin{equation*}
\begin{split}
&-\int_{\R^3}(u\cdot\nabla)\sigma(\rho)|\curl{B}_j|^2\,dx
=\int_{\R^3} (u\cdot\nabla)[\sigma(\rho)^{-1}] |\sigma(\rho)\,\curl{B}_j|^2\,dx\\
&=-2\int_{\R^3}(u\cdot\nabla)\bigl(\sigma(\rho)\curl\,B_j\bigr)|\curl\,B_j\,dx\\
&
=-2\int_{\R^3}u^k\bigl(\partial_k(\sigma(\rho)\curl\,B_j)^{\ell}-\partial_{\ell}(\sigma(\rho)\curl\,B_j)^{k}\bigr)
 |(\curl\,B_j)^{\ell}dx\\
&\qquad
-2\int_{\R^3}\nabla u^k\wedge\bigl(\sigma(\rho)\curl\,B_j\bigr)| \partial_kB_j\,dx.
\end{split}
\end{equation*}
Due to the fact
\begin{equation*}
 |\bigl(\partial_k(\sigma(\rho)\curl\,B_j)^{\ell}-\partial_{\ell}(\sigma(\rho)\curl\,B_j)^{k}\bigr)| \lesssim |\curl\bigl(\sigma(\rho)\curl\,B_j\bigr)|,
 \end{equation*}
we get
\begin{equation}\label{est-basic-2-16}
\begin{split}
&|\int_{\R^3}(u\cdot\nabla)\sigma(\rho)|\curl{B}_j|^2\,dx|\\
&\lesssim \|u\|_{L^3} \|\curl\bigl(\sigma(\rho)\curl\,B_j\bigr)\|_{L^2} \|\curl\,B_j\|_{L^6} +\|\nabla u\|_{L^2} \|\sigma(\rho)\curl\,B_j\|_{L^4} \|\nabla{B}_j\|_{L^4}\\
&\lesssim \|u\|_{L^3}\|\nabla{B}_j\|_{L^6} \bigl(\|\partial_tB_j\|_{L^2}+\|u\cdot\nabla{B}_j\|_{L^2}+\|B\cdot\nabla{u}_j\|_{L^2}\bigr) +\|\nabla u\|_{L^2}\|\nabla{B}_j\|_{L^4}^2,
\end{split}
\end{equation}
where we used the equations $\curl\bigl(\sigma(\rho)\curl\,B_j\bigr)
=-\partial_tB_j-(u\cdot\nabla)B_j+(B\cdot\nabla)u_j$ in the last inequality.

 By substituting  \eqref{est-basic-2-16} into \eqref{est-basic-2-15}, we obtain
 \begin{equation*}\label{est-basic-2-17}
\begin{split}
&\f12\f{d}{dt}\|\sqrt{\sigma(\rho)}\,\curl{B}_j(t)\|_{L^2}^2+\|\partial_tB_j\|_{L^2}^2\\
&
\lesssim \|u\|_{L^3}\|\nabla{B}_j\|_{L^6} (\|\partial_tB_j\|_{L^2}+\|u\cdot\nabla{B}_j\|_{L^2}+\|B\cdot\nabla{u}_j\|_{L^2})\\
&\quad +\|\nabla u\|_{L^2}\|\nabla{B}_j\|_{L^4}^2+(\|u\cdot\nabla{B}_j\|_{L^2}+\|B\cdot\nabla{u}_j\|_{L^2})\|\partial_tB_j\|_{L^2}.
\end{split}
\end{equation*}
Applying Young's inequality gives
 \begin{equation}\label{est-basic-2-18}
\begin{split}
\f{d}{dt}\|\sqrt{\sigma(\rho)}\,\curl\,B_j(t)\|_{L^2}^2 &+\|\partial_tB_j\|_{L^2}^2\\
&\lesssim
\|\nabla u\|_{L^2}\|\nabla u_j\|_{L^4}^2 +\|(u,B)\|_{L^3}^2
\|(\nabla{u}_j, \nabla{B}_j)\|_{L^6}^2.
\end{split}
\end{equation}

\no{\bf Step 4.} The estimate of $\|(\nabla{u}_j, \nabla{B}_j)\|_{L^p}$ for $p=4,\,6$.

  We first write
  \begin{equation*}\label{est-basic-2-19}
\begin{split}
&\na u_j=\na(-\D)^{-1}\dv\mathbb{P} \bigl(2(\mu(\rho)-1)d_j\bigr)
-\na(-\D)^{-1}\dv\mathbb{P}\bigl(2\mu(\rho)d_j\bigr),\\
& \nabla B_j=-\na(-\D)^{-1}\curl([\sigma(\rho)-1]\curl B_j)+\na(-\D)^{-1}\curl(\sigma(\rho)\curl B_j)
\end{split}
\end{equation*}
where we used $\dv {B}_j=0$ so that  $\curl\curl B_j=-\Delta B_j.$

For $p\in [2, 6]$, we get, by  using \eqref{visco-resis-1} and the interpolation inequality: $\|f\|_{L^p(\mathbb{R}^3)}\lesssim \|f\|_{L^2(\mathbb{R}^3)}^{\f3p-\f12}\|\na f\|_{L^2(\mathbb{R}^3)}^{3\bigl(\f12-\f1p\bigr)},$ that
  \begin{equation*}\label{est-basic-2-20}
\begin{split}
&\|\nabla u_j\|_{L^p}
\le
C_1\Bigl( \|\mu(\rho_0)-1\|_{L^\infty}
\|\nabla u_j\|_{L^p} +\|\nabla u_j\|_{L^2}^{\frac{3}{p}-\frac{1}{2}}
\|\mathbb{P}\dv\bigl(2\mu(\rho)d_j\bigr)
\|_{L^2}^{3\bigl(\frac{1}{2}-\frac{1}{p}\bigr)}\Bigr),\\
&
\|\nabla B_j\|_{L^p}
\le C_1
\Bigl( \|\sigma(\rho_0)-1\|_{L^\infty}\|\nabla B_j\|_{L^p}
+\|\nabla B_j\|_{L^2}^{\frac{3}{p}-\frac{1}{2}}
\|\curl\bigl(\sigma(\rho)\curl B_j\bigr)
\|_{L^2}^{3\bigl(\frac{1}{2}-\frac{1}{p}\bigr)}\Bigr),
\end{split}
\end{equation*}
for some uniform constant $C_1$.

By taking $\e_0$
sufficiently small in \eqref{small-data-1}, we obtain for $2\leq p\leq 6$  that
  \begin{equation*}\label{est-basic-2-21}
\begin{split}
&\|\nabla u_j\|_{L^p}
\le
C \|\nabla u_j\|_{L^2}^{\frac{3}{p}-\frac{1}{2}}
\|\mathbb{P}\dv\bigl(2\mu(\rho)d_j\bigr)
\|_{L^2}^{3\bigl(\frac{1}{2}-\frac{1}{p}\bigr)},\\
&
\|\nabla B_j\|_{L^p}
\le  C\|\nabla B_j\|_{L^2}^{\frac{3}{p}-\frac{1}{2}}
\|\curl\bigl(\sigma(\rho)\curl B_j\bigr)
\|_{L^2}^{3\bigl(\frac{1}{2}-\frac{1}{p}\bigr)},
\end{split}
\end{equation*}
which along with the $u_j$ and $B_j$ equations in \eqref{model-3d-freq-1} implies
  \begin{equation*}\label{est-basic-2-22}
\begin{split}
\|\nabla u_j\|_{L^p}
 &\lesssim
\|\nabla u_j\|_{L^2}^{\frac{3}{p}-\frac{1}{2}}
\bigl\|(\rho\partial_tu_j+\rho(u\cdot\nabla)u_j-(B\cdot\nabla)B_j
)\bigr\|_{L^2}^{3(\frac{1}{2}-\frac{1}{p})} \\&
\lesssim \|\nabla u_j\|_{L^2}^{\frac{3}{p}-\frac{1}{2}}
\bigl(\|\sqrt{\rho}\partial_tu_j\|_{L^2}
+\|(u, B)\|_{L^3}
\|(\nabla u_j, \nabla B_j)\|_{L^6} \bigr)^{3(\frac{1}{2}-\frac{1}{p})}
\\
\|\nabla B_j\|_{L^p} &\lesssim
\|\nabla B_j\|_{L^2}^{\frac{3}{p}-\frac{1}{2}}
\bigl\|(\partial_tB_j+(u\cdot\nabla)B_j-(B\cdot\nabla)u_j)\bigr\|_{L^2}^{3(\frac{1}{2}-\frac{1}{p})}
\\&
\lesssim \|\nabla B_j\|_{L^2}^{\frac{3}{p}-\frac{1}{2}}
\bigl(\|\partial_tB_j\|_{L^2}
+\|(u, B)\|_{L^3}
\|(\nabla u_j, \nabla B_j\bigr)\|_{L^6} )^{3(\frac{1}{2}-\frac{1}{p})}.
\end{split}
\end{equation*}
In particular, we obtain
  \begin{equation*}\label{est-basic-2-22}
\begin{split}
&\|(\nabla u_j, \nabla B_j)\|_{L^4}
\lesssim
\|(\nabla u_j, \nabla B_j)\|_{L^2}^{\frac{1}{4}}\Bigl(
\|(\partial_tu_j, \partial_tB_j)\|_{L^2}^{\frac{3}{4}}
+\|(u, B)\|_{L^3}^{\frac{3}{4}} \|(\nabla u_j, \nabla B_j)\|_{L^6}^{\f34}\Bigr),\\
&\|(\nabla u_j, \nabla B_j)\|_{L^6}
\lesssim
\|(\partial_tu_j, \partial_tB_j)\|_{L^2}
+\|(u, B)\|_{L^3}
\|(\nabla u_j, \nabla B_j)\|_{L^6}.
\end{split}
\end{equation*}
Then under the assumption of \eqref{small-assump-u-1},
 we obtain, for any $t\in [0, T^{\star}[$
\begin{equation}\label{est-basic-2-23}
\begin{split}
&\|(\nabla u_j, \nabla B_j)(t)\|_{L^6}
\lesssim
\|(\partial_tu_j, \partial_tB_j)(t)\|_{L^2},\\
&\|(\nabla u_j, \nabla B_j)(t)\|_{L^4}
\lesssim
\|(\nabla u_j, \nabla B_j)(t)\|_{L^2}^{\frac{1}{4}}
\|(\partial_tu_j, \partial_tB_j)(t)\|_{L^2}^{\frac{3}{4}}.
\end{split}
\end{equation}

\no{\bf Step 5.} The closing estimate of $ (\na u_j, \na B_j).$

We first get,
by summing up the inequalities \eqref{est-basic-2-13} and \eqref{est-basic-2-18}, that
  \begin{equation}\label{est-basic-2-19a}
\begin{split}
&\f{d}{dt}\|(\sqrt{\mu(\rho)}\, d_j,\, \sqrt{\sigma(\rho)}\,\curl{B}_j)(t)\|_{L^2}^2 +2c_1\|(\partial_tu_j,\,\partial_tB_j)\|_{L^2}^2\\
&\lesssim
\|\nabla u\|_{L^2}\|\nabla u_j\|_{L^4}^2 +\|(u,B)\|_{L^3}^2
\|(\nabla{u}_j, \nabla{B}_j)\|_{L^6}^2.
\end{split}
\end{equation}
By substituting \eqref{est-basic-2-23} into \eqref{est-basic-2-19a}, we arrive at
  \begin{equation*}
\begin{split}
&\f{d}{dt}\|(\sqrt{\mu(\rho)}\, d_j,\, \sqrt{\sigma(\rho)}\,\curl{B}_j)(t)\|_{L^2}^2 +2c_1\|(\partial_tu_j,\,\partial_tB_j)\|_{L^2}^2\\
&\lesssim
\|\nabla u\|_{L^2}\|(\nabla u_j, \nabla B_j)\|_{L^2}^{\frac{1}{2}}
\|(\partial_tu_j, \partial_tB_j)\|_{L^2}^{\frac{3}{2}} +\|(u,B)\|_{L^3}^2
\|(\partial_tu_j, \partial_tB_j)\|_{L^2}^2.
\end{split}
\end{equation*}
Thanks to Young's inequality, one gets
  \begin{equation}\label{est-basic-2-24}
\begin{split}
&\f{d}{dt}\|(\sqrt{\mu(\rho)}\, d_j,\, \sqrt{\sigma(\rho)}\,\curl{B}_j)\|_{L^2}^2 +c_1\|(\partial_tu_j,\,\partial_tB_j)\|_{L^2}^2\\
&\leq C_3(
\|\nabla u\|_{L^2}^4\|(\sqrt{\mu(\rho)}\, d_j,\, \sqrt{\sigma(\rho)}\,\curl{B}_j)\|_{L^2}^2+\|(u,B)\|_{L^3}^2
\|(\partial_tu_j, \partial_tB_j)\|_{L^2}^2).
\end{split}
\end{equation}
By integrating the above inequality over $[0, T]$ for $T\in[0, T^\star[,$ we find
  \begin{equation*}\label{est-basic-2-24aa}
\begin{split}
&\|(\sqrt{\mu(\rho)}\, d_j,\, \sqrt{\sigma(\rho)}\,\curl{B}_j)\|_{L^{\infty}_T(L^2)}^2 +c_1\|(\partial_tu_j,\,\partial_tB_j)\|_{L^2_T(L^2)}^2
\\
&\leq C_4\|(\nabla \dot\Delta_ju_0, \nabla \dot\Delta_j B_0)\|_{L^2}^2+C_3
\|\nabla u\|_{L^4_T(L^2)}^4\|(\sqrt{\mu(\rho)}\, d_j,\, \sqrt{\sigma(\rho)}\,\curl{B}_j)\|_{L^{\infty}_T(L^2)}^2\\
&\qquad+C_3\|(u,B)\|_{L^{\infty}_T(L^3)}^2
\|(\partial_tu_j,\,\partial_tB_j)\|_{L^2_T(L^2)}^2.
\end{split}
\end{equation*}
Taking $\mathfrak{c}_2>0$ in \eqref{small-assump-u-1} to be so small that $\mathfrak{c}_2 \leq \min\{(\frac{c_1}{4C_3})^{\frac{1}{2}},\,(\frac{c_1}{4C_3})^{\frac{1}{4}}\}$, we obtain that
 \begin{equation*}\label{est-basic-2-26}
\|(\nabla{u}_j, \nabla{B}_j)\|_{L^{\infty}_T(L^2)}^2+\|(\partial_tu_j, \partial_tB_j)\|_{L^2_T(L^2)}^2
\lesssim
 \|(\nabla \dot\Delta_ju_0, \nabla \dot\Delta_j B_0)\|_{L^2}^2.
\end{equation*}
This together with Lemma \ref{lem2.1} leads to the second inequality of \eqref{S2eq1}.

Next let us turn to the proof of \eqref{est-variable-j-1}.
Indeed we get, by multiplying \eqref{est-basic-2-24} by $t,$ that
  \begin{equation*}\label{est-basic-2-28}
\begin{split}
&\f{d}{dt} \|t^{\frac{1}{2}}\,(\sqrt{\mu(\rho)}\, d_j,\, \sqrt{\sigma(\rho)}\,\curl{B}_j)(t)\|_{L^2}^2 +c_1\|t^{\frac{1}{2}}\,(\partial_tu_j,\,\partial_tB_j)\|_{L^2}^2\\
&\leq \|(\sqrt{\mu(\rho)}\, d_j,\, \sqrt{\sigma(\rho)}\,\curl{B}_j)\|_{L^2}^2 +C_3\|\nabla u\|_{L^2}^4\|t^{\frac{1}{2}}\,(\sqrt{\mu(\rho)}\, d_j,\, \sqrt{\sigma(\rho)}\,\curl{B}_j)\|_{L^2}^2\\
&\qquad+C_3\|(u,B)\|_{L^3}^2
\|t^{\frac{1}{2}}\,(\partial_tu_j, \partial_tB_j)\|_{L^2}^2.
\end{split}
\end{equation*}
Integrating the above inequality over $[0, T]$ for $T\in[0, T^\star[$ and using \eqref{S2eq1} and \eqref{small-assump-u-1}, we find
  \begin{equation*}\label{est-basic-2-29}
  \begin{split}
\|t^{\frac{1}{2}}\,(\nabla u_j, \nabla B_j)\|_{L^\infty_T(L^2)}^2+\|t^{\frac{1}{2}}\,(\partial_tu_j,\partial_tB_j)\|_{L^2_T(L^2)}^2
&\lesssim \|(\nabla u_j, \nabla B_j)\|_{L^2_T(L^2)}^2\\
&\lesssim
\|(\dot\Delta_ju_0,\dot\Delta_jB_0)\|_{L^2}^2,
\end{split}
\end{equation*}
which leads to the first inequality of \eqref{est-variable-j-1}.

Finally for $T\leq  T^{\star}$, $\alpha\in]0, {1}/{2}[,$ we observe from \eqref{S2eq1} that if $T \leq  2^{-2j}$
\begin{align*}
&\int_0^T{t}^{-2\alpha}\|(\nabla u_j,\nabla B_j)(t)\|_{L^2}^2\,dt\\
&\lesssim
\int_0^{T}{t}^{-2\alpha}\,dt \,2^{2j}\|(\dot\Delta_ju_0,\dot\Delta_jB_0)\|_{L^2}^2
\lesssim
{T}^{1-2\alpha}
2^{2j}\|(\dot\Delta_ju_0,\dot\Delta_jB_0)\|_{L^2}^2\\
&\lesssim 2^{4\alpha j}\|(\dot\Delta_ju_0,\dot\Delta_jB_0)\|_{L^2}^2 \quad (\mbox{since}\,\,1-2\alpha>0),
\end{align*}
and if $2^{-2j}\leq T\leq T^{\star}$
\begin{align*}
&\int_0^T{t}^{-2\alpha}\|(\nabla u_j,\nabla B_j)(t)\|_{L^2}^2dt\\
&\le
\int_0^{2^{-2j}}{t}^{-2\alpha}\|(\nabla u_j,\nabla B_j)(t)\|_{L^2}^2dt
+\int_{2^{-2j}}^T {t}^{-2\alpha}\|(\nabla u_j,\nabla B_j)(t)\|_{L^2}^2dt
\\&
\lesssim
(2^{-2j})^{1-2\alpha}
2^{2j}\|(\dot\Delta_ju_0,\dot\Delta_jB_0)\|_{L^2}^2
+(2^{-2j})^{-2\alpha}\|(\dot\Delta_ju_0,\dot\Delta_jB_0)\|_{L^2}^2\\
&\lesssim
2^{4\alpha j}\|(\dot\Delta_ju_0,\dot\Delta_jB_0)\|_{L^2}^2.
\end{align*}
As a consequence, we obtain the second inequality of \eqref{est-variable-j-1}.
 We thus complete
the proof of Lemma \ref{S2lem1}.
\end{proof}

\begin{prop}\label{prop-vaviable-1}{\sl Under the assumptions of Theorem \ref{thm-GWS-MHD}, for any $ T\in [0, T^{\ast}[$,
 there holds  \eqref{est-variable-2}.
}
\end{prop}
\begin{proof}
We first observe from Lemma \ref{S2lem1} that
  \begin{equation*}\label{est-basic-2-31}
  \begin{split}
&\|(\nabla u_j,\nabla B_j)\|_{L^\infty_T(L^2)}
\lesssim
2^j\|(\dot\Delta_ju_0,\dot\Delta_jB_0)\|_{L^2}
\lesssim c_{j}
2^{\frac{j}{2}}\|(u_0,B_0)\|_{\dot H^{\frac{1}{2}}},\\
&
\|\sqrt{t}(\nabla u_j,\nabla B_j)\|_{L^\infty_T(L^2)}
\lesssim
\|(\dot\Delta_ju_0,\dot\Delta_jB_0)\|_{L^2}
\lesssim c_{j}
2^{-\frac{j}{2}}\|(u_0,B_0)\|_{\dot{H}^{\frac{1}{2}}}.
\end{split}
\end{equation*}
 Here and below, we always denote $\left(c_j\right)_{j\in\Z}$ to be a generic element of $\ell^2(\Z)$ so that $\sum_{j\in\Z}c_j^2=1.$ As a consequence, we deduce that
 for any $t\leq T <  T^\star,$ which is determined by \eqref{small-assump-u-1},
  \begin{equation*}\label{est-basic-2-32}
\begin{split}
 \|t^{\frac{1}{4}} \nabla u(t)\|_{L^2}^2
&=\sum_{j,k\in\Z}\int_{\R^3} t^{\frac{1}{2}} \nabla u_j(t)\cdot\nabla u_k(t)dx
\le
2\sum_{k\in\Z}\|t^{\frac{1}{2}} \nabla u_k(t)\|_{L^2}\sum_{j\le k}\|\nabla u_j(t)\|_{L^2}\\
&\le
2\sum_{k\in\Z}\|t^{\frac{1}{2}} \nabla u_k\|_{L^{\infty}_T(L^2)}\sum_{j\le k}\|\nabla u_j\|_{L^{\infty}_T(L^2)} \\
&\lesssim \|(u_0,B_0)\|_{\dot H^{\frac{1}{2}}}^2 \sum_{k\in\Z}2^{-\frac{k}{2}}c_{k} \sum_{j\le k}2^{\frac{j}{2}}c_{j}
\lesssim
\|(u_0,B_0)\|_{\dot H^{\frac{1}{2}}}^2.
\end{split}
\end{equation*}
Along the same line, we obtain $\|t^{\frac{1}{2}} \nabla B(t)\|_{L^{\infty}_T(L^2)}^2
\lesssim
\|(u_0,B_0)\|_{\dot H^{\frac{1}{2}}}^2$ for any $T<T^\star.$ Hence we obtain
  \begin{equation}\label{est-basic-2-34}
\|t^{\frac{1}{2}} (\nabla u, \nabla B) \|_{L^{\infty}_T(L^2)}^2
\lesssim
\|(u_0,B_0)\|_{\dot H^{\frac{1}{2}}}^2\quad\mbox{for any}\ \ T <  T^\star.
\end{equation}

While for $\varepsilon\in(0,\frac{1}{4}),$  it follows from the second inequality of \eqref{est-variable-j-1} that for any $T < T^\star$
  \begin{align*}
&\|{t}^{-\left(\frac{1}{2}-\varepsilon\right)}(\nabla u_j,\nabla B_j)\|_{L^2_T(L^2)}
\lesssim
2^{j\left(\frac{1}{2}-2\varepsilon\right)}c_{j}\|(u_0,B_0)\|_{\dot H^{\frac{1}{2}}},\\
&
\|{t}^{-\varepsilon}(\nabla u_j,\nabla B_j)\|_{L^2_T(L^2)}
\lesssim
2^{j\left(2\varepsilon-\frac{1}{2}\right)}c_{j}\|(u_0,B_0)\|_{\dot H^{\frac{1}{2}}},
\end{align*}
which implies
  \begin{equation*}\label{est-basic-2-36}
\begin{split}
\|{t}^{-\frac{1}{4}}\nabla u\|_{L^2_T(L^2)}^2
&\le
2\sum_{k\in \mathbb{Z}}\sum_{j\le k}
\|{t}^{-(\frac{1}{2}-\varepsilon)}\nabla u_j\|_{L^2_T(L^2)}
\|{t}^{-\varepsilon}\nabla u_k\|_{L^2_T(L^2)}
\\&
\lesssim
\|(u_0,B_0)\|_{\dot H^{\frac{1}{2}}}^2
\sum_{k\in \mathbb{Z}}\sum_{j\le k}
2^{j\left(\frac{1}{2}-2\varepsilon\right)}c_{j}
2^{k\left(2\varepsilon-\frac{1}{2}\right)}c_{k}
\lesssim
\|(u_0,B_0)\|_{\dot H^{\frac{1}{2}}}^2.
\end{split}
\end{equation*}
The same procedure yields $\|{t}^{-\frac{1}{4}}\nabla B\|_{L^2_T(L^2)}\lesssim \|(u_0,B_0)\|_{\dot H^{\frac{1}{2}}}.$ As
a result, it comes out
  \begin{equation}\label{est-basic-2-37}
\begin{split}
\|{t}^{-\frac{1}{4}}(\nabla u, \nabla B)\|_{L^2_T(L^2)}
\lesssim
\|(u_0,B_0)\|_{\dot H^{\frac{1}{2}}}\quad\mbox{for any}\ \ T < T^\star.
\end{split}
\end{equation}

Thanks to \eqref{est-basic-2-34} and \eqref{est-basic-2-37}, we deduce that for any $T< T^\star$
  \begin{equation}\label{est-basic-2-38}
\begin{split}
\|(\nabla u,\nabla B)\|_{L^4_T(L^2)}
\le
\|{t}^{-\frac{1}{4}}(\nabla u,\nabla B)\|_{L^2_T(L^2)}^{\frac{1}{2}}
\|{t}^{\frac{1}{4}}(\nabla u,\nabla B)\|_{L^\infty_T(L^2)}^{\frac{1}{2}}
\lesssim
\|(u_0,B_0)\|_{\dot H^{\frac{1}{2}}}.
\end{split}
\end{equation}
Whereas it follows from Lemma \ref{S2lem1} and Lemma \ref{lem2.1} that for any $T< T^\star$
  \begin{equation*}\label{est-basic-2-39}
\begin{split}
&\|(u,B)\|_{\widetilde L^\infty_T(\dot H^{\frac{1}{2}})}^2
 \lesssim
\sum_{(j,k)\in\Z^2}2^k\|(\dot\Delta_ku_j,\dot\Delta_k B_j)\|_{L^\infty_T(L^2)}^2
\\&
\lesssim
\sum_{(j,k)\in\Z^2, j\le k}2^{-k}\|(\dot\Delta_k\nabla u_j,\dot\Delta_k\nabla B_j)\|_{L^\infty_T(L^2)}^2
+\sum_{(j,k)\in\Z^2, k\le j}2^k\|(\dot\Delta_ku_j,\dot\Delta_k B_j)\|_{L^\infty_T(L^2)}^2
\\&
\lesssim
\|(u_0,B_0)\|_{\dot H^{\frac{1}{2}}}^2
\sum_{(j,k)\in\Z^2}2^{-|k-j|}c_{k}^2c_{j}^2\lesssim
\|(u_0,B_0)\|_{\dot H^{\frac{1}{2}}}^2,
\end{split}
\end{equation*}
which together with  Sobolev  embedding: $\dot{H}^{\frac{1}{2}}(\mathbb{R}^3) \hookrightarrow L^3(\mathbb{R}^3),$  ensures that for any $T< T^\star$
  \begin{equation}\label{est-basic-2-40}
\begin{split}
&\|(u,B)\|_{L^\infty_T(L^3)}  \lesssim\|(u,B)\|_{\widetilde L^\infty_T(\dot H^{\frac{1}{2}})} \lesssim
\|(u_0,B_0)\|_{\dot H^{\frac{1}{2}}}^2.
\end{split}
\end{equation}
Thanks to \eqref{small-assump-u-1}, \eqref{est-basic-2-38} and \eqref{est-basic-2-40},  we get, by using the classical continuous  argument,  that $T^{\star}=T^{\ast}$ provided that $\frak{c}$ is sufficiently small in
\eqref{small-data-1}.

On the other hand, in view of \eqref{1.2}, we get, by a similar derivation of  \eqref{est-basic-2-23} and \eqref{est-basic-2-24}
that  for any $t \in ]0, T^{\ast}[$,
\begin{equation}\label{est-basic-2-41}
\begin{split}
&\|(\nabla u(t), \nabla B(t))\|_{L^6}
\lesssim
\|(u_t(t), B_t(t))\|_{L^2} \andf\\
&\|(\nabla\Pi-2\dv (\mu(\rho) d),\curl (\sigma(\rho)\curl{B}))(t)\|_{L^2}\\
&\lesssim \|(u_t, B_t)(t)\|_{L^2}+\|(u, B)\|_{L^\infty_t(L^3)}\|(\nabla u, \nabla B)(t)\|_{L^6} \lesssim \|(u_t, B_t)(t)\|_{L^2},
\end{split}
\end{equation}
and
  \begin{equation}\label{est-basic-2-42}
\begin{split}
\f{d}{dt}\|(\sqrt{\mu(\rho)}\, d,\, \sqrt{\sigma(\rho)}\,\curl{B})\|_{L^2}^2 +&c_2\|(u_t,\,B_t)\|_{L^2}^2\\
&\leq C_4
\|\nabla u\|_{L^2}^4\|(\sqrt{\mu(\rho)}\, d,\, \sqrt{\sigma(\rho)}\,\curl{B})\|_{L^2}^{2}.
\end{split}
\end{equation}
Multiplying \eqref{est-basic-2-42} by $t^{\f12}$ yields
  \begin{equation*}\label{est-basic-2-43}
\begin{split}
&\f{d}{dt}\|t^{\frac{1}{4}}\,(\sqrt{\mu(\rho)}\, d,\, \sqrt{\sigma(\rho)}\,\curl{B})\|_{L^2}^2 +c_2\|t^{\frac{1}{4}}\,(u_t,\,B_t)\|_{L^2}^2\\
&\leq C\|t^{-\frac{1}{4}}\,(\nabla u,\, \nabla{B})\|_{L^2}^2 + C_4
\|\nabla u\|_{L^2}^4\|t^{\frac{1}{4}}\,(\sqrt{\mu(\rho)}\, d,\, \sqrt{\sigma(\rho)}\,\curl{B})\|_{L^2}^{2}.
\end{split}
\end{equation*}
Applying Gronwall's inequality gives rise to
   \begin{equation*}\label{est-basic-2-44}
\begin{split}
 \|t^{\frac{1}{4}}\,(\sqrt{\mu(\rho)}\, d,\, \sqrt{\sigma(\rho)}\,&\curl{B})\|_{L^{\infty}_T(L^2)}^2 +\|t^{\frac{1}{4}}\,(u_t,\,B_t)\|_{L^2_T(L^2)}^2\\
&\lesssim C\|t^{-\frac{1}{4}}\,(\nabla u,\, \nabla{B})\|_{L^2_T(L^2)}^2 \exp\left(C_4\|\nabla u\|_{L^4_T(L^2)}^4 \right),
\end{split}
\end{equation*}
from which,  \eqref{est-basic-2-37} and \eqref{est-basic-2-38}, we infer
   \begin{equation}\label{est-basic-2-45}
\begin{split}
 \|t^{\frac{1}{4}}\,(\nabla u,\, \nabla{B})\|_{L^{\infty}_T(L^2)}^2 +\|t^{\frac{1}{4}}\,(u_t,\,B_t)\|_{L^2_T(L^2)}^2
 \lesssim \|(u_0,B_0)\|_{\dot H^{\frac{1}{2}}}^2.
\end{split}
\end{equation}

 Thanks to \eqref{est-basic-2-41} and \eqref{est-basic-2-45}, we obtain
\begin{equation}\label{est-basic-2-46}
\begin{split}
\|t^{\frac{1}{4}}\,(\nabla u, \nabla B)\|_{L^2_T(L^6)}+ \bigl\|t^{\frac{1}{4}}\,(\nabla\Pi-2\dv (\mu(\rho) d),&\curl (\sigma(\rho)\curl{B}))\bigr\|_{L^2_T(L^2)}\\
&\qquad\qquad\qquad\lesssim \|(u_0, B_0)\|_{\dot H^{\frac{1}{2}}},
\end{split}
\end{equation}
 from which,  \eqref{est-basic-2-37} and the interpolation inequality: $$\|\nabla f\|_{L^3(\mathbb{R}^3)}+\|f\|_{L^{\infty}(\mathbb{R}^3)} \lesssim \| \nabla f\|_{L^2(\mathbb{R}^3)}^{\frac{1}{2}}
\|\nabla f\|_{L^6(\mathbb{R}^3)}^{\frac{1}{2}},$$   we deduce that
  \begin{equation}\label{est-basic-2-48}
\begin{split}
&\|(\nabla u,\nabla B)\|_{L^2_T(L^3)}^2+\|(u,B)\|_{L^2_t(L^\infty)}^2 \\
&\lesssim \|t^{-\frac{1}{4}} (\nabla u,\nabla B)\|_{L^2_T(L^2)}^{\frac{1}{2}}
\|t^{\frac{1}{4}} (\nabla u,\nabla B)\|_{L^2_T(L^6)}^{\frac{1}{2}}
\lesssim
\|(u_0,B_0)\|_{\dot H^{\frac{1}{2}}}^2.
\end{split}
\end{equation}

By summarizing the estimates (\ref{est-basic-2-34}-\ref{est-basic-2-40}) and (\ref{est-basic-2-45}-\ref{est-basic-2-48}),
we conclude the proof of  \eqref{est-variable-2}.
This completes the proof of Proposition \ref{prop-vaviable-1}.
\end{proof}

\subsection{Constant viscosity case}

Let us first derive the equation satisfied by $D_tB.$ Indeed by applying $D_t=\partial_t+(u\cdot\nabla)$ to the  magnetic equation of \eqref{1.2}, we find
\begin{equation}\label{Dt-eqns-B-1}
D_t\bigl(D_t{B}+\curl(\sigma(\rho)\curl\,{B})\bigr)=D_t\bigl( B\cdot\nabla{u}\bigr).
\end{equation}
Notice that
\begin{equation*}\label{Dt-eqns-B-3}
\begin{split}
&[D_t; \nabla]f=\bigl(\partial_t\nabla{f}+u\cdot\nabla\,\nabla{f}\bigr)-\nabla \bigl(\partial_t{f}+u\cdot\nabla\,{f}\bigr)=-\nabla u^i\nabla_i\,{f},
\end{split}
\end{equation*}
from which, we infer
\begin{equation*}\label{Dt-eqns-B-4}
\begin{split}
D_t(B\cdot\nabla{u})&=D_t B\cdot\nabla{u} +B\cdot\,\nabla\,D_t{u}+B\cdot\,[D_t;\nabla]{u}\\
&=(D_t{B}\cdot\nabla) u
+(B\cdot\nabla) D_t{u}
-\bigl[\left((B\cdot\nabla)u\right)\cdot\nabla\bigr]u.
\end{split}
\end{equation*}
Next let's compute $D_t \curl(\sigma(\rho)\curl\,{B})$. We first observe that
\begin{equation*}\label{Dt-eqns-B-6}
\begin{split}
[D_t; \curl]\,f&= D_t (\curl \,f)-\curl \,D_tf=u\cdot \nabla (\curl \,f)-\curl \,(u\cdot \nabla\,f)\\
&=-\nabla\,u^{i} \wedge \partial_{i} f,
\end{split}
\end{equation*}
so that we get, by using the transport equation of \eqref{1.2}, that
\begin{align*}
D_t &\left(\curl(\sigma(\rho)\curl\,{B})\right)\\
=&\curl \left(\sigma(\rho)D_t \curl\,{B})\right)+[D_t; \curl]\left(\sigma(\rho)\curl\,{B}\right)\\
=&\curl \left(\sigma(\rho) \curl \,D_t\,{B})\right)+\curl \left(\sigma(\rho)[D_t; \curl]\,{B}\right)+[D_t; \curl]\left(\sigma(\rho)\curl\,{B}\right)\\
=&\curl \left(\sigma(\rho) \curl \,D_t\,{B}\right)-\curl \left(\sigma(\rho)\nabla\,u^{i} \wedge \partial_{i}{B}\right)-\nabla\,u^{i} \wedge \partial_{i} (\sigma(\rho)\curl\,{B}).
\end{align*}

By substituting the above equalities into \eqref{Dt-eqns-B-1}, we obtain
\begin{equation}\label{Dt-eqns-B-8}
\begin{split}
\partial_tD_t{B}+&u\cdot\nabla D_t{B}
+\curl\bigl(\sigma(\rho)\curl\, D_t{B}\bigr)
=g
\with\\
g=&(D_t{B}\cdot\nabla) u
+(B\cdot\nabla) D_t{u}
-\bigl[\left((B\cdot\nabla)u\right)\cdot\nabla\bigr]u\\
&+\curl \left(\sigma(\rho)\nabla\,u^i \wedge \partial_i \,{B}\right)+\nabla\,u^i \wedge \partial_i \left(\sigma(\rho)\curl\,{B}\right).
\end{split}
\end{equation}

\begin{prop}\label{prop-const-priori-1}
{\sl Under the assumptions of Proposition \ref{prop-vaviable-1}, if in addition the viscosity coefficient $\mu \equiv 1$, then \eqref{est-const-basic-0}  holds for $T<T^\ast.$
 }
\end{prop}

\begin{proof} Below we always assume that for $T<T^\ast.$ We first deduce from Proposition \ref{prop-vaviable-1} that
\begin{equation}\label{est-const-basic-1}
\begin{split}
&\|(u,B)\|_{\widetilde L^\infty_T(\dot H^{\frac{1}{2}})}
+\|{t}^{\frac{1}{4}}(\nabla u,\nabla B)\|_{L^\infty_T(L^2)}
+\|(\nabla u,\nabla B)\|_{L^4_T(L^2)}
\\&
+\bigl\|{t}^{\frac{1}{4}}\left(u_t,B_t, \nabla^2u,
\nabla\Pi, \curl(\sigma(\rho)\curl\, B)\right)\bigr\|_{L^2_T(L^2)}+\|{t}^{-\frac{1}{4}}(\nabla u,\nabla B)\|_{L^2_T(L^2)}
\\&
+\|{t}^{\frac{1}{4}}(\nabla u,\nabla B)\|_{L^2_T(L^6)}+\|(\nabla u,\nabla B)\|_{L^2_T(L^3)}+\|(u,B)\|_{L^2_T(L^\infty)}
\lesssim
\|(u_0,B_0)\|_{\dot H^{\frac{1}{2}}}.
\end{split}
\end{equation}

Let's turn to the  remaining  terms in \eqref{est-const-basic-0}.
By applying the operator $\partial_t$  to the momentum equation of \eqref{1.2}
and using the transport equation of \eqref{1.2}, we find
\begin{equation}\label{est-const-basic-3}
\begin{split}
\rho\partial_t u_t
&+\rho(u\cdot\nabla)u_t
-\Delta u_t
+\nabla\Pi_t=f,\with
\\
f&= (B_t\cdot\na)B+(B\cdot\na)B_t-\rho_tD_tu-\rho(u_t\cdot\nabla)u
\\&=(B_t\cdot\na)B+(B\cdot\na)B_t+(u\cdot\na)\rho\,D_tu-\rho(u_t\cdot\nabla)u.
\end{split}
\end{equation}

By taking $L^2$ inner product of \eqref{est-const-basic-3} with $u_t$, we obtain
\begin{equation}\label{est-const-basic-4}
\begin{split}
\frac{1}{2}\frac{d}{dt}\|\sqrt{\rho}u_t(t)\|_{L^2}^2
+\|\nabla u_t\|_{L^2}^2&=I_1+I_2\with\\
&I_1\eqdefa \int_{\mathbb{R}^3} \bigl((B_t\cdot\na)B-\rho(u_t\cdot\nabla)u\bigr)\cdot u_t\,dx,\\
&I_2 \eqdefa \int_{\mathbb{R}^3} \bigl((B\cdot\na)B_t+(u\cdot\na)\rho\,D_tu\bigr) \cdot u_t\,dx.
\end{split}
\end{equation}
Notice that
\begin{equation*}\label{est-const-basic-6}
\begin{split}
&|I_1| \lesssim \|(u_t, B_t)\|_{L^2}\|(\nabla u, \nabla B)\|_{L^3}\|u_t\|_{L^6}\lesssim \|(u_t, B_t)\|_{L^2}\|(\nabla u, \nabla B)\|_{L^3}\|\nabla u_t\|_{L^2},\\
& \bigl|\int_{\mathbb{R}^3} ((B\cdot\na)B_t) \cdot u_t\,dx\bigr|=\bigl|\int_{\mathbb{R}^3} (B\otimes B_t) :\nabla u_t\,dx\bigr|\lesssim \|B\|_{L^{\infty}}\|B_t\|_{L^2} \|\nabla u_t\|_{L^2},
\end{split}
\end{equation*}
and
\begin{align*}
\bigl|\int_{\R^3}(u\cdot\na)\rho\,D_tu\cdot u_t\,dx\bigr|
=&\bigl|\int_{\R^3}\rho[(u\cdot\na)(u_t+u\cdot\nabla u)]\cdot u_t\,dx
+\int_{\R^3}\rho[(u\cdot\na)u_t]\cdot D_tu\,dx\bigr|
\\
\lesssim &
\|u\|_{L^\infty}\|\nabla u_t\|_{L^2}\|u_t\|_{L^2}
+\|u\|_{L^3}\|\nabla u\|_{L^3}\|\nabla u\|_{L^6}\|u_t\|_{L^6}
\\&
+\|u\|_{L^\infty}\| u\|_{L^3}\|\nabla^2 u\|_{L^2}\|u_t\|_{L^6}
+\|u\|_{L^\infty}\|D_tu\|_{L^2}\|\nabla u_t\|_{L^2}
\\
\lesssim &\bigl(
\|u\|_{L^\infty}\bigl(\|u_t\|_{L^2}+\|D_tu\|_{L^2}\bigr)\\
&+\bigl(\|u\|_{L^\infty}+\|\na u\|_{L^3}\bigr)\|u\|_{L^3}\|\nabla^2 u\|_{L^2}
\bigr)\|\nabla u_t\|_{L^2}.
\end{align*}
Yet it follows from
the $u$ equation in \eqref{1.2} and
\eqref{est-basic-2-41} that
$$
\begin{aligned}
\|\nabla^2u\|_{L^2}
&\lesssim
\|\sqrt{\rho}u_t\|_{L^2}+\|(u\cdot\nabla)u\|_{L^2}
+\|(B\cdot\nabla)B\|_{L^2}
\\&
\lesssim
\|\sqrt{\rho}u_t\|_{L^2}+\|u\|_{L^3}\|\nabla u\|_{L^6}
+\|B\|_{L^3}\|\nabla B\|_{L^6}
\lesssim
\|(u_t, B_t)\|_{L^2},
\end{aligned}
$$
so that one has
\beno
|I_2|\lesssim \bigl(\|\nabla u\|_{L^3}+\|(u,B)\|_{L^\infty}\bigr)\|(u_t,B_t,D_tu)\|_{L^2} \|\nabla u_t\|_{L^2}.
\eeno

As a consequence,
we get, by substituting the above inequalities into \eqref{est-const-basic-4}, that
\begin{equation*}\label{est-const-basic-7}
\begin{split}
 \frac{d}{dt}\|\sqrt{\rho}u_t(t)\|_{L^2}^2
+2\|\nabla u_t\|_{L^2}^2 \lesssim \bigl(\|(u, B)\|_{L^{\infty}}+\|(\nabla u, \nabla B)\|_{L^3}\bigr)\|(u_t, D_tu, B_t)\|_{L^2} \|\nabla u_t\|_{L^2}.
\end{split}
\end{equation*}
Applying Young's inequality gives
\begin{equation}\label{est-const-basic-8}
\begin{split}
 \frac{d}{dt}\|\sqrt{\rho}u_t(t)\|_{L^2}^2
+&\|\nabla u_t\|_{L^2}^2 \\
\lesssim &\bigl(\|(u, B)\|_{L^{\infty}}^2+\|(\nabla u, \nabla B)\|_{L^3}^2\bigr)\bigl(\|(u_t, B_t)\|_{L^2}^2+
\|u\cdot \nabla u\|_{L^2}^2\bigr).
\end{split}
\end{equation}

On the other hand, taking $L^2$ inner product of  \eqref{Dt-eqns-B-8} with $D_tB,$ we find
\begin{equation*}\label{est-const-basic-9}
\begin{split}
\frac{1}{2}&\frac{d}{dt}\|D_t{B}\|_{L^2}^2
+\|\sqrt{\sigma(\rho)}\curl{D_t{B}}\|_{L^2}^2=II_1+II_2+II_3 \with\\
&II_1\eqdefa \int_{\mathbb{R}^3}\Bigl((D_t{B}\cdot\nabla) u
+(B\cdot\nabla) D_t{u}
-\bigl[\left((B\cdot\nabla)u\right)\cdot\nabla\bigr]u\Bigr)\cdot D_tB\,dx,\\
&II_2 \eqdefa \int_{\mathbb{R}^3} \curl (\sigma(\rho)\nabla\,u^{i} \wedge \partial_{i} \,{B})\cdot D_tB\,dx,\\
&II_3 \eqdefa \int_{\mathbb{R}^3}(\nabla\,u^{i} \wedge \partial_{i} (\sigma(\rho)\curl\,{B}))\cdot D_tB\,dx.
\end{split}
\end{equation*}
Notice that
\begin{equation*}\label{est-const-basic-11}
\begin{split}
|II_1| &\lesssim \bigl(\|D_t{B}\|_{L^2}\|\nabla  u\|_{L^3}+
\|B\|_{L^3}\|\nabla u\|_{L^4}^2\bigr)\|D_tB\|_{L^6}
+\|B\|_{L^{\infty}}\|\nabla D_t{u}\|_{L^2}\|D_tB\|_{L^2},\\
|II_2|&= \bigl|\int_{\mathbb{R}^3} \bigl[\sigma(\rho)\nabla\,u^{i} \wedge \partial_{i} \,{B}] : \curl D_tB\,dx\bigr|\lesssim \|(\nabla\,B, \nabla u)\|_{L^4}^2 \|\nabla D_tB\|_{L^2},\\
|II_3|&= |\int_{\mathbb{R}^3}(\nabla\,u^{i} \wedge  (\sigma(\rho)\curl\,{B}))\cdot  \partial_{i}D_tB\,dx|\\
& \leq  \|\nabla\,u\|_{L^3} \|\sigma(\rho)\curl\,{B}\|_{L^6}\|\nabla{D_tB}\|_{L^2}\lesssim \|\nabla\,u\|_{L^3}\|(u_t, B_t)\|_{L^2}\|\nabla{D_tB}\|_{L^2},
\end{split}
\end{equation*}
where we used \eqref{est-basic-2-41} in the last inequality, we thus obtain
\begin{equation}\label{est-const-basic-12}
\begin{split}
&\frac{1}{2}\frac{d}{dt}\|D_t{B}(t)\|_{L^2}^2
+\|\sqrt{\sigma(\rho)}\curl{D_t{B}}\|_{L^2}^2\lesssim \|B\|_{L^{\infty}}\|\nabla D_t{u}\|_{L^2}\|D_tB\|_{L^2}\\
&\qquad +\bigl(\|(D_t{B}, u_t, B_t)\|_{L^2}\|\nabla  u\|_{L^3}+
\|B\|_{L^3}\|\nabla u\|_{L^4}^2+\|(\nabla\,B, \nabla u)\|_{L^4}^2\bigr)\|\nabla D_tB\|_{L^2}.
\end{split}
\end{equation}

By summing up  \eqref{est-const-basic-8} and \eqref{est-const-basic-12}, we arrive at
\begin{equation}\label{est-const-basic-13}
\begin{split}
\frac{d}{dt}&\|(\sqrt{\rho}u_t, D_t{B})\|_{L^2}^2+c_3\|(\nabla u_t, \curl{D_t{B}})\|_{L^2}^2
\lesssim \|B\|_{L^{\infty}}\|\nabla D_t{u}\|_{L^2}\|D_tB\|_{L^2}\\
&+\bigl(\|(u, B)\|_{L^{\infty}}^2+\|(\nabla u, \nabla B)\|_{L^3}^2\bigr)\bigl(\|(u_t, B_t)\|_{L^2}^2+
 \|u\|_{L^3}^2\|\nabla^2 u\|_{L^2}^2\bigr)\\
& +\bigl(\|(D_t{B}, u_t, B_t)\|_{L^2}\|\nabla  u\|_{L^3}+
\|B\|_{L^3}\|\nabla u\|_{L^4}^2+\|(\nabla\,B, \nabla u)\|_{L^4}^2\bigr)\|\nabla D_tB\|_{L^2}.
\end{split}
\end{equation}
Due to $\dv\, u=\dv\, B=0$, we have
\begin{equation*}\label{est-const-basic-14}
\begin{split}
\|\nabla{D_t{B}}\|_{L^2}& \lesssim \|\curl{D_t{B}}\|_{L^2}+\|\dv {D_t{B}}\|_{L^2}\\
&\lesssim \|\curl{D_t{B}}\|_{L^2}+\|\nabla u\otimes \nabla B\|_{L^2} \lesssim \|\curl{D_t{B}}\|_{L^2}+\|(\nabla u, \nabla B)\|_{L^4}^2, \\
 \|\nabla{D_t{u}}\|_{L^2} &\lesssim \|\nabla u_t\|_{L^2}+\|u\|_{L^{\infty}}\|\nabla^2u\|_{L^2}+\|\nabla u\|_{L^4}^2,
\end{split}
\end{equation*}
which implies
\begin{equation*}\label{est-const-basic-15}
\begin{split}
&\|(\nabla{D_t{u}}, \nabla{D_t{B}})\|_{L^2} \lesssim \|(\nabla u_t, \curl{D_t{B}})\|_{L^2}+\|(\nabla u, \nabla B)\|_{L^4}^2+\|u\|_{L^{\infty}}\|\nabla^2u\|_{L^2},
\end{split}
\end{equation*}
from which and \eqref{est-const-basic-13}, we infer
\begin{equation*}\label{est-const-basic-16}
\begin{split}
& \frac{d}{dt}\|(\sqrt{\rho}u_t, D_t{B})(t)\|_{L^2}^2+2c_4\|(\nabla u_t, \nabla{D_t{u}}, \nabla{D_t{B}})\|_{L^2}^2 \\
 &\lesssim \|(\nabla u, \nabla B)\|_{L^4}^4+\|u\|_{L^{\infty}}^2\|\nabla^2u\|_{L^2}^2+\|B\|_{L^{\infty}}\|D_tB\|_{L^2}\|\nabla D_t{u}\|_{L^2}\\
&\quad+\bigl(\|(u, B)\|_{L^{\infty}}^2+\|(\nabla u, \nabla B)\|_{L^3}^2\bigr)\bigl(\|(u_t, B_t)\|_{L^2}^2+
 \|u\|_{L^3}^2\|\nabla^2 u\|_{L^2}^2\bigr)\\
&\quad +\bigl(\|(D_t{B}, u_t, B_t)\|_{L^2}\|\nabla  u\|_{L^3}+
\|B\|_{L^3}\|\nabla u\|_{L^4}^2+\|(\nabla\,B, \nabla u)\|_{L^4}^2\bigr)\|\nabla D_tB\|_{L^2}.
\end{split}
\end{equation*}
Applying Young's inequality yields
\begin{equation}\label{est-const-basic-17}
\begin{split}
& \frac{d}{dt}\|(\sqrt{\rho}u_t, D_t{B})(t)\|_{L^2}^2+c_4\|(\nabla u_t, \nabla{D_t{u}}, \nabla{D_t{B}})\|_{L^2}^2 \\
 &\lesssim (1+\|B\|_{L^3}^2)\|(\nabla u, \nabla B)\|_{L^4}^4 +\bigl(\|B\|_{L^{\infty}}^2+\|\nabla  u\|_{L^3}^2 \bigr)\|D_tB\|_{L^2}^2\\
 &\quad+\bigl(1+\|u\|_{L^3}^2\bigr)\bigl(\|(u, B)\|_{L^{\infty}}^2+\|(\nabla u, \nabla B)\|_{L^3}^2\bigr)\|(u_t, B_t, \nabla^2u)\|_{L^2}^2.
\end{split}
\end{equation}
Yet it follows from a similar derivation of \eqref{est-basic-2-23} that
\beno
\|(\nabla u, \nabla B)\|_{L^6}
\lesssim
\|(u_t, B_t)\|_{L^2},\, \|(\nabla u, \nabla B)\|_{L^4}
\lesssim
\|(\nabla u, \nabla B)\|_{L^2}^{\frac{1}{4}}
\|(u_t, B_t)\|_{L^2}^{\frac{3}{4}},\eeno
from which, \eqref{small-data-1} and Proposition \ref{prop-vaviable-1}, we infer
\begin{equation}\label{est-const-basic-2}
\begin{split}
\|(\nabla^2u(t), \nabla\Pi(t)\|_{L^2}
&\lesssim
\|\sqrt{\rho}u_t\|_{L^2}+\|(u\cdot\nabla)u\|_{L^2}
+\|(B\cdot\nabla)B\|_{L^2}
\\&
\lesssim
\|\sqrt{\rho}u_t\|_{L^2}+\|(u, B)\|_{L^3}\|(\nabla u, \nabla B)\|_{L^6}
\lesssim
\|(\sqrt{\rho}u_t,B_t)\|_{L^2},
\end{split}
\end{equation}
and
\begin{equation*}\label{est-const-basic-20}
\begin{split}
\|(u_t, B_t)\|_{L^2}
&\lesssim
\|(u_t,D_tB)\|_{L^2}+\|(u\cdot\nabla)B\|_{L^2}
\lesssim
\|(u_t,D_tB)\|_{L^2}+\|u\|_{L^3}\|\nabla B\|_{L^6}\\
&
\lesssim \|(u_t,D_tB)\|_{L^2}+\|u\|_{L^3}\|(u_t, B_t)\|_{L^2},\\
\|(u_t,D_tB)\|_{L^2}
&\lesssim
\|(u_t, B_t)\|_{L^2}+\|(u\cdot\nabla)B\|_{L^2}
\lesssim
\|(u_t, B_t)\|_{L^2}+\|u\|_{L^3}\|\nabla B\|_{L^6}.
\end{split}
\end{equation*}
As a result, it comes out
\begin{equation}\label{est-const-basic-21}
\begin{split}
\|(u_t, B_t)\|_{L^2} \lesssim \|(u_t,D_tB)\|_{L^2}, \quad \|(u_t,D_tB)\|_{L^2} \lesssim \|(u_t, B_t)\|_{L^2}.
\end{split}
\end{equation}

By inserting \eqref{est-const-basic-21} into
 \eqref{est-const-basic-17} and using \eqref{est-variable-2}, we obtain
\begin{equation}\label{est-const-basic-22}
\begin{split}
& \frac{d}{dt}\|(\sqrt{\rho}u_t, D_t{B})(t)\|_{L^2}^2+c_4\|(\nabla u_t, \nabla{D_t{u}}, \nabla{D_t{B}})\|_{L^2}^2 \\
 &\lesssim \bigl(\|(\nabla u, \nabla B)\|_{L^2}
\|(u_t, B_t)\|_{L^2}+\|(u, B)\|_{L^{\infty}}^2+\|(\nabla u, \nabla B)\|_{L^3}^2\bigr) \|(\sqrt{\rho}u_t,D_tB)\|_{L^2}^2.
\end{split}
\end{equation}
By multiplying  \eqref{est-const-basic-22} by $t^{\frac{3}{2}}$, we achieve
\begin{equation*}\label{est-const-basic-23}
\begin{split}
& \frac{d}{dt}\bigl\|t^{\frac{3}{4}}\,(\sqrt{\rho}u_t, D_t{B})(t)\bigr\|_{L^2}^2+c_4\|t^{\frac{3}{4}}\,(\nabla u_t, \nabla{D_t{u}}, \nabla{D_t{B}})\|_{L^2}^2 \\
 &\leq C_2 \|t^{\frac{1}{4}}(u_t, B_t)\|_{L^2}^2+C_2\bigl(\|(\nabla u, \nabla B)\|_{L^2}
\|(u_t, B_t)\|_{L^2}\\
&\qquad\qquad\qquad\qquad\qquad\qquad+\|(u, B)\|_{L^{\infty}}^2+\|(\nabla u, \nabla B)\|_{L^3}^2\bigr) \|t^{\frac{3}{4}}\,(\sqrt{\rho}u_t,D_tB)\|_{L^2}^2.
\end{split}
\end{equation*}
Applying Gronwall's inequality and using \eqref{est-const-basic-1} and \eqref{est-const-basic-21} gives rise to
\begin{equation}\label{est-const-basic-24}
\begin{split}
& \bigl\|t^{\frac{3}{4}}\,(\sqrt{\rho}u_t, B_t, D_t{B})\bigr\|_{L^{\infty}_T(L^2)}^2+c_4\bigl\|t^{\frac{3}{4}}\,(\nabla u_t, \nabla{D_t{u}}, \nabla{D_t{B}})\bigr\|_{L^2_T(L^2)}^2 \\
 &\leq C_2 \|t^{\frac{1}{4}}(u_t, B_t)\|_{L^2_T(L^2)}^2\,\exp\Bigl(C_2\bigl(\|t^{-\frac{1}{4}}(\nabla u, \nabla B)\|_{L^2_T(L^2)}\|t^{\frac{1}{4}}(u_t, \partial_t{B})\|_{L^2_T(L^2)}\\
&\qquad\qquad\qquad\qquad\qquad\qquad\qquad+\|(u, B)\|_{L^2_T(L^{\infty})}^2+\|(\nabla u, \nabla B)\|_{L^2_T(L^3)}^2\bigr)\Bigr)\\
 &\lesssim \|(u_0, B_0)\|_{\dot{H}^{\frac{1}{2}}}^2 \exp\Bigl(C \|(u_0, B_0)\|_{\dot{H}^{\frac{1}{2}}}^2\Bigr),
\end{split}
\end{equation}
and
\begin{equation*}\label{est-const-basic-24aab}
\begin{split}
\|{t}^{\frac{3}{4}}(\nabla u,\nabla B)\|_{L^{\infty}_T(L^6)}\lesssim \|t^{\frac{3}{4}}(u_t,B_t)\|_{L^\infty_T(L^2)}\lesssim \|(u_0, B_0)\|_{\dot{H}^{\frac{1}{2}}}.
\end{split}
\end{equation*}

While it follows from  \eqref{est-const-basic-21} that
\begin{equation}\label{est-const-basic-25}
\begin{split}
\|(u, B)\|_{L^{\infty}}^2+\|(\nabla u, \nabla B)\|_{L^3}^2  &\lesssim \|(\nabla u, \nabla B)\|_{L^2} \|(\nabla u, \nabla B)\|_{L^6}\\
& \lesssim \|(\nabla u, \nabla B)\|_{L^2}
\|(u_t, B_t)\|_{L^2},
\end{split}
\end{equation}
which together with \eqref{est-variable-2} and \eqref{est-const-basic-24} ensures that
\begin{equation*}\label{est-const-basic-26}
\begin{split}
&\|t^{\frac{1}{2}}(u,B)\|_{L^\infty_T(L^\infty)}+\|t^{\frac{1}{2}}(\nabla u, \nabla B)\|_{L^\infty_T(L^3)}\\
&\lesssim
\|t^{\frac{1}{4}}(\nabla u,\nabla B)\|_{L^\infty_t(L^2)}^{\frac{1}{2}}
\|t^{\frac{3}{4}}(u_t, B_t)\|_{L^\infty_t(L^2)}^{\frac{1}{2}}
\lesssim \|(u_0, B_0)\|_{\dot{H}^{\frac{1}{2}}} \exp\Bigl(C \|(u_0, B_0)\|_{\dot{H}^{\frac{1}{2}}}^2\Bigr),
\end{split}
\end{equation*}
and
\begin{equation*}\label{est-const-basic-27}
\begin{split}
&\|t^{\frac{1}{2}}\,u_t\|_{L^2_T(L^3)}^2\lesssim \|t^{\frac{1}{4}}\,u_t\|_{L^2_T(L^2)}\|t^{\frac{3}{4}}\,\nabla u_t\|_{L^2_T(L^2)} \lesssim \|(u_0, B_0)\|_{\dot{H}^{\frac{1}{2}}}^2.
\end{split}
\end{equation*}

Whereas we observe from the momentum equations of \eqref{1.2}, \eqref{est-const-basic-2} and \eqref{est-const-basic-25} that
\begin{equation*}\label{est-const-basic-28}
\begin{split}
\|(\nabla^2 u, \nabla\Pi)\|_{L^6}
&\lesssim
\|u_t\|_{L^6}+\|(u,B)\|_{L^\infty}\|(\nabla u,\nabla B)\|_{L^6}\\
&\lesssim
\|\nabla u_t\|_{L^2}
+\|(\nabla u, \nabla B)\|_{L^2}^{\frac{1}{2}}\|(u_t, {B}_t)\|_{L^2}^{\frac{3}{2}},
\end{split}
\end{equation*}
so that we have
\begin{equation*}\label{est-const-basic-29}
\begin{split}
&\|t^{\frac{3}{4}}\,(\nabla^2 u, \nabla\Pi)\|_{L^2_T(L^6)}^2\lesssim  \|t^{\frac{3}{4}}\,\nabla u_t\|_{L^2_T(L^2)}^2\\
&+\|t^{\frac{1}{4}}\,(\nabla u, \nabla B)\|_{L^{\infty}_T(L^2)}\|t^{\frac{1}{4}}\,({u}_t,{B}_t)\|_{L^2_T(L^2)}^2\|t^{\frac{3}{4}}\,({u}_t,{B}_t)\|_{L^{\infty}_T(L^2)}\lesssim \|(u_0, B_0)\|_{\dot{H}^{\frac{1}{2}}}^2,
\end{split}
\end{equation*}
and
\begin{equation}\label{est-const-basic-29}
\begin{split}
 \int_0^T \|t^{\frac{1}{2}}\,\nabla u(t)\|_{L^\infty}^2\,dt
&\lesssim \int_0^T\|t^{\frac{1}{4}}\,\nabla^2 u(t)\|_{L^2}\|t^{\frac{3}{4}}\,\nabla^2 u(t)\|_{L^6}\,dt\\
&
\lesssim\|t^{\frac{1}{4}}\,\nabla^2 u\|_{L^2_T(L^2)}\|t^{\frac{3}{4}}\,\nabla^2 u\|_{L^2_T(L^6)}
\lesssim
\|(u_0,B_0)\|_{\dot H^{\frac{1}{2}}}^2.
\end{split}
\end{equation}

By summarizing the estimates \eqref{est-const-basic-1} and (\ref{est-const-basic-24}-\ref{est-const-basic-29}),
We finish the proof of Proposition \ref{prop-const-priori-1}.
\end{proof}

\begin{lem}\label{Prop-const-j-Lip-1}{\sl  Let $(\rho, u, B, \nabla\Pi)$ be a smooth enough solution of \eqref{1.2} on $[0, T^{\ast}[$, and for $j\in\Z,$ let  $(u_j, \nabla\Pi_j, B_j)$ be determined by the system \eqref{model-3d-freq-1} with $\mu(\rho)=1.$
Then   under the assumptions of Proposition \ref{prop-const-priori-1}, for any  $T\in [0, T^{\ast}[$ and $ j \in \mathbb{Z}$, there holds
\begin{equation}\label{est-const-basic-j-1}
\begin{split}
&\|(\nabla u_j, \nabla B_j)\|_{L^\infty_T(L^2)} +\|\mathcal{Q}_j\|_{L^2_T(L^2)} +\|\nabla B_j\|_{L^2_T(L^6)}+\|{t}^{\frac{1}{2}}\,\mathcal{Q}_j\|_{L^\infty_T(L^2)}\\
&+\|{t}^{\frac{1}{2}}\,\nabla B_j\|_{L^\infty_T(L^6)}+\|{t}^{\frac{1}{2}}\,(\nabla D_t{u}_j,\nabla \partial_t{u}_j, \nabla{D_t{B}}_j)\|_{L^2_T(L^2)}\lesssim 2^j\|(\dot\Delta_ju_0,\dot\Delta_jB_0)\|_{L^2},
\end{split}
\end{equation}
and
\begin{equation}\label{est-const-basic-j-1a}
\begin{split}
\|{t}^{\frac{1}{2}}&\,(\nabla u_j, \nabla B_j)\|_{L^\infty_T(L^2)}+\|{t}^{\frac{1}{2}}\,\mathcal{Q}_j\|_{L^2_T(L^2)}+\|{t}^{\frac{1}{2}}\,\nabla B_j\|_{L^2_T(L^6)}\\
&+\|t\,\nabla B_j\|_{L^\infty_T(L^6)}
+\|t\,(\nabla D_t{u}_j,\nabla \partial_t{u}_j, \nabla{D_t{B}}_j)\|_{L^2_T(L^2)}\\
&+\|t\,\mathcal{Q}_j\|_{L^\infty_T(L^2)}+\|t\,\mathcal{Q}_j\|_{L^2_T(L^6)}
\lesssim
\|(\dot\Delta_ju_0,\dot\Delta_jB_0)\|_{L^2},
\end{split}
\end{equation} where
\begin{equation*}\label{notation-Q-1}
\mathcal{Q}_j\eqdefa \left(\p_tu_j,\partial_tB_j, D_t{u}_j, D_t{B}_j, \nabla^2u_j,\nabla\Pi_j,\curl(\sigma(\rho)\curl B_j)\right).
\end{equation*}
}
\end{lem}

\begin{proof} Due to $\mu(\rho)=1,$
 we deduce from the classical estimate on Stokes system and \eqref{model-3d-freq-1} that for $p \in [2, 6]$,
\begin{equation*}
\begin{split}
\|\mathcal{Q}_j\|_{L^p}
&\lesssim
\|(\partial_tu_j,D_tB_j)\|_{L^p}+\|(u\cdot\nabla u_j, B\cdot\nabla B_j, u\cdot\nabla B_j, B\cdot\nabla u_j)\|_{L^p},
\end{split}
\end{equation*}
which in particular implies
\begin{align*}
\|\mathcal{Q}_j\|_{L^2}
&\lesssim
\|(\partial_tu_j, \partial_tB_j)\|_{L^2}+\|(u,B)\|_{L^3}\|(\nabla u_j,\nabla B_j)\|_{L^6},\\
\|\mathcal{Q}_j\|_{L^6}
&\lesssim
\|(\partial_tu_j,D_tB_j)\|_{L^6}
+\|(u,B)\|_{L^{\infty}}\|(\nabla u_j,\nabla B_j)\|_{L^6}
\end{align*}
Hence it follows from \eqref{est-const-basic-0} and \eqref{est-basic-2-23} that
\begin{equation}\label{est-const-basic-j-2}
\begin{split}
\|\mathcal{Q}_j\|_{L^2}+\|\nabla{B}_j\|_{L^6} \lesssim
\|(\partial_tu_j,\partial_tB_j)\|_{L^2}
\end{split}
\end{equation}
and
\begin{equation}\label{est-const-basic-jL6-2}
\begin{split}
\|\mathcal{Q}_j\|_{L^6}
&\lesssim
\|(\partial_tu_j, D_tB_j)\|_{L^6}
+\|(u,B)\|_{L^{\infty}}\|(\partial_tu_j,\partial_tB_j)\|_{L^2},
\end{split}
\end{equation}

On the other hand, similar to the derivation of the equations \eqref{Dt-eqns-B-8} and \eqref{est-const-basic-3}, we get,
by applying  $D_t=\partial_t+(u\cdot\nabla)$ to the momentum equation of\eqref{model-3d-freq-1} (resp. magnetic equation), that
\begin{equation}\label{eqns-Dtub-j-1}
\begin{cases}
&\rho D_t(D_t{u}_j) -\Delta D_t{u}_j+D_t \nabla\Pi_j=f_j,\\
&D_t(D_t{B}_j)
+\curl\bigl(\sigma(\rho)\curl\, \dot{B}_j\bigr)
=g_j
\end{cases}
\end{equation}
with
\begin{equation*}\label{est-const-basic-j-3}
\begin{split}
f_j=&(B\cdot\nabla)D_t{B}_j+(D_t{B}\cdot\nabla)B_j-(B\cdot\nabla u \cdot\nabla) B_j
 -\partial_k(\partial_ku\cdot\nabla u_j)-(\partial_ku\cdot\nabla)\partial_ku_j,\\
g_j=&(B\cdot\nabla)D_t{u}_j+(D_t{B}\cdot\nabla) u_j
- [(B\cdot\nabla)u]\cdot\nabla u_j\\
&
+\curl (\sigma(\rho)\nabla\,u^{i} \wedge \partial_{i} \,{B}_j)+\nabla\,u^{i} \wedge \partial_{i} (\sigma(\rho)\curl\,{B}_j).
\end{split}
\end{equation*}
By taking $L^2$ inner product of \eqref{eqns-Dtub-j-1} with $(D_t{u}_j,D_t{B}_j),$ we find
\begin{equation}\label{est-const-basic-j-5}
\begin{split}
&\frac{1}{2}\frac{d}{dt}\|\sqrt{\rho}\,D_t{u}_j(t)\|_{L^2}^2
+\|\nabla D_t{u}_j\|_{L^2}^2=\sum_{i=1}^6K_i,
\end{split}
\end{equation}
and
\begin{equation}\label{est-const-basic-j-6}
\begin{split}
\frac{1}{2}\frac{d}{dt}\| D_t{B}_j(t)\|_{L^2}^2
&+\|\sqrt{\sigma(\rho)}\curl{ D_t{B}}_j\|_{L^2}^2=\sum_{i=1}^5J_i,
\end{split}
\end{equation}
where
\begin{equation*}\label{est-const-basic-j-7}
\begin{split}
\sum_{i=1}^6K_i \eqdefa &
\int_{\R^3}(B\cdot\nabla) D_t{B}_j | D_t{u}_j\,dx
+\int_{\R^3} D_t{B}\cdot\nabla B_j | D_t{u}_j\,dx
\\&
-\int_{\R^3}[(B\cdot\nabla u)\cdot\nabla]B_j  | D_t{u}_j\,dx
-\int_{\R^3}\partial_k[(\partial_ku\cdot\nabla)u_j] |  D_t{u}_j\,dx
\\&
-\int_{\R^3}(\partial_ku\cdot\nabla)\partial_ku_j |  D_t{u}_j\,dx-\int_{\R^3}D_t\nabla\Pi_j\cdot D_t{u}_j\,dx
\end{split}
\end{equation*}
and
\begin{equation*}\label{est-const-basic-j-8}
\begin{split}
\sum_{i=1}^5J_i\eqdefa &
\int_{\R^3}(B\cdot\nabla)  D_t{u}_j | D_t{B}_j\,dx
+\int_{\R^3}( D_t{B}\cdot\nabla) u_j | D_t{B}_j\,dx
\\&
-\int_{\R^3}\bigl[\left((B\cdot\nabla)u\right)\cdot\nabla\bigr]u_j | D_t{B}_j\,dx
+\int_{\R^3}
\curl (\sigma(\rho)\nabla\,u^{\ell} \wedge \partial_{\ell} \,{B}_j) |
 D_t{B}_j\,dx\\
 &+\int_{\R^3}\nabla\,u^{i} \wedge \partial_{i} (\sigma(\rho)\curl\,{B}_j) |
 D_t{B}_j\,dx.
\end{split}
\end{equation*}
Due to $\dv\, B=0,$ one has
\begin{equation*}\label{est-const-basic-j-9}
K_1+J_1=0.
\end{equation*}
By H\"older's inequality, we obtain
\begin{equation*}\label{est-const-basic-j-10}
\begin{split}
|K_2|+|J_2|
&\leq
\|D_t{B}\|_{L^2} \|(\nabla B_j, \nabla u_j)\|_{L^3}
 \|(D_t{u}_j, D_t{B}_j)\|_{L^6}\\
 &
\lesssim
\|D_t{B}\|_{L^2}\|(\nabla B_j, \nabla u_j)\|_{L^3}
 \|(\nabla D_t{u}_j, \nabla D_t{B}_j)\|_{L^2}
\end{split}
\end{equation*}
and
\begin{equation*}\label{est-const-basic-j-11}
\begin{split}
|K_3|+|J_3|
&\le
\|B\|_{L^3}\|\nabla u\|_{L^6}\|(\nabla B_j, \nabla u_j)\|_{L^3}
 \|(D_t{u}_j, D_t{B}_j)\|_{L^6}
\\&
\lesssim
\|B\|_{L^3}\|\nabla u\|_{L^6}\|(\nabla B_j, \nabla u_j)\|_{L^3}
 \|(\nabla D_t{u}_j, \nabla D_t{B}_j)\|_{L^2}.
\end{split}
\end{equation*}
While due to the fact that $\dv\,u=0,$ we get, by using integration by parts, that
\begin{equation*}\label{est-const-basic-j-12}
|K_4|+|K_5|
\le
\|\nabla u\|_{L^6}\|\nabla u_j\|_{L^3}\|\nabla D_t{u}_j\|_{L^2}
\end{equation*}
and
\begin{equation*}\label{est-const-basic-j-13}
\begin{split}
K_6&=-\frac{d}{dt}\int_{\R^3}\nabla\Pi_j |(u\cdot\nabla)u_j(t)\,dx+K_6^{R}\with\\
K_6^{R}&\eqdefa \int_{\R^3}\nabla\Pi_j |(u_t\cdot\nabla)u_j\,dx+\int_{\R^3}\nabla\Pi_j| (u\cdot\nabla)\partial_tu_j\,dx-\int_{\R^3}(u\cdot\nabla) D_t{u}_j\nabla\Pi_j\,dx.
\end{split}
\end{equation*}
It is easy to observe that
\begin{equation*}\label{est-const-basic-j-15}
\begin{split}
|K_6^{R}|
\lesssim &\|\nabla\Pi_j\|_{L^2} \bigl(\|u_t\|_{L^3}\|\nabla u_j\|_{L^6}
+\|u\|_{L^\infty}\|\nabla\partial_tu_j\|_{L^2}
+\|u\|_{L^\infty}\|\nabla D_t{u}_j\|_{L^2}\bigr),
\end{split}
\end{equation*}
and
\begin{equation*}\label{est-const-basic-j-14}
\begin{split}
&|J_4|
\lesssim
\|\nabla u\|_{L^6}\|\curl{B_j}\|_{L^3}\|\curl D_t{B}_j\|_{L^2},\\
&|J_5|=|\int_{\R^3}\nabla\,u^{i} \wedge  (\sigma(\rho)\curl\,{B}_j) |
 \partial_{i}D_t{B}_j\,dx| \lesssim
\|\nabla u\|_{L^6}\|\curl{B_j}\|_{L^3}\|\nabla{D_t{B}_j}\|_{L^2}.
\end{split}
\end{equation*}

By substituting the above estimates into \eqref{est-const-basic-j-5} and \eqref{est-const-basic-j-6}, we obtain
\begin{equation}\label{est-const-basic-j-15}
\begin{split}
&\frac{d}{dt}\Bigl(\|\bigl(\sqrt{\rho} D_t{u}_j, D_t{B}_j\bigr)(t)\|_{L^2}^2+2\int_{\R^3}\nabla\Pi_j |(u\cdot\nabla)u_j(t)\,dx\Bigr)
+\|\bigl(\nabla D_t{u}_j,\curl{D_t{B}}_j\bigr)\|_{L^2}^2
\\&
\lesssim
\|D_t{B}\|_{L^2} \|(\nabla B_j, \nabla u_j)\|_{L^3}
 \|(\nabla D_t{u}_j, \nabla D_t{B}_j)\|_{L^2}
\\&\quad
+(1+\|B\|_{L^3})\|\nabla u\|_{L^6}\|(\nabla B_j, \nabla u_j)\|_{L^3}
 \|(\nabla D_t{u}_j, \nabla D_t{B}_j)\|_{L^2}
\\&\quad
+\|\nabla\Pi_j\|_{L^2}\bigl(\|u_t\|_{L^3}\|\nabla^2 u_j\|_{L^2}
+\|u\|_{L^\infty}\|(\nabla\partial_tu_j, \nabla D_t{u}_j)\|_{L^2}\bigr).
\end{split}
\end{equation}
Due to $\dv\,B_j=0,$  one has
\begin{equation*}\label{est-const-basic-j-16}
\begin{split}
\|\nabla D_t{B}_j\|_{L^2}
&\lesssim \|\curl D_t{B}_j\|_{L^2}+ \|\dv D_t{B}_j\|_{L^2}
\lesssim
\|\curl D_t{B}_j\|_{L^2}+\|\nabla u\nabla B_j\|_{L^2}\\
&\lesssim
\|\curl D_t{B}_j\|_{L^2}+\|\nabla u\|_{L^3}\|\nabla B_j\|_{L^6},\\
 \|\nabla D_t{u}_j\|_{L^2}&\lesssim \|\nabla \partial_t{u}_j\|_{L^2}+\|\nabla u\nabla u_j\|_{L^2}+\|u\cdot\nabla \nabla u_j\|_{L^2}\\
 &\lesssim \|\nabla \partial_t{u}_j\|_{L^2}+(\|\nabla u\|_{L^3}+\|u\|_{L^{\infty}})\|\nabla^2u_j\|_{L^2},\\
\end{split}
\end{equation*}
from which and \eqref{est-const-basic-j-2}, we infer
 \begin{equation}\label{est-const-basic-j-17}
\begin{split}
\|(\nabla D_t{u}_j, \nabla D_t{B}_j)\|_{L^2}\lesssim &\|(\nabla \partial_t{u}_j, \curl D_t{B}_j)\|_{L^2}\\
&+\bigl(\|\nabla u\|_{L^3}+\|u\|_{L^{\infty}}\bigr)
\|(\partial_tu_j,\partial_tB_j)\|_{L^2}.
\end{split}
\end{equation}

Thanks to  \eqref{est-const-basic-j-2}, \eqref{est-const-basic-j-15} and \eqref{est-const-basic-j-17}, we obtain
\begin{equation*}\label{est-const-basic-j-18}
\begin{split}
\frac{d}{dt}\Bigl(&\|\bigl(\sqrt{\rho} D_t{u}_j, D_t{B}_j(t)\bigr)\|_{L^2}^2+2\int_{\R^3}\nabla\Pi_j |(u\cdot\nabla)u_j(t)\,dx\Bigr) +2c_3\|(\nabla D_t{u}_j,\nabla \partial_t{u}_j, \nabla{D_t{B}}_j)\|_{L^2}^2\\
\lesssim &
\|D_t{B}\|_{L^2} \|(\nabla B_j, \nabla u_j)\|_{L^3}
 \|(\nabla D_t{u}_j, \nabla D_t{B}_j)\|_{L^2}
\\&
+(1+\|B\|_{L^3})\|\nabla u\|_{L^6}\|(\nabla B_j, \nabla u_j)\|_{L^3}
 \|(\nabla D_t{u}_j, \nabla D_t{B}_j)\|_{L^2}
\\&
+\|(\partial_tu_j,\partial_tB_j)\|_{L^2}\|u\|_{L^\infty}\|(\nabla\partial_tu_j, \nabla D_t{u}_j)\|_{L^2}\\
&+\bigl(\|u_t\|_{L^3}+\|\nabla u\|_{L^3}^2+\|u\|_{L^{\infty}}^2)
\|(\partial_tu_j,\partial_tB_j\bigr)\|_{L^2}^2.
\end{split}
\end{equation*}
Applying Young's inequality gives rise to
\begin{equation}\label{est-const-basic-j-22a}
\begin{split}
\frac{d}{dt}\Bigl(&\|\bigl(\sqrt{\rho} D_t{u}_j, D_t{B}_j\bigr)(t)\|_{L^2}^2+2\int_{\R^3}\nabla\Pi_j |(u\cdot\nabla)u_j(t)\,dx\Bigr)\\
 &+c_3\|(\nabla D_t{u}_j,\nabla \partial_t{u}_j, \nabla{D_t{B}}_j)\|_{L^2}^2\\
\lesssim &
\bigl(\|D_t{B}\|_{L^2}^2
+(1+\|B\|_{L^3}^2)\|\nabla^2 u\|_{L^2}^2\bigr)\|(\nabla B_j, \nabla u_j)\|_{L^3}^2
\\&
+ \bigl(\|u_t\|_{L^3}+\|\nabla u\|_{L^3}^2+\|u\|_{L^{\infty}}^2\bigr)
\|(\partial_tu_j,\partial_tB_j)\|_{L^2}^2.
\end{split}
\end{equation}
Notice that
\begin{equation*}\label{est-const-basic-j-23}
\begin{split}
\|(\nabla B_j, \nabla u_j)\|_{L^3}^2&\lesssim  \|(\nabla B_j, \nabla u_j)\|_{L^2}\|(\nabla B_j, \nabla u_j)\|_{L^6}\\
&\lesssim  \|(\nabla B_j, \nabla u_j)\|_{L^2}\|(\partial_tu_j,\partial_tB_j)\|_{L^2},
\end{split}
\end{equation*}
and
\begin{equation*}\label{est-const-basic-j-25}
|\int_{\R^3}\nabla\Pi_j |(u\cdot\nabla)u_j\,dx| \lesssim \|\nabla\Pi_j\|_{L^2} \|u\|_{L^{3}}\|\nabla u_j\|_{L^6} \lesssim  \|u\|_{L^{3}}\|(\partial_tu_j,\partial_tB_j)\|_{L^2}^2,
\end{equation*}
we get, by multiplying $t$ to \eqref{est-const-basic-j-22a}, that
\begin{equation}\label{est-const-basic-j-26}
\begin{split}
\frac{d}{dt}\Bigl(&\|t^{\frac{1}{2}}\,(\sqrt{\rho} D_t{u}_j, D_t{B}_j)(t)\|_{L^2}^2+2 t\,\int_{\R^3}\nabla\Pi_j |(u\cdot\nabla)u_j(t)\,dx\Bigr)\\
& +c_3\|t^{\frac{1}{2}}\,(\nabla D_t{u}_j,\nabla \partial_t{u}_j, \nabla{D_t{B}}_j)\|_{L^2}^2\\
\lesssim &\|\bigl(\sqrt{\rho} D_t{u}_j, D_t{B}_j\bigr)\|_{L^2}^2+\|u\|_{L^{3}}\|(\partial_tu_j,\partial_tB_j)\|_{L^2}^2\\
&+
\bigl(\|t^{\frac{1}{4}}\,D_t{B}\|_{L^2}^2
+(1+\|B\|_{L^3}^2)\|t^{\frac{1}{4}}\,\nabla^2 u\|_{L^2}^2\bigr)\|(\nabla B_j, \nabla u_j)\|_{L^2}\|t^{\frac{1}{2}}\,(\partial_tu_j,\partial_tB_j)\|_{L^2}
\\&
+ \|t^{\frac{1}{2}}\,u_t\|_{L^3}
\|t^{\frac{1}{2}}\,(\partial_tu_j,\partial_tB_j)\|_{L^2} \|(\partial_tu_j,\partial_tB_j)\|_{L^2}\\
& + (\|t^{\frac{1}{2}}\,\nabla u\|_{L^3}^2+\|t^{\frac{1}{2}}\,u\|_{L^{\infty}}^2)
\|(\partial_tu_j,\partial_tB_j)\|_{L^2}^2.
\end{split}
\end{equation}
Observing that
 \begin{equation*}\label{est-const-basic-j-19}
\begin{split}
&\| D_t{u}_j\|_{L^2}-\|u\|_{L^3}\|\nabla u_j\|_{L^6}
\le
\|\partial_tu_j\|_{L^2}
\le
\| D_t{u}_j\|_{L^2}+\|u\|_{L^3}\|\nabla u_j\|_{L^6},\\
&\| D_t{B}_j\|_{L^2}-\|u\|_{L^3}\|\nabla B_j\|_{L^6}
\le
\|\partial_tB_j\|_{L^2}
\le
\| D_t{B}_j\|_{L^2}+\|u\|_{L^3}\|\nabla B_j\|_{L^6},
\end{split}
\end{equation*}
 we deduce from the fact: $\|u\|_{L^\infty_T(L^3)} \lesssim \|(u_0, B_0)\|_{\dot{H}^{\frac{1}{2}}},$ which is sufficiently small, and \eqref{est-const-basic-j-2} that
 \begin{equation}\label{est-const-basic-j-20}
\begin{split}
&\| (D_t{u}_j, D_t{B}_j)\|_{L^2}\lesssim
\|(\partial_tu_j, \partial_tB_j)\|_{L^2} \lesssim\| (D_t{u}_j, D_t{B}_j)\|_{L^2}.
\end{split}
\end{equation}
Let
$$E_j^{(1)}(T)\eqdefa \|t^{\frac{1}{2}}\,(\sqrt{\rho} D_t{u}_j, D_t{B}_j)\|_{L^2_T(L^2)}^2+2 \|t\,\int_{\R^3}\nabla\Pi_j |(u\cdot\nabla)u_j\,dx\|_{L^{\infty}_T},$$
then it follows  from  \eqref{est-const-basic-j-20} that
 \begin{equation}\label{est-const-basic-j-28}
\begin{split}
&E_j^{(1)}(T)\lesssim
\|t^{\frac{1}{2}}\,(\partial_tu_j, \partial_tB_j)\|_{L^{\infty}_T(L^2)}^2 \lesssim E_j^{(1)}(T).
\end{split}
\end{equation}
By integrating \eqref{est-const-basic-j-26} over $[0, T]$ for $T<T^\ast,$ we find
\begin{equation*}\label{est-const-basic-j-29}
\begin{split}
E_j^{(1)}(T)&+\|t^{\frac{1}{2}}\,(\nabla D_t{u}_j,\nabla \partial_t{u}_j, \nabla{D_t{B}}_j)\|_{L^2_T(L^2)}^2\\
\leq  C_4\Bigl(&\|(D_t{u}_j, D_t{B}_j)\|_{L^2_T(L^2)}^2+ \|u\|_{L^{\infty}_T(L^{3})}\|(\partial_tu_j,\partial_tB_j)\|_{L^2_T(L^2)}^2\\
&+ \bigl((\|t^{\frac{1}{4}}\,D_t{B}\|_{L^2_T(L^2)}^2
+(1+\|B\|_{L^{\infty}_T(L^3)}^2)\|t^{\frac{1}{4}}\,\nabla^2 u\|_{L^2_T(L^2)}^2\bigr)\\
&\qquad\times\|(\nabla B_j, \nabla u_j)\|_{L^{\infty}_T(L^2)}\|t^{\frac{1}{2}}\,(\partial_tu_j,\partial_tB_j)\|_{L^{\infty}_T(L^2)}
\\&
+ \|t^{\frac{1}{2}}\,u_t\|_{L^{2}_T(L^3)}
\|t^{\frac{1}{2}}\,(\partial_tu_j,\partial_tB_j)\|_{L^{\infty}_T(L^2)}\|(\partial_tu_j,\partial_tB_j)\|_{L^2_T(L^2)}\\
& + \bigl(\|t^{\frac{1}{2}}\,\nabla u\|_{L^{\infty}_T(L^3)}^2+\|t^{\frac{1}{2}}\,u\|_{L^{\infty}_T(L^{\infty})}^2\bigr)
\|(\partial_tu_j,\partial_tB_j)\|_{L^2_T(L^2)}^2\Bigr),
\end{split}
\end{equation*}
which together with\eqref{est-const-basic-0}, \eqref{S2eq1}  and  \eqref{est-variable-j-1}  ensures that
\begin{equation*}\label{est-const-basic-j-30}
\begin{split}
E_j^{(1)}&(T)+\|t^{\frac{1}{2}}\,(\nabla D_t{u}_j,\nabla \partial_t{u}_j, \nabla{D_t{B}}_j)\|_{L^2_T(L^2)}^2
\leq C_5\Bigl(\bigl(1+ \|(u_0, B_0)\|_{\dot{H}^{\frac{1}{2}}}\bigr)\\
&\quad\times\bigl(2^{2j}\|(\dot\Delta_ju_0,\dot\Delta_jB_0)\|_{L^2}^2+\|(u_0, B_0)\|_{\dot{H}^{\frac{1}{2}}}
2^{j}\|(\dot\Delta_ju_0,\dot\Delta_jB_0)\|_{L^2}(E_j^{(1)}(T))^{\frac{1}{2}}\bigr) \\
&\qquad\qquad\qquad\qquad\qquad\qquad\qquad\qquad\quad+ \|(u_0, B_0)\|_{\dot{H}^{\frac{1}{2}}}^2
2^{2j}\|(\dot\Delta_ju_0,\dot\Delta_jB_0)\|_{L^2}^2\Bigr),
\end{split}
\end{equation*}
which implies
\begin{equation*}\label{est-const-basic-j-31}
\begin{split}
&E_j^{(1)}(T)+\|t^{\frac{1}{2}}\,(\nabla D_t{u}_j,\nabla \partial_t{u}_j, \nabla{D_t{B}}_j)\|_{L^2_T(L^2)}^2 \lesssim  2^{2j}\|(\dot\Delta_ju_0,\dot\Delta_jB_0)\|_{L^2}^2.
\end{split}
\end{equation*}

Similarly, we get, by multiplying $t^2$ to \eqref{est-const-basic-j-22a}, that
\begin{equation}\label{est-const-basic-j-32}
\begin{split}
\frac{d}{dt}\Bigl(&\|t\,(\sqrt{\rho} D_t{u}_j, D_t{B}_j)(t)\|_{L^2}^2+2\int_{\R^3}t^2\nabla\Pi_j |(u\cdot\nabla)u_j(t)\,dx\Bigr)\\
& +c_3\|t\,(\nabla D_t{u}_j,\nabla \partial_t{u}_j, \nabla{D_t{B}}_j)\|_{L^2}^2\\
\lesssim &\|t^{\frac{1}{2}}\,(\sqrt{\rho} D_t{u}_j, D_t{B}_j)\|_{L^2}^2+\|u\|_{L^{3}}\|t^{\frac{1}{2}}\,(\partial_tu_j,\partial_tB_j)\|_{L^2}^2\\
&+
\bigl(\|t^{\frac{1}{4}}\,D_t{B}\|_{L^2}^2
+(1+\|B\|_{L^3}^2)\|t^{\frac{1}{4}}\,\nabla^2 u\|_{L^2}^2\bigr)\|t^{\frac{1}{2}}\,(\nabla B_j, \nabla u_j)\|_{L^2}\|t\,(\partial_tu_j,\partial_tB_j)\|_{L^2}
\\&
+ \|t^{\frac{1}{2}}\,u_t\|_{L^3}
\|t^{\frac{1}{2}}\,(\partial_tu_j,\partial_tB_j)\|_{L^2} \|t\,(\partial_tu_j,\partial_tB_j)\|_{L^2}\\
&+ \bigl(\|t^{\frac{1}{2}}\,\nabla u\|_{L^3}^2+\|t^{\frac{1}{2}}\,u\|_{L^{\infty}}^2\bigr)
\|t^{\frac{1}{2}}\,(\partial_tu_j,\partial_tB_j)\|_{L^2}^2.
\end{split}
\end{equation}
Let
$$E_j^{(2)}(T)\eqdefa \|t\,(\sqrt{\rho} D_t{u}_j, D_t{B}_j)\|_{L^2_T(L^2)}^2+2 \|t^2\,\int_{\R^3}\nabla\Pi_j |(u\cdot\nabla)u_j(t)\,dx\|_{L^{\infty}_T},$$
we get, by a similar derivation of  \eqref{est-const-basic-j-28}, that
 \begin{equation*}\label{est-const-basic-j-33}
\begin{split}
&E_j^{(2)}(T)\lesssim
\|t \,(\partial_tu_j, \partial_tB_j)\|_{L^{\infty}_T(L^2)}^2 \lesssim E_j^{(2)}(T).
\end{split}
\end{equation*}
Then by integrating \eqref{est-const-basic-j-32} over $[0, T]$ for $T<T^\ast,$ we find
\begin{equation*}\label{est-const-basic-j-34}
\begin{split}
E_j^{(2)}(T)&+\|t^{\frac{1}{2}}\,(\nabla D_t{u}_j,\nabla \partial_t{u}_j, \nabla{D_t{B}}_j)\|_{L^2_T(L^2)}^2\\
\leq C_4\Bigl(&\|t^{\frac{1}{2}}\,(D_t{u}_j, D_t{B}_j)\|_{L^2_T(L^2)}^2+ \|u\|_{L^{\infty}_T(L^{3})}\|t^{\frac{1}{2}}\,(\partial_tu_j,\partial_tB_j)\|_{L^2_T(L^2)}^2\\
&+ \bigl(\|t^{\frac{1}{4}}\,D_t{B}\|_{L^2_T(L^2)}^2
+(1+\|B\|_{L^{\infty}_T(L^3)}^2)\|t^{\frac{1}{4}}\,\nabla^2 u\|_{L^2_T(L^2)}^2\bigr)\\
&\qquad\times\|t^{\frac{1}{2}}\,(\nabla B_j, \nabla u_j)\|_{L^{\infty}_T(L^2)}\|t\,(\partial_tu_j,\partial_tB_j)\|_{L^{\infty}_T(L^2)}
\\&
+ \|t^{\frac{1}{2}}\,u_t\|_{L^{2}_T(L^3)}
\|t^{\frac{1}{2}}\,(\partial_tu_j,\partial_tB_j)\|_{L^{\infty}_T(L^2)}\|t\,(\partial_tu_j,\partial_tB_j)\|_{L^2_T(L^2)}\\
&+ \bigl(\|t^{\frac{1}{2}}\,\nabla u\|_{L^{\infty}_T(L^3)}^2+\|t^{\frac{1}{2}}\,u\|_{L^{\infty}_T(L^{\infty})}^2\bigr)
\|t^{\frac{1}{2}}\,(\partial_tu_j,\partial_tB_j)\|_{L^2_T(L^2)}^2\Bigr),
\end{split}
\end{equation*}
which together with  \eqref{est-const-basic-0} and \eqref{est-variable-j-1} ensures that
\begin{equation*}\label{est-const-basic-j-35}
\begin{split}
E_j^{(2)}(T)+&\|t^{\frac{1}{2}}\,(\nabla D_t{u}_j,\nabla \partial_t{u}_j, \nabla{D_t{B}}_j)\|_{L^2_T(L^2)}^2\\
\leq C_5 \Bigl( \bigl(&1+\|(u_0, B_0)\|_{\dot{H}^{\frac{1}{2}}}^2\bigr) \|(\dot\Delta_ju_0,\dot\Delta_jB_0)\|_{L^2}^2\\
&
+  \|(u_0, B_0)\|_{\dot{H}^{\frac{1}{2}}}(1+\|(u_0, B_0)\|_{\dot{H}^{\frac{1}{2}}})
 \|(\dot\Delta_ju_0,\dot\Delta_jB_0)\|_{L^2}(E_j^{(2)}(T))^{\frac{1}{2}}\Bigr),
\end{split}
\end{equation*}
from which, we infer
\begin{equation*}\label{est-const-basic-j-36}
\begin{split}
&E_j^{(2)}(T)+\|t\,(\nabla D_t{u}_j,\nabla \partial_t{u}_j, \nabla{D_t{B}}_j)\|_{L^2_T(L^2)}^2 \lesssim  \|(\dot\Delta_ju_0,\dot\Delta_jB_0)\|_{L^2}^2.
\end{split}
\end{equation*}
Then it follows from  \eqref{est-const-basic-jL6-2} that
\begin{equation*}\label{est-const-basic-j-37}
\begin{split}
\|t\,\mathcal{Q}_j\|_{L^2_T(L^6)}&\lesssim
\|t\,(\partial_tu_j, D_tB_j)\|_{L^2_T(L^6)} +\|(u,B)\|_{L^2_T(L^{\infty})}\|t\,(\partial_tu_j,\partial_tB_j)\|_{L^{\infty}_T(L^2)}\\
&\lesssim
\|(\dot\Delta_ju_0,\dot\Delta_jB_0)\|_{L^2}.
\end{split}
\end{equation*}
We thus complete the proof of Lemma \ref{Prop-const-j-Lip-1}.
\end{proof}

\begin{lem}\label{prop-Lip1}{\sl Under the assumptions of Lemma \ref{Prop-const-j-Lip-1}, for any $T<T^\ast,$ there holds
\begin{equation}\label{est-const-hL6-j-0}
\begin{split}
\|t\,\nabla D_t{u}_j\|_{L^{\infty}_T(L^2)}&+\|t\,(\partial_t u_j, D_t u_j, \nabla^2 u_j, \nabla\Pi_j)\|_{L^\infty_T(L^6)}\\
&+\|t\,(\partial_t D_t{u}_j, \nabla D_t\Pi_j, \nabla^2 D_t{u}_j)\|_{L^2_T(L^2)}\lesssim 2^{j}\|(\dot\Delta_ju_0,\dot\Delta_jB_0)\|_{L^2}.
\end{split}
\end{equation}
}
\end{lem}

\begin{proof}
Let us recall from \eqref{eqns-Dtub-j-1} that
\begin{equation}\label{est-const-hL6-j-1}
\rho \partial_t(D_t{u}_j) +\rho u \cdot \nabla (D_t{u}_j)-\Delta D_t{u}_j+D_t \nabla\Pi_j=f_j.
\end{equation}
By
taking $L^2$ inner product of  \eqref{est-const-hL6-j-1}  with $\partial_t D_t{u}_j,$ we find
\begin{equation}\label{est-const-hL6-j-2}
\begin{split}
\|\sqrt{\rho}\,\partial_t D_t{u}_j\|_{L^2}^2
+&\frac{1}{2}\frac{d}{dt}\|\nabla D_t{u}_j\|_{L^2}^2=\sum_{i=1}^7L_i\with\\
\sum_{i=1}^7L_i\eqdefa& -\int_{\R^3}\rho(u\cdot\nabla) D_t u_j |\partial_t D_t{u}_j\,dx
+\int_{\R^3}(B\cdot\nabla) D_t{B}_j |\partial_t D_t{u}_j\,dx
\\&
+\int_{\R^3}(D_t{B}\cdot\nabla)B_j | \partial_t D_t{u}_j\,dx
-\int_{\R^3}[(B\cdot\nabla u)\cdot\nabla]B_j | \partial_t D_t{u}_j\,dx
\\&
-\int_{\R^3}\partial_k[(\partial_ku\cdot\nabla)u_j] | \partial_t D_t{u}_j\,dx
-\int_{\R^3}(\partial_ku\cdot\nabla)\partial_ku_j | \partial_t D_t{u}_j\,dx
\\&
-\int_{\R^3}D_t\nabla\Pi_j|
\partial_t D_t{u}_j\,dx.
\end{split}
\end{equation}
It is easy to observe that
\begin{equation*}\label{est-const-hL6-j-4}
\begin{split}
|L_1|+|L_2|+|L_3|
\lesssim
\bigl(\|u\|_{L^\infty}\|\nabla D_t{u}_j\|_{L^2}&+\|B\|_{L^\infty}\|\nabla D_t{B}_j\|_{L^2}\\
&+\|D_t{B}\|_{L^3}\|\nabla{B}_j\|_{L^6}\bigr)
\|\partial_t D_t{u}_j\|_{L^2},
\end{split}
\end{equation*}
and
\begin{align*}
&|L_4|
\le
\|B\|_{L^6}\|\nabla u\|_{L^6}\|\nabla{B}_j\|_{L^6}
\|\partial_t D_t{u}_j\|_{L^2},\\
&
|L_5|+|L_6|
\le
\bigl(\|\nabla^2 u\|_{L^3}
\|\nabla u_j\|_{L^6}+\|\nabla u\|_{L^3}\|\nabla^2 u_j\|_{L^6}\bigr)
\|\partial_t D_t{u}_j\|_{L^2}.
\end{align*}

Notice that $
D_t\nabla\Pi_j
=
\nabla D_t\Pi_j-(\nabla u\cdot\nabla)\Pi_j,$ one has
\begin{equation*}\label{est-const-hL6-j-5}
|L_7| \leq |\int_{\R^3}\nabla{D}_t\Pi_j|\partial_t D_t{u}_j\,dx|+ |\int_{\R^3}(\nabla u\cdot\nabla)\Pi_j\partial_t D_t{u}_j\,dx |.
\end{equation*}
Yet due to $\dive u_j=0,$ we get, by using integration by parts, that
\begin{align*}
\bigl|\int_{\R^3}\nabla D_t\Pi_j|\partial_t D_t{u}_j\,dx\bigr|
& =
\bigl|\int_{\R^3}D_t\Pi_j\partial_t\bigl[\partial_ku^{\ell}\partial_{\ell}u_j^k\bigr]\,dx\bigr|
\\&
=\bigl|\int_{\R^3}D_t\Pi_j\,\dv\bigl[\partial_tu\cdot\nabla u_j\bigr]\,dx
+\int_{\R^3} D_t\Pi_j\,\dv\bigl[\partial_tu_j\cdot\nabla u\bigr]\,dx\bigr|
\\&
=\bigl|\int_{\R^3}\nabla D_t\Pi_j |(u_t\cdot\nabla) u_j\,dx
+\int_{\R^3}\nabla D_t\Pi_j |(\partial_tu_j\cdot\nabla) u\,dx\bigr|
\\&
\le
\bigl(\|u_t\|_{L^3}\|\nabla u_j\|_{L^6}
+\|\partial_tu_j\|_{L^6}\|\nabla u\|_{L^3}\bigr)\|\nabla D_t\Pi_j\|_{L^2},
\end{align*}
we thus obtain
\begin{equation*}\label{est-const-hL6-j-8}
\begin{split}
 |L_7|\lesssim &
\|\nabla u\|_{L^3}\|\nabla\Pi_j\|_{L^6}
\|\partial_t D_t{u}_j\|_{L^2}\\
&+\bigl(\|u_t\|_{L^3}\|\nabla u_j\|_{L^6}
+\|\partial_tu_j\|_{L^6}\|\nabla u\|_{L^3}\bigr)\|\nabla D_t\Pi_j\|_{L^2}.
\end{split}
\end{equation*}

By substituting the above estimates into \eqref{est-const-hL6-j-2}, we find
\begin{equation}\label{est-const-hL6-j-16a}
\begin{split}
\frac{d}{dt}\|\nabla &D_t{u}_j(t)\|_{L^2}^2+2\|\sqrt{\rho}\,\partial_t D_t{u}_j\|_{L^2}^2\\
\lesssim \Bigl(&\|(u, B)\|_{L^\infty}\|(\nabla D_t{u}_j, \nabla D_t{B}_j)\|_{L^2} +\|\nabla u\|_{L^3}\|(\nabla^2 u_j, \nabla\Pi_j)\|_{L^6}\\
&+\bigl(\|(D_t{B}, \nabla^2 u)\|_{L^3}+
\|B\|_{L^6}\|\nabla u\|_{L^6}\bigr)\|(\nabla{u}_j, \nabla{B}_j)\|_{L^6}\Bigr)
\|\partial_t D_t{u}_j\|_{L^2}\\
& +\bigl(\|u_t\|_{L^3}\|\nabla u_j\|_{L^6}
+\|\partial_tu_j\|_{L^6}\|\nabla u\|_{L^3}\bigr)\|\nabla D_t\Pi_j\|_{L^2}.
\end{split}
\end{equation}

In order to control $\|\nabla D_t\Pi_j\|_{L^2}$, in view of  \eqref{est-const-hL6-j-1}, we write
\begin{equation*}\label{est-const-hL6-j-9}
\nabla D_t \Pi_j- D_t\Delta{u}_j=-\rho \,D_t^2{u}_j +D_tB\cdot \nabla B_j+B\cdot D_t\nabla B_j+\nabla u\cdot \nabla\Pi_j,
\end{equation*}
which implies
\begin{equation*}\label{est-const-hL6-j-10}
\begin{split}
\|\nabla D_t \Pi_j- D_t\Delta{u}_j\|_{L^2}^2
\lesssim&
\|D_t^2u_j\|_{L^2}^2+\|D_tB\|_{L^3}^2\|\nabla B_j\|_{L^6}^2\\
&+\|B\|_{L^{\infty}}^2\|D_t\nabla B_j\|_{L^2}^2
+\|\nabla u\|_{L^3}^2\|\nabla\Pi_j\|_{L^6}^2.
\end{split}
\end{equation*}
Notice from $\dv\,u=0$ that,
\begin{equation*}\label{est-const-hL6-j-11}
\begin{split}
&\|\nabla D_t \Pi_j- D_t\Delta{u}_j\|_{L^2}^2=\|\nabla D_t \Pi_j\|_{L^2}^2+\|D_t\Delta{u}_j\|_{L^2}^2-2\int_{\mathbb{R}^3}\nabla D_t \Pi_j| D_t\Delta{u}_j\,dx,\\
&\bigl|\int_{\mathbb{R}^3}\nabla D_t \Pi_j| D_t\Delta{u}_j\,dx\bigr|=\bigl|\int_{\mathbb{R}^3}D_t \Pi_j| \nabla \cdot (D_t\Delta{u}_j)\,dx\bigr|\\
&\qquad\qquad\qquad\qquad\qquad=\bigl|\int_{\mathbb{R}^3}\partial_k D_t \Pi_j| (\partial_{\ell}u^{k}\Delta{u}_j^{\ell})\,dx\bigr|\lesssim\|\nabla D_t \Pi_j\|_{L^2} \|\nabla u\|_{L^3}\|\Delta{u}_j\|_{L^6}
\end{split}
\end{equation*}
and
\begin{equation*}\label{est-const-hL6-j-12}
\begin{split}
&\|D_t^2u_j\|_{L^2}
\lesssim
\|\partial_t D_t{u}_j\|_{L^2}+\|u\|_{L^\infty}\|\nabla D_t{u}_j\|_{L^2},\\
&\|D_t\nabla B_j\|_{L^2}\lesssim
\|\nabla D_t B_j\|_{L^2}+\|\nabla u\|_{L^3}\|\nabla {B}_j\|_{L^6},
\end{split}
\end{equation*}
we thus obtain
\begin{equation*}\label{est-const-hL6-j-13}
\begin{split}
&\|\nabla D_t\Pi_j\|_{L^2}^2+\|D_t\Delta{u}_j\|_{L^2}^2\\
&\lesssim
\|\partial_t D_t{u}_j\|_{L^2}^2+\|(u, B)\|_{L^\infty}^2\|(\nabla D_t{u}_j, \nabla D_t B_j)\|_{L^2}^2
\\
&+ (\|D_tB\|_{L^3}^2+\|B\|_{L^{\infty}}^2\|\nabla u\|_{L^3}^2) \|\nabla {B}_j\|_{L^6}^2
+\|\nabla u\|_{L^3}^2\|(\nabla^2 u_j, \nabla\Pi_j)\|_{L^6}^2.
\end{split}
\end{equation*}
By substituting the above estimate into \eqref{est-const-hL6-j-16a}, we find
\begin{equation*}\label{est-const-hL6-j-14}
\begin{split}
\frac{d}{dt}&\|\nabla D_t{u}_j(t)\|_{L^2}^2+2\|\sqrt{\rho}\,\partial_t D_t{u}_j\|_{L^2}^2\\
\lesssim& \Bigl(\|(u, B)\|_{L^\infty}\|(\nabla D_t{u}_j, \nabla D_t{B}_j)\|_{L^2} +\|\nabla u\|_{L^3}\|(\partial_t u_j, \nabla^2 u_j, \nabla\Pi_j)\|_{L^6}\\
& +\bigl(\|(D_t{B}, \nabla^2 u, \partial_t u)\|_{L^3}+
\|B\|_{L^6}\|\nabla u\|_{L^6})\|(\nabla{u}_j, \nabla{B}_j\bigr)\|_{L^6}\Bigr)
\|\partial_t D_t{u}_j\|_{L^2}\\
&+\bigl(\|u_t\|_{L^3}\|\nabla u_j\|_{L^6}
+\|\partial_tu_j\|_{L^6}\|\nabla u\|_{L^3}\bigr)\Bigl( \|(u, B)\|_{L^\infty}\|(\nabla D_t{u}_j, \nabla D_t B_j)\|_{L^2}\\
&
+ (\|D_tB \|_{L^3}+\|B\|_{L^{\infty}}\|\nabla u\|_{L^3}) \|\nabla {B}_j\|_{L^6}
+\|\nabla u\|_{L^3}\|(\nabla^2 u_j, \nabla\Pi_j)\|_{L^6}\Bigr).
\end{split}
\end{equation*}
Applying Young's inequality yields
\begin{equation*}\label{est-const-hL6-j-16}
\begin{split}
&\frac{d}{dt}\|\nabla D_t{u}_j(t)\|_{L^2}^2+c_5\|\partial_t D_t{u}_j\|_{L^2}^2\\
&\lesssim  \|(u, B)\|_{L^\infty}^2\|(\nabla D_t{u}_j, \nabla D_t{B}_j)\|_{L^2}^2 +\|\nabla u\|_{L^3}^2\|(\partial_t u_j, \nabla^2 u_j, \nabla\Pi_j)\|_{L^6}^2\\
& +\bigl(\|(D_t{B}, \nabla^2 u, \partial_t u)\|_{L^3}^2+
\|B\|_{L^6}^2\|\nabla u\|_{L^6}^2+\|B\|_{L^{\infty}}^2\|\nabla u\|_{L^3}^2\bigr)\|(\nabla{u}_j, \nabla{B}_j)\|_{L^6}^2,
\end{split}
\end{equation*}
from which and \eqref{est-const-basic-j-2}, we infer
\begin{equation}\label{est-const-hL6-j-16}
\begin{split}
&\frac{d}{dt}\|\nabla D_t{u}_j(t)\|_{L^2}^2+c_5\|\partial_t D_t{u}_j\|_{L^2}^2\\
&\lesssim  \|(u, B)\|_{L^\infty}^2\|(\nabla D_t{u}_j, \nabla D_t{B}_j)\|_{L^2}^2 +\|\nabla u\|_{L^3}^2\|(\partial_t u_j, \nabla^2 u_j, \nabla\Pi_j)\|_{L^6}^2\\
& +\bigl(\|(D_t{B}, \nabla^2 u, u_t)\|_{L^3}^2+
\|B\|_{L^6}^2\|\nabla u\|_{L^6}^2+\|B\|_{L^{\infty}}^2\|\nabla u\|_{L^3}^2\bigr)\|(\partial_t{u}_j, \partial_t{B}_j)\|_{L^2}^2.
\end{split}
\end{equation}
Due to  $\dv\,u=0$, one has
\begin{align*}
\|(\nabla D_t\Pi_j, \nabla^2 D_t{u}_j)\|_{L^2}^2
\lesssim &\|2\partial_ku \cdot \nabla \partial_k {u}_j+\Delta u\cdot \nabla {u}_j\|_{L^2}^2+\|(\nabla D_t\Pi_j, D_t\Delta{u}_j)\|_{L^2}^2\\
\lesssim & \|\partial_t D_t{u}_j\|_{L^2}^2+\|(u, B)\|_{L^\infty}^2\|(\nabla D_t{u}_j, \nabla D_t B_j)\|_{L^2}^2
\\
&+ \bigl(\|(D_tB, \nabla^2u)\|_{L^3}^2+\|B\|_{L^{\infty}}^2\|\nabla u\|_{L^3}^2\bigr) \|(\nabla {u}_j, \nabla {B}_j)\|_{L^6}^2\\
&+\|\nabla u\|_{L^3}^2\|(\nabla^2 u_j, \nabla\Pi_j)\|_{L^6}^2.
\end{align*}
While in view of  \eqref{model-3d-freq-1} with $\mu(\rho)=1,$ we deduce from the classical estimates on Stokes system that
\begin{equation*}\label{est-const-hL6-j-19}
\begin{split}
\|(\partial_t u_j, D_t u_j, \nabla^2 u_j, \nabla\Pi_j)\|_{L^6}
&\lesssim
\|D_tu_j\|_{L^6}
+\|(u,B)\|_{L^{\infty}}\|(\nabla u_j,\nabla B_j)\|_{L^6}.
\end{split}
\end{equation*}
Thanks to   \eqref{est-basic-2-23}, we obtain
\begin{equation}\label{est-const-hL6-j-19aaa}
\begin{split}
\|(\partial_t u_j, D_t u_j, \nabla^2 u_j, \nabla\Pi_j)\|_{L^6} \lesssim
\|\nabla D_tu_j\|_{L^2}
+\|(u,B)\|_{L^{\infty}}\|(\partial_t u_j, \partial_t  B_j)\|_{L^2}
\end{split}
\end{equation}
and
 \begin{equation}\label{est-const-hL6-j-19bbb}
\begin{split}
\|(\nabla D_t\Pi_j, \nabla^2 D_t{u}_j)&\|_{L^2}^2\lesssim\bigl(\|(D_tB, \nabla^2u)\|_{L^3}^2+\|(u, B)\|_{L^{\infty}}^2\|\nabla u\|_{L^3}^2\bigr) \|(\partial_t u_j, \partial_t  B_j)\|_{L^2}^2\\
&+\|\partial_t D_t{u}_j\|_{L^2}^2+(\|(u, B)\|_{L^\infty}^2+\|\nabla u\|_{L^3}^2)\|(\nabla D_t{u}_j, \nabla D_t B_j)\|_{L^2}^2.
\end{split}
\end{equation}

 By combining \eqref{est-const-hL6-j-16} with \eqref{est-const-hL6-j-19aaa}-\eqref{est-const-hL6-j-19bbb}  and using the
  fact: $\|D_t{B}\|_{L^3}\lesssim \|\partial_t{B}\|_{L^3}+\|u\|_{L^{6}}\|\nabla B\|_{L^6}$, we achieve
 \begin{equation}\label{est-const-hL6-j-16ccc}
\begin{split}
\frac{d}{dt}\|\nabla &D_t{u}_j(t)\|_{L^2}^2+c_6\|(\partial_t D_t{u}_j, \nabla D_t\Pi_j, \nabla^2 D_t{u}_j)\|_{L^2}^2\\
\lesssim &  (\|(u, B)\|_{L^\infty}^2+\|\nabla u\|_{L^3}^2)\|(\nabla D_t{u}_j, \nabla D_t B_j)\|_{L^2}^2\\
& +\bigl(\|({B}_t, \nabla^2 u, u_t)\|_{L^3}^2+
\|(u, B)\|_{L^6}^2\|(\nabla u, \nabla B)\|_{L^6}^2\\
&+\|(u, B)\|_{L^{\infty}}^2\|\nabla u\|_{L^3}^2\bigr)\|(\partial_t{u}_j, \partial_t{B}_j)\|_{L^2}^2.
\end{split}
\end{equation}
By multiplying \eqref{est-const-hL6-j-16ccc} by $t^2$ and then integrating the resulting inequality over $[0,T],$ we obtain
\begin{equation*}\label{est-const-hL6-j-18}
\begin{split}
\|t\,&\nabla D_t{u}_j\|_{L^{\infty}_T(L^2)}^2+\|t\,(\partial_t D_t{u}_j, \nabla D_t\Pi_j, \nabla^2 D_t{u}_j)\|_{L^2_T(L^2)}^2\lesssim  \|{t}^{\frac{1}{2}}\,\nabla D_t{u}_j\|_{L^2_T(L^2)}^2 \\
&+\bigl(\|{t}^{\frac{1}{2}}\,(u, B)\|_{L^{\infty}_T(L^\infty)}^2+\|t^{\frac{1}{2}}\,\nabla u\|_{L^{\infty}_T(L^3)}^2\bigr)\|{t}^{\frac{1}{2}}\,(\nabla D_t{u}_j, \nabla D_t{B}_j)\|_{L^2_T(L^2)}^2 \\
& +\bigl(\|{t}^{\frac{1}{2}}\,({B}_t, \nabla^2 u,  u_t)\|_{L^2_T(L^3)}^2+
\|{t}^{\frac{1}{4}}\,(u, B)\|_{L^{\infty}_T(L^6)}^2\|{t}^{\frac{1}{4}}\,(\nabla u, \nabla B)\|_{L^2_T(L^6)}^2\\
&\quad +\|{t}^{\frac{1}{2}}\,(u, B)\|_{L^{\infty}_T(L^{\infty})}^2\|\nabla u\|_{L^2_T(L^3)}^2\bigr)\|{t}^{\frac{1}{2}}\,(\partial_t{u}_j, \partial_t{B}_j)\|_{L^{\infty}_T(L^2)}^2,
\end{split}
\end{equation*}
from which, \eqref{est-const-basic-0}, \eqref{est-const-basic-j-1} and the inequality
 \begin{align*}
 \|{t}^{\frac{1}{2}}\,({B}_t, \nabla^2 u,  u_t)\|_{L^2_T(L^3)}^2\lesssim
  &\|{t}^{\frac{1}{4}}\,({B}_t, \nabla^2 u, u_t)\|_{L^2_T(L^2)}\|{t}^{\frac{3}{4}}\,({B}_t, \nabla^2 u, u_t)\|_{L^2_T(L^6)}\\
 \lesssim& \|{t}^{\frac{1}{4}}\,({B}_t, \nabla^2 u, u_t)\|_{L^2_T(L^2)}\bigl(\|{t}^{\frac{3}{4}}\,(\nabla{B}_t, \nabla u_t)\|_{L^2_T(L^2)}\\
 &+\|{t}^{\frac{3}{4}}\,\nabla^2 u\|_{L^2_T(L^6)}\bigr)\lesssim
\|(u_0,B_0)\|_{\dot H^{\frac{1}{2}}}^2,
\end{align*}
we infer
\begin{equation}\label{est-const-hL6-j-19}
\begin{split}
&\|t\,\nabla D_t{u}_j\|_{L^{\infty}_T(L^2)}^2+\|t\,(\partial_t D_t{u}_j, \nabla D_t\Pi_j, \nabla^2 D_t{u}_j)\|_{L^2_T(L^2)}^2\\
&\lesssim \bigl(1+\|(u_0,B_0)\|_{\dot H^{\frac{1}{2}}}^4\bigr) 2^{2j}\|(\dot\Delta_ju_0,\dot\Delta_jB_0)\|_{L^2}^2\lesssim 2^{2j}\|(\dot\Delta_ju_0,\dot\Delta_jB_0)\|_{L^2}^2.
\end{split}
\end{equation}
\eqref{est-const-hL6-j-19} together with \eqref{est-const-hL6-j-19aaa} ensures that
\begin{equation}\label{est-const-hL6-j-22}
\begin{split}
&\|t\,(\partial_t u_j, D_t u_j, \nabla^2 u_j, \nabla\Pi_j)\|_{L^\infty_t(L^6)} \lesssim
\|t\,\nabla D_t{u}_j\|_{L^{\infty}_T(L^2)}^2\\
&\qquad +\|{t}^{\frac{1}{2}}\,(u, B)\|_{L^{\infty}_T(L^{\infty})}^2 \|{t}^{\frac{1}{2}}\,(\partial_t{u}_j, \partial_t{B}_j)\|_{L^{\infty}_T(L^2)}^2\lesssim  2^{2j}\|(\dot\Delta_ju_0,\dot\Delta_jB_0)\|_{L^2}^2.
\end{split}
\end{equation}

By summarizing the estimates \eqref{est-const-hL6-j-19} and \eqref{est-const-hL6-j-22} we conclude the proof of \eqref{est-const-hL6-j-0}.
This completes the proof of  Lemma \ref{prop-Lip1}.
\end{proof}

\begin{prop}\label{prop-B21-lip2}{\sl Under the assumptions of Lemma \ref{Prop-const-j-Lip-1}, if we assume in addition that $(u_0, B_0) \in \dot{B}^{\frac{1}{2}}_{2,1}\times \dot{B}^{\frac{1}{2}}_{2,1}$, then  \eqref{est-const-B21-0} holds for any $T<T^\ast.$
}
\end{prop}
\begin{proof} We first deduce from \eqref{est-const-basic-j-1a} and \eqref{est-const-hL6-j-0} that
\begin{equation*}\label{est-const-B21-1}
\begin{split}
&\|t(\nabla^2u_j,\nabla\Pi_j)\|_{L^2_T(L^6)}
\lesssim
\|(\dot\Delta_ju_0,\dot\Delta_jB_0)\|_{L^2}
\lesssim
{d}_j2^{-\frac{j}{2}}\|(u_0,B_0)\|_{\dot B^{\frac{1}{2}}_{2,1}},\\
&\|t(\nabla^2u_j,\nabla\Pi_j)\|_{L^\infty_T(L^6)}\lesssim 2^{j}
\|(\dot\Delta_ju_0,\dot\Delta_jB_0)\|_{L^2}
\lesssim
{d}_j2^{\frac{j}{2}}\|(u_0,B_0)\|_{\dot B^{\frac{1}{2}}_{2,1}}.
\end{split}
\end{equation*} Here and below, we always denote $(d_j)_{j\in\Z}$ to be a generic element of $\ell^1(\Z)$ so that $\sum_{j\in\Z}d_j=1.$
Then it follows from the interpolation:  $L^{4, 1}_T(L^6)=[L^2_T(L^6), L^\infty_T(L^6)]_{\frac{1}{2}, 1}$, that
\begin{equation*}\label{est-const-B21-2}
\begin{split}
\|t\bigl(\nabla^2u_j,\nabla\Pi_j\bigr)\|_{L^{4,1}_T(L^6)}
\lesssim d_j
\|(u_0,B_0)\|_{\dot B^{\frac{1}{2}}_{2,1}},
\end{split}
\end{equation*}
which implies
\begin{equation}\label{est-const-B21-3}
\begin{split}
\|t\bigl(\nabla^2u,\nabla\Pi\bigr)\|_{L^{4,1}_T(L^6)} \lesssim \sum_{j\in \mathbb{Z}}\|t\bigl(\nabla^2u_j,\nabla\Pi_j\bigr)\|_{L^{4,1}_T(L^6)}
\lesssim
\|(u_0,B_0)\|_{\dot B^{\frac{1}{2}}_{2,1}}.
\end{split}
\end{equation}

Along the same line, it follows from  Lemma \ref{Prop-const-j-Lip-1} that
\begin{align*}
&\|t^{\frac{1}{2}}(\partial_tu_j,\,\nabla^2\,u_j,\nabla\Pi_j)\|_{L^2_T(L^2)}
\lesssim d_{j}2^{-\frac{j}{2}}\|u_0\|_{\dot{B}^{\frac{1}{2}}_{2, 1}},
\\&
\|t^{\frac{1}{2}}(\partial_tu_j,\,\nabla^2\,u_j,\nabla\Pi_j)\|_{L^\infty_T(L^2)}
\lesssim d_{j}2^{\frac{j}{2}}\|u_0\|_{\dot{B}^{\frac{1}{2}}_{2, 1}},
\end{align*}
which together with the interpolation:
$L^{4, 1}_T(L^2)=[L^2_T(L^2), L^\infty_T(L^2)]_{\frac{1}{2}, 1},$ ensures that
\begin{equation*}\label{est-const-B21-13}
\begin{split}
\|t^{\frac{1}{2}}(\partial_tu_j,\,\nabla^2\,u_j,\nabla\Pi_j)\|_{L^{4, 1}_T(L^2)}
\lesssim
 d_j\|(u_0,B_0)\|_{\dot{B}^{\frac{1}{2}}_{2, 1}},
\end{split}
\end{equation*}
so that one has
\begin{equation}\label{est-const-B21-14}
\begin{split}
\|t^{\frac{1}{2}}(u_t,\,\nabla^2\,u,\nabla\Pi)\|_{L^{4, 1}_T(L^2)}
\lesssim \sum_{j\in \mathbb{Z}}\|t^{\frac{1}{2}}(\partial_tu_j,\,\nabla^2\,u_j,\nabla\Pi_j)\|_{L^{4, 1}_T(L^2)}
\lesssim
 \|(u_0,B_0)\|_{\dot{B}^{\frac{1}{2}}_{2, 1}}.
\end{split}
\end{equation}

Thanks to \eqref{est-const-B21-3} and \eqref{est-const-B21-14}, we deduce from Proposition \ref{Neil} that
\begin{equation}\label{est-const-B21-4}
\begin{split}
\|\nabla u\|_{L^1_T(L^\infty)}
&\lesssim
\int_0^Tt^{-\frac{3}{4}}\|{t}^{\frac{1}{2}}\,\nabla^2u(t)\|_{L^2}^{\frac{1}{2}}
\|t\nabla^2u(t)\|_{L^6}^{\frac{1}{2}}\,dt
\\&
\lesssim
\|t^{-\frac{3}{4}}\|_{L^{\frac{4}{3},\infty}(\R^+)}
\|{t}^{\frac{1}{2}}\,\nabla^2u\|_{L^{4,1}_T(L^2)}^{\frac{1}{2}}
\|t\nabla^2u\|_{L^{4,1}_T(L^6)}^{\frac{1}{2}}
\lesssim
\|(u_0,B_0)\|_{\dot B^{\frac{1}{2}}_{2,1}},
\end{split}
\end{equation}
and
\begin{equation}\label{est-const-B21-16}
\begin{split}
\int_0^T&\|(\nabla^2u,\nabla\Pi)(t)\|_{L^3}\,dt\\
\lesssim&
\int_0^T t^{-\frac{3}{4}}\|t^{\frac{1}{2}}(\nabla^2u,\nabla\Pi)(t)\|_{L^2}^{\frac{1}{2}}
\|t(\nabla^2u,\nabla\Pi)(t)\|_{L^6}^{\frac{1}{2}}\,dt
\\
\lesssim&
\|t^{-\frac{3}{4}}\|_{L^{\frac{4}{3},\infty}(\R^+)}
\|t^{\frac{1}{2}}(\nabla^2\,u,\nabla\Pi)\|_{L^{4,1}_T(L^2)}^{\frac{1}{2}}
\|t(\nabla^2\,u,\nabla\Pi)\|_{L^{4,1}_T(L^6)}^{\frac{1}{2}}
\lesssim
\|(u_0,B_0)\|_{\dot B^{\frac{1}{2}}_{2,1}}.
\end{split}
\end{equation}

While we get, by applying Lemma   \ref{Prop-const-j-Lip-1} and Lemma \ref{lem2.1}, that
\begin{equation*}
\begin{split}
&\|\dot\Delta_qu\|_{L^\infty_T(L^2)}
+\|\nabla\dot\Delta_qu\|_{L^2_T(L^2)}\\
&\lesssim
\sum_{q\le j}\bigl(\|\dot\Delta_qu_j\|_{L^\infty_T(L^2)}
+\|\nabla\dot\Delta_qu_j\|_{L^2_T(L^2)}\bigr)
+2^{-q}\sum_{j\le q}\bigl(\|\nabla\dot\Delta_qu_j\|_{L^\infty_T(L^2)}
+\|\nabla^2\dot\Delta_qu_j\|_{L^2_T(L^2)}\bigr)
\\&
\lesssim
\sum_{q\le j}\bigl(\|u_j\|_{L^\infty_T(L^2)}
+\|\nabla u_j\|_{L^2_T(L^2)}\bigr)
+2^{-q}\sum_{j\le q}\bigl(\|\nabla u_j\|_{L^\infty_T(L^2)}
+\|\nabla^2 u_j\|_{L^2_T(L^2)}\bigr)
\\&
\lesssim
d_q2^{-\frac{q}{2}}\|(u_0,B_0)\|_{\dot B^{\frac{1}{2}}_{2,1}},
\end{split}
\end{equation*}
which implies
\begin{equation}\label{est-const-B21-6}
\begin{split}
\|u\|_{\widetilde L^\infty_T(\dot B^{\frac{1}{2}}_{2,1})}
+\|u\|_{\widetilde L^2_T(\dot B^{\frac{3}{2}}_{2,1})}
\lesssim
\|(u_0,B_0)\|_{\dot B^{\frac{1}{2}}_{2,1}}.
\end{split}
\end{equation}
Similarly we deduce that
\begin{equation}\label{est-const-B21-7}
\begin{split}
\|B\|_{\widetilde L^\infty_T(\dot B^{\frac{1}{2}}_{2,1})}
\lesssim
\|(u_0,B_0)\|_{\dot B^{\frac{1}{2}}_{2,1}}.
\end{split}
\end{equation}

On the other hand, we deduce from \eqref{S2eq1} and \eqref{est-const-basic-j-1},
 that
\begin{equation*}\label{est-const-B21-8}
\begin{split}
\|\dot\Delta_qB\|_{L^2_T(L^6)}
&\lesssim
\sum_{q\le j}\|\dot\Delta_qB_j\|_{L^2_T(L^6)}
+2^{-q}\sum_{j\le q}\|\nabla\dot\Delta_qB_j\|_{L^2_T(L^6)}
\\&
\lesssim
\sum_{q\le j}\|\nabla B_j\|_{L^2_T(L^2)}
+2^{-q}\sum_{j\le q}\|\nabla B_j\|_{L^2_T(L^6)}
\\&
\lesssim
d_q2^{-\frac{q}{2}}\|(u_0,B_0)\|_{\dot B^{\frac{1}{2}}_{2,1}},
\end{split}
\end{equation*}
which ensures that
\begin{equation}\label{est-const-B21-9}
\begin{split}
\|B\|_{\widetilde L^2_T(\dot B^{\frac{1}{2}}_{6,1})}
\lesssim
\|(u_0,B_0)\|_{\dot B^{\frac{1}{2}}_{2,1}}.
\end{split}
\end{equation}

Finally it follows from
Lemmas \ref{Prop-const-j-Lip-1} and  \ref{prop-Lip1}  that
\begin{equation*}\label{est-const-B21-10}
\begin{split}
\|t\nabla\dot\Delta_q D_t{u}\|_{L^2_T(L^2)}
&\lesssim
\sum_{q\le j}\|t\nabla\dot\Delta_q D_t{u}_j\|_{L^2_T(L^2)}
+2^{-q}\sum_{j\le q}\|\nabla^2\dot\Delta_q D_t{u}_j\|_{L^2_T(L^2)}
\\&
\lesssim
d_q2^{-\frac{q}{2}}\|(u_0,B_0)\|_{\dot B^{\frac{1}{2}}_{2,1}},
\end{split}
\end{equation*}
so that one has
\begin{equation}\label{est-const-B21-11}
\begin{split}
\|t D_t{u}\|_{\widetilde L^2_T(\dot B^{\frac{3}{2}}_{2,1})}
\lesssim
\|(u_0,B_0)\|_{\dot B^{\frac{1}{2}}_{2,1}}.
\end{split}
\end{equation}

By summarizing the estimates (\ref{est-const-B21-3}-\ref{est-const-B21-11}), we conclude the proof of \eqref{est-const-B21-0}. This
 completes the proof of  Proposition \ref{prop-B21-lip2}.
\end{proof}

\renewcommand{\theequation}{\thesection.\arabic{equation}}
\setcounter{equation}{0}

\section{Proof of Theorem \ref{thm-GWS-MHD}}\label{Sect3}

This section is devoted to  the  proof of  Theorem \ref{thm-GWS-MHD}.

\begin{proof}[Proof of Theorem \ref{thm-GWS-MHD}]
By mollifying the initial data $(\rho_0, u_0)$ to be $(\rho_{0 \epsilon }, u_{0 \epsilon },  B_{0 \epsilon })$ , we get, by using
modifications of the
  classical well-posedness theory of inhomogeneous incompressible Navier-Stokes system (see \cite{AZ2015-1}) that  the system
  \eqref{1.2} has a unique local solution $(\rho_{\epsilon}, u_{\epsilon}, B_{\epsilon})$ on $[0, T^{\ast}_{\epsilon})$ for some positive lifespan $T^\ast_{\epsilon}.$ If the constants $\mathfrak{c}$ and $\varepsilon_0$ are sufficiently small in \eqref{small-data-1}, we deduce from Proposition
  \ref{prop-vaviable-1} that  $(\rho_{\epsilon}, u_{\epsilon}, B_{\epsilon})$ verify the estimates \eqref{bdd-density-visc-1}  and \eqref{est-variable-2} for any $T<T^\ast_{\epsilon}.$ Then a standard continuous argument shows  that $T_{\epsilon}^\ast=+\infty$.
In particular, we have $(u_{\epsilon}, B_{\epsilon}) \in (C([0, +\infty); \dot{H}^{\frac{1}{2}})\cap L^4(\mathbb{R}^+; \dot{H}^{1}))^2$, and for any $T \in [0, +\infty]$,  $(u_{\epsilon}, B_{\epsilon})$ satisfy  the inequality \eqref{est-variable-2}.
Then we  get, by using a  compactness argument  similar to that in \cite{LP1996}, that there exists  $\rho \in C_{\rm w}([0,\infty); L^{\infty})$
so that
\beq \label{S3eq1}
\begin{split}
&\rho_{\epsilon} \rightharpoonup \rho \quad \mbox{weak $\ast$ in} \ \ L^\infty(\R^+\times\R^3) \andf\\
&\rho_{\epsilon} \to \rho \quad \mbox{strongly in} \ \ L^r_{\mbox{loc}}(\R^+\times\R^3)\ \ \mbox{for any} \ r<\infty.
\end{split}
\eeq
Notice that for any $T<\infty,$  it follows from Proposition \ref{Neil} that
\beno
\|(\p_tu_{\epsilon},\p_tB_{\epsilon})\|_{L^{\f43,\infty}_T(L^2)}\lesssim \|t^{-\f14}\|_{L^{4,\infty}_T}\|(\p_tu_{\epsilon},\p_tB_{\epsilon})\|_{L^{2}_T(L^2)},
\eeno
which together with \eqref{est-variable-2} ensures that
\beno
\|(\p_tu_{\epsilon},\p_tB_{\epsilon})\|_{L^{\f43,\infty}_T(L^2)}+\|(\na u_{\epsilon}, \na B_{\epsilon})\|_{L^2_T(L^3)}\lesssim
\|(u_0,B_0)\|_{\dot H^{\f12}}.
\eeno
Then we deduce from Ascoli-Arzela Theorem that there exist $(u, B) \in \bigl(L^{\infty}([0, +\infty); \dot{H}^{\frac{1}{2}})\cap L^4(\mathbb{R}^+; \dot{H}^{1})\cap L^2(\mathbb{R}^+; \dot{W}^{1,3})\bigr)^2$ so that
\beq \label{S3eq2}
\begin{split}
&(u_{\epsilon},B_{\epsilon}) \rightharpoonup  (u,B) \quad\mbox{weakly in} \  L^4(\mathbb{R}^+; \dot{H}^{1})\andf\\
&(u_{\epsilon},B_{\epsilon})  \to (u,B) \quad\mbox{strongly in} \ \ L^2_{\mbox{loc}}(\R^+;\ L^r_{\mbox{loc}}(\R^3))\ \ \mbox{for any} \ r<\infty.
\end{split}
\eeq

Thanks to \eqref{S3eq1} and \eqref{S3eq2}, we conclude that $(\rho, u, B)$ thus obtained is indeed a global weak solution of \eqref{1.2}.
 Furthermore, it follows from \eqref{est-variable-2} and Fatou's Lemma that $(\rho, u, \nabla\Pi,  B)$ satisfies the estimates \eqref{bdd-density-visc-1} and \eqref{est-variable-2} for any $T \in [0, +\infty]$.

Finally let us prove  $(u, B)\in (C([0,\infty); {\dot{H}^{\frac{1}{2}}}))^2.$  Indeed it follows from  \eqref{est-variable-2} that $$\|u\|_{\widetilde{L}^{\infty}(\mathbb{R}^+; \dot{H}^{\frac{1}{2}})}\leq C\|(u_0, B_0)\|_{\dot{H}^{\frac{1}{2}}}.$$ Then for any $\varepsilon>0$, there is a positive $j_0=j_0(\varepsilon)\in \mathbb{N}$ so that
 \begin{equation*}\label{1.3aa}
\begin{split}
4\sum_{|j| \geq j_0}2^{j}\|\dot{\Delta}_ju\|_{L^{\infty}(\R^+; L^2)}^2<\varepsilon.
\end{split}
\end{equation*}
Then for any  $t \in [0, +\infty),\,h>0$, we have
\begin{equation*}\label{1.4-1}
\begin{split}
\|u(t+h)-u(t)\|_{\dot{H}^{\frac{1}{2}}}^2&=\sum_{j \in \mathbb{Z}}2^{ j}\|\dot{\Delta}_j(u(t+h)-u(t))\|_{L^2}^2\\
&\leq \sum_{|j| \leq j_0-1}2^{j}\|\dot{\Delta}_j(u(t+h)-u(t))\|_{L^2}^2+4\sum_{|j| \geq j_0}2^{j}\|\dot{\Delta}_ju\|_{L^{\infty}_T(L^2)}^2\\
&\leq 2^{j_0} \|u(t+h)-u(t)\|_{L^2}^2+\varepsilon,
\end{split}
\end{equation*}
from which, we infer
\begin{equation*}\label{1.4-2}
\begin{split}
\|u(t+h)-u(t)\|_{\dot{H}^{\frac{1}{2}}}^2&\leq 2^{j_0} \bigl\|\int_t^{t+h}\tau^{-\frac{1}{4}}\, \tau^{\frac{1}{4}}u_{\tau}(\tau)\,d\tau\bigr\|_{L^2}^2+\varepsilon\\
&\leq 2^{j_0}  \|\tau^{-\frac{1}{4}}\|_{L^2([t, t+h])}^2\|\tau^{\frac{1}{4}}u_{\tau} \|_{L^2_{[t, t+h]}(L^2)}^2 +\varepsilon\\
&\leq  C\, 2^{j_0+2}\|(u_0, B_0)\|_{\dot{H}^{\frac{1}{2}}}\,h^{\frac{1}{2}} +\varepsilon,
\end{split}
\end{equation*}
where we used the fact: $\|t^{\frac{1}{4}}u_t\|_{L^2(\R^+; L^2)} \leq C \|(u_0, B_0)\|_{\dot{H}^{\frac{1}{2}}}$ in \eqref{est-variable-2}.
This shows that $u\in C([0,\infty);$ $ \dot{H}^{\frac{1}{2}})$. Along the same line, we can verify that $B\in C([0,\infty);$ $ \dot{H}^{\frac{1}{2}})$.
This completes the proof of Theorem \ref{thm-GWS-MHD}.
\end{proof}

\renewcommand{\theequation}{\thesection.\arabic{equation}}
\setcounter{equation}{0}

\section{Proof of Theorem \ref{mainthm-GWP}}\label{Sect4}

The goal of the this section is to present the proof of Theorem \ref{mainthm-GWP}. Indeed with
Propositions \ref{prop-const-priori-1} and \ref{prop-B21-lip2},  the existence part of Theorem \ref{mainthm-GWP}
follows exactly
along the same line to that of Theorem \ref{thm-GWS-MHD}. In particular, the system \eqref{1.2} with
$\mu(\rho)=1$ has a global solution $(\rho, \, u,\, B, \,\nabla\Pi)$ with $\rho\in C_{\rm w}([0,\infty); L^{\infty})$ and $(u, B) \in (C([0, +\infty); \dot{H}^{\frac{1}{2}})\cap L^4(\mathbb{R}^+; \dot{H}^{1}))^2,$ which satisfy \eqref{bdd-density-visc-1}, \eqref{est-const-basic-0} and \eqref{est-const-B21-0}.  Below let us focus on the uniqueness part of Theorem \ref{mainthm-GWP}, which
we shall use the Lagrangian approach as that in \cite{DM1}. Let $(\rho, \, u,\, B, \,\nabla\Pi)$ be the global solution of the system \eqref{1.2} obtained above.
Due to $\nabla{u}\in L^1_{loc}(\mathbb{R}^+; L^{\infty})$, we can define $\eta$ the  position of the fluid particle $x$ in $\mathbb{R}^3$ at time $t\in \mathbb{R}^+$ through
\begin{equation}\label{def-flowmap-1}
\begin{cases}
&\frac{d}{dt}\eta(t, x)=u(t, \eta(t, x)), \quad \forall\,(t, x) \in \mathbb{R}^+\times \mathbb{R}^3,\\
&\eta|_{t=0}=x, \quad  \forall\, x\in \mathbb{R}^3,
\end{cases}
\end{equation}
then the displacement $\xi(t, x)\eqdefa \eta(t, x)-x$ satisfies
\begin{equation}\label{def-flowmap-2}
\begin{cases}
&\frac{d}{dt}\xi(t, x)=u(t, x+\xi(t, x)), \quad \forall\,(t, x) \in \mathbb{R}^+\times \mathbb{R}^3,\\
&\xi|_{t=0}=0.
\end{cases}
\end{equation}
We define Lagrangian quantities as follows:
\begin{equation*} \label{S4eq1}
\begin{split}
&v(t, x)\eqdefa u(t, \eta(t, x)),\, q(t, x)\eqdefa \Pi(t, \eta(t, x)), \, b(t, x)\eqdefa B(t, \eta(t, x)),\\
&\mathfrak{J}(t, x)\eqdefa J(t, \eta(t, x)),\, \mathfrak{f}(t, x)\eqdefa \rho(t, \eta(t, x)),
\end{split}
\end{equation*}
where $J\eqdefa\curl{B}$ denotes the current density.

Let $D\eta$ be  Jacobian matrix of the flow map $\eta$
\begin{equation*}
\begin{split}
 &D\eta\eqdefa \left(
\begin{array}{lll}
 1+\partial_1\xi^1 &  \partial_2\xi^1& \partial_3\xi^1\\
 \partial_1\xi^2&1+\partial_2\xi^2  &\partial_3 \xi^2\\
  \partial_1\xi^3 &\partial_2\xi^3& 1+\partial_3\xi^3 \\
\end{array}\right) \andf \mathcal{A} \eqdefa  (D\eta)^{-T}.
\end{split}
\end{equation*}
Thanks to $\nabla\cdot{u}=0$ and $\partial_t\rho+u\cdot\nabla\rho=0$, we get $\nabla_{\mathcal{A}}\cdot v=0$, $\partial_t \mbox{det}(D\eta)=0$ and $\partial_t  \mathfrak{f} =0$, which implies that $\mbox{det}(D\eta)\equiv 1$ and $ \mathfrak{f}(t, x)\equiv\rho_0(x),$ and there holds
\begin{equation}\label{expre-a-1}
  \begin{split}
  & \mathcal{A}_{11}=(1+\partial_2\xi^2)(1+\partial_3\xi^3)-\partial_2\xi^3  \partial_3\xi^2,\, \ \mathcal{A}_{12}=-(\partial_1\xi^2+\partial_1\xi^2\partial_3\xi^3-\partial_1\xi^3  \partial_3\xi^2),\\
&\mathcal{A}_{13}=-(\partial_1\xi^3+\partial_1\xi^3\partial_2\xi^2-\partial_1\xi^2  \partial_2\xi^3),\, \ \mathcal{A}_{21}=-(\partial_2\xi^1+\partial_2\xi^1\partial_3\xi^3-\partial_2\xi^3  \partial_3\xi^1),\\
&\mathcal{A}_{22}=(1+\partial_1\xi^1)(1+\partial_3\xi^3)-\partial_1\xi^3  \partial_3\xi^1,\, \ \mathcal{A}_{23}=-(\partial_2\xi^3+\partial_2\xi^3\partial_1\xi^1-\partial_1\xi^3  \partial_2\xi^1),\\
&\mathcal{A}_{31}=-(\partial_3\xi^1+\partial_3\xi^1\partial_2\xi^2-\partial_2\xi^1  \partial_3\xi^2),\, \ \mathcal{A}_{32}=-(\partial_3\xi^2+\partial_3\xi^2\partial_1\xi^1-\partial_1\xi^2  \partial_3\xi^1),\\
&\mathcal{A}_{33}=(1+\partial_1\xi^1)(1+\partial_2\xi^2)-\partial_1\xi^2  \partial_2\xi^1.
  \end{split}
\end{equation}

It follows from \eqref{def-flowmap-1} and \eqref{def-flowmap-2} that
\begin{equation*}\label{flow-map-identity-1}
\mathcal{A}_{i}^k \partial_{k} \eta^j=\mathcal{A}_{k}^j \partial_{i} \eta^k=\delta_i^j,\quad\partial_{i} \eta^j=\delta_i^j+\partial_{i} \xi^j, \quad \mathcal{A}_{i}^j=\delta_i^j-\mathcal{A}_{i}^k \partial_k\xi^j.
\end{equation*}
Since $\mathcal{A}(D\eta)^T=\mathbb{I}$, by differentiating it with respect to $t$ and $x,$ one has
\begin{equation}\label{identity-Lagrangian-1}
\begin{split}
&\partial_t \mathcal{A}_{i}^j=-\mathcal{A}_{k}^j\mathcal{A}_{i}^{m} \partial_{m}v^k,\quad \partial_{\ell} \mathcal{A}_{i}^j=-\mathcal{A}_{k}^j\mathcal{A}_{i}^{m} \partial_{m}\partial_{\ell}\xi^k,
\end{split}
\end{equation}
where we used the fact $\partial_t\eta=v$ in the first equation in \eqref{identity-Lagrangian-1}.

Moreover, it is easy to verify the following Piola identity:
\begin{equation*}\label{identity-Piola}
\begin{split}
&\partial_j (\mbox{det}(D\eta) \mathcal{A}_{i}^j) =0 \quad \forall \,i = 1, 2, 3.
\end{split}
\end{equation*}
Here and in what follows, the subscript notation for vectors and tensors as well as the Einstein summation convention has been adopted unless otherwise specified.

In the  Lagrangian coordinates, we may introduce the differential operators with
their actions given by $(\nabla_{\mathcal{A}}f)_i=\mathcal{A}_i^j  \partial_jf$, $ \mathbb{D}_{\mathcal{A}} (v)=\nabla_{\mathcal{A}} v+(\nabla_{\mathcal{A}} v)^T$, $\Delta_{\mathcal{A}} f=\nabla_{\mathcal{A}}\cdot \nabla_{\mathcal{A}} f$, so that in
the Lagrangian coordinates,  the system \eqref{1.2} reads
\begin{equation}\label{Lagrangian-MHD-1}
\begin{cases}
  & \rho_0\partial_t{v} + \nabla_{\mathcal{A}}{q}-\grad_{\mathcal{A}} \cdot \nabla_{\mathcal{A}}{v}=b\cdot \nabla_{\mathcal{A}}{b}\quad \mbox{in}\quad  \mathbb{R}^+\times \mathbb{R}^3,\\
    &\partial_t{b}+\nabla_{\mathcal{A}}\wedge (\sigma(\rho_0)\nabla_{\mathcal{A}}\wedge b)=b\cdot \nabla_{\mathcal{A}} v,\\
    & \nabla_{\mathcal{A}} \cdot v=\nabla_{\mathcal{A}} \cdot b=0,\\
    &(v, b)|_{t=0}=(u_0, B_0).
     \end{cases}
\end{equation}

\begin{rmk}
\label{rmk-Lagr-form-1}
It follows from  $(D\eta)^T\mathcal{A}=\mathbb{I}$ that if $\nabla_{\mathcal{A}} g=f$,  there hold
\begin{equation*}
\begin{split}
 &\nabla{g}=(D\eta)^Tf,\quad \nabla_{\mathfrak{B}}g= \mathfrak{B}\nabla{g}=\mathfrak{B}(D\eta)^Tf.
\end{split}
\end{equation*}
\end{rmk}
\begin{lem}\label{lem-lagr-quant-est-1}
Let $u(t, x)$ be a solenoidal vector field so that $\int_0^T\|\nabla{u}\|_{L^{\infty}}\,dt\leq \frac{1}{16}$ for some $T>0$,
let the flow map $\eta$ and the displacement $\xi$ be defined respectively by \eqref{def-flowmap-1} and \eqref{def-flowmap-2}. Then
 for any Euler quantity $h(t, x)$ and  its corresponding  Lagrangian quantity $\widetilde{h}(t, x)=h(t, \eta(t, x))$ ($\forall\,(t, x) \in \mathbb{R}^+\times \mathbb{R}^3$), and for all  $p,\, q \in [1, +\infty]$, there are two positive constants $C_1$ and $C_2$ such that
\begin{equation}\label{est-Lagr-Euler-1}
\begin{split}
&\|{\widetilde{h}}\|_{L^q_T(L^p)}=\|{h}\|_{L^q_T(L^p)},\quad \|\partial_t{\widetilde{h}}\|_{L^q_T(L^p)}=\|D_t{h}\|_{L^q_T(L^p)},\\
& \|\mathcal{A}-\mathbb{I}\|_{L^{\infty}_T(L^{\infty})} \leq \frac{20}{3}\int_0^T\|\nabla{u}\|_{L^{\infty}}\,dt,\quad \|\nabla^2 v\|_{L^1_T(L^3)} \lesssim \|\nabla^2u\|_{L^1_T(L^3)},\\
&C_1\|\nabla{\widetilde{h}}\|_{L^q_T(L^p)}\leq \|\nabla_{\mathcal{A}}{\widetilde{h}}\|_{L^q_T(L^p)}\leq C_2\|\nabla{\widetilde{h}}\|_{L^q_T(L^p)},\\
&\|\nabla{\widetilde{h}}\|_{L^q_T(L^p)}\lesssim \|\nabla{h}\|_{L^q_T(L^p)},\quad\|\nabla\partial_t{\widetilde{h}}\|_{L^q_T(L^p)}\lesssim \|\nabla{D}_th\|_{L^q_T(L^p)},\\
&\|\nabla \wedge (\sigma(\rho_0)\,\nabla_{\mathcal{A}} \wedge \widetilde{h})\|_{L^q_T(L^p)}\lesssim \|\|\nabla \wedge (\sigma(\rho)\,\nabla \wedge h)\|_{L^q_T(L^p)}.
\end{split}
\end{equation}
\end{lem}
\begin{proof}
 Due to $\det (D\eta)=1$, we get, by using changes of variables, that
 \begin{equation*}\label{est-Lagr-Euler-2}
\begin{split}
&\int_{\mathbb{R}^3}|\widetilde{h}(t, x)|^p\,dx=\int_{\mathbb{R}^3}|h(t, \eta(t, x))|^p\,dx=\int_{\mathbb{R}^3}|h(t, x)|^p\,dx,
\end{split}
\end{equation*}
which gives
 \begin{equation*}\label{est-Lagr-Euler-3}
\begin{split}
&\|{\widetilde{h}}\|_{L^q_T(L^p)}=\|{h}\|_{L^q_T(L^p)}.
\end{split}
\end{equation*}
Along the same line, one has
 \begin{equation*}\label{est-Lagr-Euler-4}
\begin{split}
&\|\partial_t{\widetilde{h}}\|_{L^q_T(L^p)}=\|D_t{h}\|_{L^q_T(L^p)},\quad\|\nabla_{\mathcal{A}}{\widetilde{h}}\|_{L^q_T(L^p)}=\|\nabla{h}\|_{L^q_T(L^p)},\\
&\|\nabla_{\mathcal{A}}\partial_t{\widetilde{h}}\|_{L^q_T(L^p)}\lesssim \|\nabla{D}_th\|_{L^q_T(L^p)},\\
&\|\nabla_{\mathcal{A}} \wedge (\sigma(\rho_0)\,\nabla_{\mathcal{A}} \wedge \widetilde{h})\|_{L^q_T(L^p)}=\|\nabla \wedge (\sigma(\rho)\,\nabla \wedge h)\|_{L^q_T(L^p)}.
\end{split}
\end{equation*}
Due to \eqref{expre-a-1}, we have
 \begin{equation}\label{est-Lagr-Euler-5}
\begin{split}
&\|\mathcal{A}-\mathbb{I}\|_{L^{\infty}_T(L^{\infty})} \leq 2\|\nabla\xi(t)\|_{L^{\infty}_T(L^{\infty})}(1+\|\nabla\xi(t)\|_{L^{\infty}_T(L^{\infty})}),
\end{split}
\end{equation}
which along with $\int_0^T\|\nabla_{\mathcal{A}} v(t)\|_{L^{\infty}}\,dt=\int_0^T\|\nabla{u}(t)\|_{L^{\infty}}\,dt$ ensures that
 \begin{equation*}\label{est-Lagr-Euler-6}
\begin{split}\int_0^T\|\nabla v(t)\|_{L^{\infty}}\,dt&\leq \int_0^T\|\nabla_{\mathcal{A}} v(t)\|_{L^{\infty}}\,dt+\|\mathcal{A}-\mathbb{I}\|_{L^{\infty}_T(L^{\infty})}\int_0^t \|\nabla v(t)\|_{L^{\infty}}\,dt\\
&\leq \int_0^T\|\nabla{u}(t)\|_{L^{\infty}}\,dt+2\|\nabla\xi\|_{L^{\infty}_T(L^{\infty})}\bigl(1+\|\nabla\xi\|_{L^{\infty}_T(L^{\infty})}\bigr)\int_0^t \|\nabla v(t)\|_{L^{\infty}}\,dt.
\end{split}
\end{equation*}
We assume by the classical continuous argument that
 \begin{equation}\label{est-Lagr-Euler-7}
\begin{split}
&\|\nabla\xi\|_{L^{\infty}_T(L^{\infty})}\leq \frac{1}{4},
\end{split}
\end{equation}
then we obtain
 \begin{equation*}\label{est-Lagr-Euler-8}
\begin{split}
&\int_0^T\|\nabla v(t)\|_{L^{\infty}}\,dt \leq \frac{8}{3}\int_0^T\|\nabla{u}(t)\|_{L^{\infty}}\,dt\leq \frac{1}{6},
\end{split}
\end{equation*}
which together with the fact: $\|\nabla\xi\|_{L^{\infty}_T(L^{\infty})}\leq \int_0^T\|\nabla v(t)\|_{L^{\infty}}\,dt,$ ensures that
  \begin{equation*}\label{est-Lagr-Euler-9}
\begin{split}
&\|\nabla\xi(t)\|_{L^{\infty}_T(L^{\infty})}\leq \frac{1}{6}
\end{split}
\end{equation*}
and then \eqref{est-Lagr-Euler-7} holds.

Therefore thanks to \eqref{est-Lagr-Euler-7} and \eqref{est-Lagr-Euler-5}, we obtain
 \begin{equation}\label{est-Lagr-Euler-10}
\begin{split}
&\|\nabla\xi(t)\|_{L^{\infty}_T(L^{\infty})}\leq \frac{1}{4},\quad \int_0^T\|\nabla v(t)\|_{L^{\infty}}\,dt \leq \frac{1}{6},\\ &\|\mathcal{A}-\mathbb{I}\|_{L^{\infty}_T(L^{\infty})} \leq \frac{20}{3}\int_0^T\|\nabla{u}\|_{L^{\infty}}\,dt\leq \frac{5}{12}.
\end{split}
\end{equation}

Notice that
 \begin{equation*}\label{est-Lagr-Euler-11}
\begin{split}
&\|\nabla_{\mathcal{A}}{\widetilde{h}}\|_{L^q_T(L^p)}=\|\nabla{\widetilde{h}}+\nabla_{\mathcal{A}-\mathbb{I}}{\widetilde{h}}\|_{L^q_T(L^p)},
\end{split}
\end{equation*}
which implies that
 \begin{equation*}\label{est-Lagr-Euler-12}
\begin{split}
&\|\nabla{\widetilde{h}}\|_{L^q_T(L^p)}-\|\mathcal{A}-\mathbb{I}\|_{L^{\infty}_T(L^{\infty})}\|\nabla{\widetilde{h}}\|_{L^q_T(L^p)}\\
&\qquad\leq \|\nabla_{\mathcal{A}}{\widetilde{h}}\|_{L^q_T(L^p)}\leq \|\nabla{\widetilde{h}}\|_{L^q_T(L^p)}+\|\mathcal{A}-\mathbb{I}\|_{L^{\infty}_T(L^{\infty})}\|\nabla{\widetilde{h}}\|_{L^q_T(L^p)}.
\end{split}
\end{equation*}
We thus deduce from \eqref{est-Lagr-Euler-10} that
 \begin{equation*}\label{est-Lagr-Euler-13}
\begin{split}
C_1\|\nabla{\widetilde{h}}\|_{L^q_T(L^p)}\leq \|\nabla_{\mathcal{A}}{\widetilde{h}}\|_{L^q_T(L^p)}\leq C_2\|\nabla{\widetilde{h}}\|_{L^q_T(L^p)},
\end{split}
\end{equation*}
so that there hold
\begin{equation*}\label{est-Lagr-Euler-14}
\begin{split}
&\|\nabla\partial_t{\widetilde{h}}\|_{L^q_T(L^p)}\lesssim \|\nabla{D}_th\|_{L^q_T(L^p)},\\
&\|\nabla \wedge (\sigma(\rho_0)\,\nabla_{\mathcal{A}} \wedge \widetilde{h})\|_{L^q_T(L^p)}\lesssim \|\nabla \wedge (\sigma(\rho)\,\nabla \wedge h)\|_{L^q_T(L^p)}.
\end{split}
\end{equation*}

Finally, due to $\nabla^2 v=\nabla\nabla_{\mathcal{A}} v+\nabla\nabla_{\mathbb{I}-\mathcal{A}} v$, we have
\begin{equation}\label{est-Lagr-Euler-15}
\begin{split}
\|\nabla^2 v\|_{L^1_T(L^3)}
\leq &\|\nabla\nabla_{\mathcal{A}} v\|_{L^1_T(L^3)}+\|\mathbb{I}-\mathcal{A}\|_{L^{\infty}_T(L^{\infty})} \|\nabla^2v\|_{L^1_T(L^3)}\\
&+\|\nabla \mathcal{A}\|_{L^{\infty}_T(L^3)}\|\nabla{v}\|_{L^1_T(L^{\infty})}.
\end{split}
\end{equation}
While it follows from  \eqref{expre-a-1} and \eqref{est-Lagr-Euler-10} that
\begin{equation*}\label{est-Lagr-Euler-16}
\begin{split}
&\|\nabla \mathcal{A}\|_{L^{\infty}_T(L^3)}\leq  \|\nabla^2\xi\|_{L^{\infty}_T(L^3)} \bigl(2+4\|\nabla\xi\|_{L^{\infty}_T(L^{\infty})}\bigr)\leq  3\|\nabla^2 v\|_{L^1_T(L^3)}
\end{split}
\end{equation*}
which along with \eqref{est-Lagr-Euler-15}  and \eqref{est-Lagr-Euler-10} ensures that
\begin{equation*}\label{est-Lagr-Euler-17}
\begin{split}
&\|\nabla^2 v\|_{L^1_T(L^3)}\leq \|\nabla\nabla_{\mathcal{A}} v\|_{L^1_T(L^3)}+\frac{11}{12} \|\nabla^2v\|_{L^1_T(L^3)}.
\end{split}
\end{equation*}
We thus obtain
\begin{equation*}\label{est-Lagr-Euler-18}
\begin{split}
&\|\nabla^2 v\|_{L^1_T(L^3)}\leq 12\|\nabla\nabla_{\mathcal{A}} v\|_{L^1_T(L^3)}  \lesssim \|\nabla_{\mathcal{A}}\nabla_{\mathcal{A}} v\|_{L^1_T(L^3)}=\|\nabla^2u\|_{L^1_T(L^3)}.
\end{split}
\end{equation*}
This finishes the proof of Lemma \ref{lem-lagr-quant-est-1}.
\end{proof}

Let  $(v, \nabla{q}, b)$ and $(\bar{v}, \nabla\bar{q}, \bar{b})$ be  two solutions of the system \eqref{Lagrangian-MHD-1}.
We denote
$\delta {f}\eqdefa f-\bar{f}$. Then $(\d v, \nabla{\d q}, \d b)$ solves
\begin{equation}\label{diffdeltavb-MHD-1}
\begin{cases}
  & \rho_0\partial_t \delta{v} + \nabla_{\mathcal{A}}\,\delta{q}-\nabla_{\mathcal{A}} \cdot \nabla_{\mathcal{A}}\delta{v} =\delta\mathcal{F},\\
    &\partial_t \delta{b}+\nabla_{\mathcal{A}}\wedge(\sigma(\rho_0)\nabla_{\mathcal{A}}\wedge \delta{b})\\
    &\quad =\delta{H}-\nabla_{\mathcal{A}}\wedge(\sigma(\rho_0)\nabla_{\delta\mathcal{A}}\wedge \bar{b})-\nabla_{\delta\mathcal{A}}\wedge(\sigma(\rho_0)\nabla_{\bar{\mathcal{A}}}\wedge\bar{b}),\\
    & \nabla_{\mathcal{A}} \cdot \delta{v}+ \nabla_{\delta\mathcal{A}} \cdot \bar{v}=\nabla_{\mathcal{A}} \cdot \delta{b}+\nabla_{\delta\mathcal{A}} \cdot \bar{b}=0,\\
         &(\delta{v}, \delta{b})|_{t=0}=(0, 0),
     \end{cases}
\end{equation}
where $\delta\mathcal{F}=\delta{F}^{(1)}+\delta{F}^{(2)}$ with
\begin{equation*}\label{diffdeltavb-j-MHD-2}
\begin{split}
\delta{F}^{(1)}&\eqdefa \nabla_{\mathcal{A}}\cdot ( b\otimes {b})-\nabla_{\bar{\mathcal{A}}}\cdot (\bar{b}\otimes\bar{b})=b\cdot\nabla_{\mathcal{A}}\delta{b}+b\cdot\nabla_{\delta\mathcal{A}} \bar{b}+\delta{b}\cdot\nabla_{\bar{\mathcal{A}}} \bar{b}\\
&=\nabla_{\mathcal{A}}\cdot ( b\otimes \delta{b})+\nabla_{\mathcal{A}}\cdot ( \delta{b}\otimes \bar{b})+\nabla_{\delta{\mathcal{A}}}\cdot (\bar{b}\otimes \bar{b}),\\
\delta{F}^{(2)}&\eqdefa - \nabla_{\delta\mathcal{A}}\,\bar{q}+\nabla_{\mathcal{A}} \cdot \nabla_{\delta\mathcal{A}}\bar{v}+\nabla_{\delta\mathcal{A}} \cdot \nabla_{\bar{\mathcal{A}}}\bar{v},
     \end{split}
\end{equation*}
and
 \begin{equation}\label{diffdeltavb-j-MHD-4}
\begin{split}
\delta{H}&\eqdefa\nabla_{\mathcal{A}}\cdot ( b\otimes v)-\nabla_{\bar{\mathcal{A}}}\cdot (\bar{b}\otimes\bar{v})=b\cdot\nabla_{\mathcal{A}}\delta{v}+b\cdot\nabla_{\delta\mathcal{A}}\bar{v}+\delta{b}\cdot\nabla_{\bar{\mathcal{A}}}\bar{v}\\
&=\nabla_{\mathcal{A}}\cdot (b\otimes \delta{v})+\nabla_{\mathcal{A}}\cdot (\delta b\otimes \bar{v})+\nabla_{\delta{\mathcal{A}}}\cdot ( \bar{b}\otimes \bar{v}).
     \end{split}
\end{equation}

Alternatively, the $\delta{b}$ equation in \eqref{diffdeltavb-MHD-1} can also  be reformulated as
 \begin{equation}\label{diffb-Lagrangian-MHD-2}
\begin{split}
&\partial_t \delta{b}+\nabla_{\mathcal{A}}\wedge(\sigma(\rho_0) \delta\mathfrak{J})=\delta{H}-\nabla_{\delta\mathcal{A}}\wedge(\sigma(\rho_0)\nabla_{\bar{\mathcal{A}}}\wedge\bar{b}) \with\\
& \delta\mathfrak{J}\eqdefa \nabla_{\mathcal{A}}\wedge {b}-\nabla_{\bar{\mathcal{A}}}\wedge\bar{b}=\nabla_{\mathcal{A}}\wedge\delta{b}+ \nabla_{\delta\mathcal{A}}\wedge \bar{b}.
     \end{split}
\end{equation}

\begin{lem}\label{S4lem2}
{\sl Under the assumptions of Lemma \ref{lem-lagr-quant-est-1}, there exists a positive constant $c_7$ so that
 \begin{equation}\label{diffvb-Lagr-noj-H1-1}
\begin{split}
& \frac{d}{dt}E_{1}(t)+2c_{7} D_{1}(t) \lesssim \|(\nabla \delta{v}, \nabla\delta{b})\|_{L^2}^2\bigl(\|\partial_t\mathcal{A}\|_{L^{\infty}}+\|b\|_{L^{\infty}}^2+\|(\nabla\bar{b}, \nabla\bar{v}, \partial_t\mathcal{A})\|_{L^3}^2\bigr)\\
&+\|(b, \nabla\mathcal{A})\|_{L^3}^2\| \nabla^2\delta{v}\|_{L^2}^2+\|(\delta\mathfrak{J}, \nabla\delta{b})\|_{L^3}^2\|\partial_t\mathcal{A}\|_{L^3}+f_1(t)\,\|(\delta\mathfrak{J}, \nabla\delta{b})\|_{L^3}+\mathcal{R}_{1}+\mathcal{R}_{2},
     \end{split}
\end{equation}
where
 \begin{equation}\label{diffvb-Lagr-noj-H1-1aaa}
\begin{split}
E_{1}(t)\eqdefa &\|(\nabla_{\mathcal{A}}\delta{v}, \sqrt{\sigma(\rho_0)}\, \delta\mathfrak{J})\|_{L^2}^2-2\int_{\mathbb{R}^3} \sigma(\rho_0)(\nabla_{\bar{\mathcal{A}}}\wedge \bar{b} )  \cdot (\nabla_{\delta\mathcal{A}}\wedge \delta{b})\,dx\\
&+4C_5\|\delta\mathcal{A}\otimes\nabla\bar{b}\|_{L^2}^2,\\
D_{1}(t)\eqdefa& \|(\partial_t\delta{v}, \nabla\delta{q}, \nabla^2\delta{v}, \partial_t\delta{b})\|_{L^2}^2,\\
f_1(t)\eqdefa & \|\partial_t \bar{\mathcal{A}}\|_{L^{3}} \|\nabla\bar{b}\|_{L^6}\|\delta\mathcal{A}\|_{L^6} +\| \nabla\bar{b}\|_{L^6}\|\partial_t\delta\mathcal{A}\|_{L^2}+\|\delta\mathcal{A}\|_{L^6} \|\nabla\partial_t\bar{b}\|_{L^2},\\
 \mathcal{R}_{1}(t)\eqdefa& \|(\nabla{v}, \nabla\bar{v})\|_{L^{\infty}}^2(\|\nabla\delta\xi\|_{L^2}^2\|\nabla\bar{v}\|_{L^3}^2+\|\nabla\delta\mathcal{A}\|_{L^2}^2 +\|\delta\mathcal{A}\|_{L^6}^2 \|\nabla^2\bar{\xi}\|_{L^3}^2)\\
  & + \|\delta\mathcal{A}\|_{L^{3}}^2(\|(\partial_t\bar{v}, \nabla\bar{q}, \nabla\nabla_{\bar{\mathcal{A}}}\bar{v})\|_{L^6}^2+\|b\|_{L^{\infty}}^2 \|\nabla\bar{b}\|_{L^{6}}^2),\\
\mathcal{R}_{2}(t) \eqdefa& \|\delta\mathcal{A}\|_{L^6}^2(\|b\|_{L^3}^2  \|\nabla\bar{v}\|_{L^{\infty}}^2 +\|\nabla\bar{b}\|_{L^6}\|\partial_t\nabla\bar{b}\|_{L^2})+\|\nabla\bar{b}\|_{L^6}^2 \|\delta\mathcal{A}\|_{L^6}\|\partial_t\delta\mathcal{A}\|_{L^2}.
     \end{split}
\end{equation}}
\end{lem}

\begin{proof} We divide the proof of this lemma into the following two steps:

\no{\bf Step 1.} The energy estimate of $\d v.$

We first get, by
taking the $L^2$ inner product of the  $\d v$ equations in \eqref{diffdeltavb-MHD-1} with $\partial_t\delta{v},$ that
 \begin{equation*}\label{diffv-j-Lagrangian-MHD-H1-1}
\begin{split}
& \|\sqrt{\rho_0}\partial_t\delta{v}\|_{L^2}^2+\int_{\mathbb{R}^3} \nabla_{\mathcal{A}}\delta{v}:\nabla_{\mathcal{A}}\partial_t\delta{v}\,dx\\
&=\int_{\mathbb{R}^3} \delta{q}\,\nabla_{\mathcal{A}}\cdot\partial_t\delta{v}\,dx
+\int_{\mathbb{R}^3}\bigl(b\cdot\nabla_{\mathcal{A}}\delta{b}+b\cdot\nabla_{\delta\mathcal{A}}\bar{b}
+\delta{b}\cdot\nabla_{\bar{\mathcal{A}}}\bar{b}\bigr)\cdot\partial_t\delta{v}\,dx.
     \end{split}
\end{equation*}
Notice that
\begin{align*}
 \int_{\mathbb{R}^3} \nabla_{\mathcal{A}}\delta{v}:\nabla_{\mathcal{A}}\partial_t\delta{v}\,dx
 &=\frac{1}{2}\frac{d}{dt}\|\nabla_{\mathcal{A}}\delta{v}(t)\|_{L^2}^2-\int_{\mathbb{R}^3} \nabla_{\mathcal{A}}\delta{v}:\nabla_{\partial_t\mathcal{A}}\delta{v}\,dx,\\
\int_{\mathbb{R}^3} \delta{q}\,\nabla_{\mathcal{A}}\cdot\partial_t\delta{v}\,dx&=\int_{\mathbb{R}^3} \delta{q}\,\partial_t(\nabla_{\mathcal{A}}\cdot\delta{v})\,dx-\int_{\mathbb{R}^3} \delta{q}\,\nabla_{\partial_t\mathcal{A}}\cdot\delta{v}\,dx\\
&=-\int_{\mathbb{R}^3} \delta{q}\,\partial_t(\nabla_{\delta\mathcal{A}} \cdot \bar{v})\,dx+\int_{\mathbb{R}^3} \delta{v}\cdot \nabla_{\partial_t\mathcal{A}}\delta{q}\,dx,
\end{align*}
we obtain
\begin{equation}\label{diffv-j-Lagrangian-MHD-H1-5}
\begin{split}
& \|\sqrt{\rho_0}\partial_t\delta{v}\|_{L^2}^2+\frac{1}{2}\frac{d}{dt}\|\nabla_{\mathcal{A}}\delta{v}(t)\|_{L^2}^2=\int_{\mathbb{R}^3} \nabla_{\mathcal{A}}\delta{v}:\nabla_{\partial_t\mathcal{A}}\delta{v}\,dx\\
&\qquad+\int_{\mathbb{R}^3} \,(\bar{v}\cdot\nabla_{\partial_t\delta\mathcal{A}}\delta{q} +\partial_t\bar{v}\cdot\nabla_{\delta\mathcal{A}}\delta{q} )\,dx+\int_{\mathbb{R}^3} \delta{v}\cdot \nabla_{\partial_t\mathcal{A}}\delta{q}\,dx\\
&\qquad+\int_{\mathbb{R}^3}\bigl(b\cdot\nabla_{\mathcal{A}}\delta{b}
+b\cdot\nabla_{\delta\mathcal{A}}\bar{b}+\delta{b}\cdot\nabla_{\bar{\mathcal{A}}}\bar{b}\bigr)\cdot\partial_t\delta{v}\,dx.
     \end{split}
\end{equation}
It is easy to observe that
\begin{align*}
\bigl|\int_{\mathbb{R}^3}& \nabla_{\mathcal{A}}\delta{v}:\nabla_{\partial_t\mathcal{A}}\delta{v}\,dx\bigr|+\bigl|\int_{\mathbb{R}^3} \,\partial_t\bar{v}\cdot\nabla_{\delta\mathcal{A}}\delta{q} \,dx\bigr|\\
&+\bigl|\int_{\mathbb{R}^3}\bigl(b\cdot\nabla_{\mathcal{A}}\delta{b}+b\cdot\nabla_{\delta\mathcal{A}}\bar{b}
+\delta{b}\cdot\nabla_{\bar{\mathcal{A}}}\bar{b}\bigr)\cdot\partial_t\delta{v}\,dx\bigr|\\
\lesssim &\|\nabla_{\mathcal{A}} \delta{v}\|_{L^2}^2\|\partial_t\mathcal{A}\|_{L^{\infty}}+\|\delta\mathcal{A}\|_{L^{3}}\|\partial_t\bar{v}\|_{L^6}\|\nabla\delta{q} \|_{L^2}\\
&+\bigl(\|b\|_{L^{\infty}}\|\nabla_{\mathcal{A}}\delta{b}\|_{L^{2}}+\|b\|_{L^{\infty}}\|\delta\mathcal{A}\|_{L^3}\|\nabla\bar{b}\|_{L^6}
+\|\delta{b}\|_{L^{6}}\|\nabla_{\bar{\mathcal{A}}}\bar{b}\|_{L^{3}}\bigr)\|\partial_t\delta{v}\|_{L^{2}}.
     \end{align*}
While by using integration by parts, one has
 \begin{align*}
|\int_{\mathbb{R}^3} \bar{v}\cdot\nabla_{\partial_t\delta\mathcal{A}}\delta{q} \,dx|&=|\int_{\mathbb{R}^3}\delta{q} \,\nabla_{\partial_t\delta\mathcal{A}}\cdot \bar{v}\, dx|\lesssim \|\delta{q}\|_{L^6}\,\|\partial_t\delta\mathcal{A}\|_{L^2}\|\nabla\bar{v}\|_{L^3}\\
&\lesssim \|\nabla\delta{q}\|_{L^2}\bigl(\|\nabla\delta\xi\|_{L^2}\|\nabla{v}\|_{L^{\infty}}+\|\nabla\delta{v}\|_{L^2} \bigr)\|\nabla\bar{v}\|_{L^3}
  \end{align*}
  and
  \begin{align*}
& \bigl|\int_{\mathbb{R}^3} \delta{v}\cdot \nabla_{\partial_t\mathcal{A}}\delta{q}\,dx\bigr|\lesssim \|\delta{v}\|_{L^6}\,\|\partial_t\mathcal{A}\|_{L^3}\|\nabla\delta{q}\|_{L^2}|\lesssim \|\nabla\delta{v}\|_{L^2}\,\|\partial_t\mathcal{A}\|_{L^3}\|\nabla\delta{q}\|_{L^2}.
  \end{align*}
  By substituting the above estimates into \eqref{diffv-j-Lagrangian-MHD-H1-5},
 we arrive at
 \begin{align*}
 &\frac{d}{dt}\|\nabla_{\mathcal{A}}\delta{v}(t)\|_{L^2}^2+2c_{2}\|\partial_t\delta{v}\|_{L^2}^2\lesssim \|\nabla_{\mathcal{A}} \delta{v}\|_{L^2}^2\|\partial_t\mathcal{A}\|_{L^{\infty}}\\
&+\|(\partial_t\delta{v}, \nabla\delta{q})\|_{L^2}\|(\nabla\delta{v}, \nabla\delta{b})\|_{L^2}(\|b\|_{L^{\infty}}+\|(\nabla\bar{v}, \nabla\bar{b}, \partial_t\mathcal{A})\|_{L^3})\\
&+\|(\partial_t\delta{v}, \nabla\delta{q})\|_{L^2} \bigl(\|\delta\mathcal{A}\|_{L^3}(\|b\|_{L^{\infty}}\|\nabla\bar{b}\|_{L^6}+\|\partial_t\bar{v}\|_{L^6})+\|\nabla\delta\xi\|_{L^2}\|\nabla{v}\|_{L^{\infty}}\|\nabla\bar{v}\|_{L^3}\bigr).
\end{align*}
By applying Young's inequality, we obtain
 \begin{equation}\label{diffv-j-Lagrangian-MHD-H1-7}
\begin{split}
\frac{d}{dt}&\|\nabla_{\mathcal{A}}\delta{v}(t)\|_{L^2}^2+ 2c_{2}\|\partial_t\delta{v}\|_{L^2}^2\\
\leq &\varepsilon \|(\partial_t\delta{v}, \nabla\delta{q}) \|_{L^{2}}^2+C_{\varepsilon}\|(\nabla \delta{v}, \nabla\delta{b})\|_{L^2}^2(\|\partial_t\mathcal{A}\|_{L^{\infty}}+\|b\|_{L^{\infty}}^2+\|(\nabla\bar{v}, \nabla\bar{b}, \partial_t\mathcal{A})\|_{L^3}^2)\\
&+C_{\varepsilon}\bigl(\|\delta\mathcal{A}\|_{L^3}^2(\|b\|_{L^{\infty}}^2\|\nabla\bar{b}\|_{L^6}^2+\|\partial_t\bar{v}\|_{L^6}^2)
+\|\nabla\delta\xi\|_{L^2}^2\|\nabla{v}\|_{L^{\infty}}^2\|\nabla\bar{v}\|_{L^3}^2\bigr)
     \end{split}
\end{equation}
for any positive constant $\varepsilon$.

 In order to handle the estimate about $\|(\nabla^2\delta{v}, \nabla\delta{q}) \|_{L^{2}}$, we use the momentum equations in \eqref{diffdeltavb-MHD-1} to get
 \begin{equation*}\label{diffv-j-Lagrangian-MHD-H1-9}
\begin{split}
  &\bigl\|(\nabla_{\mathcal{A}}\,\delta{q}-\nabla_{\mathcal{A}} \cdot \nabla_{\mathcal{A}}\delta{v})\bigr\|_{L^2} \leq \|\rho_0\partial_t \delta{v}\|_{L^2} + \|\delta\mathcal{F}\|_{L^2}.
     \end{split}
\end{equation*}
While due to $ \nabla_{\mathcal{A}} \cdot \delta{v}+ \nabla_{\delta\mathcal{A}} \cdot \bar{v}=0$, one has
 \begin{align*}
  & \|\nabla_{\mathcal{A}}\,\delta{q}-\nabla_{\mathcal{A}} \cdot \nabla_{\mathcal{A}}\delta{v}\|_{L^2}^2\\
  &=\|\nabla_{\mathcal{A}}\,\delta{q}\|_{L^2}^2+\|\nabla_{\mathcal{A}} \cdot \nabla_{\mathcal{A}}\delta{v}\|_{L^2}^2-2\int_{\mathbb{R}^3} \nabla_{\mathcal{A}} \delta{q} :\nabla_{\mathcal{A}}(\nabla_{\mathcal{A}}\cdot\delta{v})\,dx\\
  &=\|\nabla_{\mathcal{A}}\,\delta{q}\|_{L^2}^2+\|\nabla_{\mathcal{A}} \cdot \nabla_{\mathcal{A}}\delta{v}\|_{L^2}^2+2\int_{\mathbb{R}^3} \nabla_{\mathcal{A}} \delta{q} :\nabla_{\mathcal{A}}(\nabla_{\delta\mathcal{A}} \cdot \bar{v})\,dx.
     \end{align*}
As a result, it comes out
 \begin{equation*}\label{diffv-j-Lagrangian-MHD-H1-11}
\begin{split}
  &\|\nabla_{\mathcal{A}}\,\delta{q}\|_{L^2}^2+\|\nabla_{\mathcal{A}} \cdot \nabla_{\mathcal{A}}\delta{v}\|_{L^2}^2\lesssim \|\nabla_{\mathcal{A}} \delta{q}\|_{L^2}\|\nabla_{\mathcal{A}}(\nabla_{\delta\mathcal{A}} \cdot \bar{v})\|_{L^2}+ \|\partial_t \delta{v}\|_{L^2}^2 + \|\delta\mathcal{F}\|_{L^2}^2.
     \end{split}
\end{equation*}
Applying Young's inequality yields
\begin{equation*}\label{diffv-j-Lagrangian-MHD-H1-13}
\begin{split}
&\|\nabla_{\mathcal{A}}\,\delta{q}\|_{L^2}^2+\|\nabla_{\mathcal{A}} \cdot \nabla_{\mathcal{A}}\delta{v}\|_{L^2}^2\lesssim  \|\nabla_{\mathcal{A}}(\nabla_{\delta\mathcal{A}} \cdot \bar{v})\|_{L^2}^2+\|\partial_t \delta{v}\|_{L^2}^2 + \|\delta\mathcal{F}\|_{L^2}^2.
     \end{split}
\end{equation*}
Notice that
 \begin{equation*}\label{diffv-j-Lagrangian-MHD-H1-14}
\begin{split}
&\|\nabla_{\mathcal{A}}(\nabla_{\delta\mathcal{A}} \cdot \bar{v})\|_{L^2}^2+ \|\delta\mathcal{F}\|_{L^2}^2\lesssim \|\nabla\delta\mathcal{A}  \nabla\bar{v}\|_{L^2}^2+\|\delta\mathcal{A} \nabla{D\bar{\eta}}^T \nabla\bar{v}\|_{L^2}^2+ \|b\|_{L^{\infty}}^2\|\nabla\delta{b}\|_{L^{2}}^2\\
&\qquad+\|\delta{b}\|_{L^{6}}^2\|\nabla\bar{b}\|_{L^{3}}^2 +\|b\|_{L^{\infty}}^2\|\delta\mathcal{A}\|_{L^3}^2\|\nabla\bar{b}\|_{L^{6}}^2+\|\delta\mathcal{A}\|_{L^{3}}^2\|(\nabla\bar{q}, \nabla\nabla_{\bar{\mathcal{A}}}\bar{v})\|_{L^6}^2.
     \end{split}
\end{equation*}
We thus obtain
\begin{equation*}\label{diffv-j-Lagrangian-MHD-H1-16}
\begin{split}
&\|\nabla\delta{q}\|_{L^2}^2+\|\nabla^2\delta{v}\|_{L^2}^2\lesssim \|\nabla\mathcal{A} \nabla\delta{v}\|_{L^2}^2 +\|\nabla\delta\mathcal{A}  \nabla\bar{v}\|_{L^2}^2+\|\delta\mathcal{A} \nabla{D\bar{\eta}}^T \nabla\bar{v}\|_{L^2}^2+\|\partial_t \delta{v}\|_{L^2}^2\\
  & +(\|b\|_{L^{\infty}}^2+\|\nabla\bar{b}\|_{L^{3}}^2)\|\nabla\delta{b}\|_{L^{2}}^2+\|b\|_{L^{\infty}}^2\|\delta\mathcal{A}\|_{L^3}^2\|\nabla\bar{b}\|_{L^{6}}^2
  +\|\delta\mathcal{A}\|_{L^{3}}^2\|(\nabla\bar{q}, \nabla\nabla_{\bar{\mathcal{A}}}\bar{v})\|_{L^6}^2,
     \end{split}
\end{equation*}
from which, we infer
 \begin{equation}\label{diffv-j-Lagrangian-MHD-H1-19}
\begin{split}
\|\nabla\delta{q}\|_{L^2}^2+\|\nabla^2\delta{v}\|_{L^2}^2
  \lesssim&\|\partial_t \delta{v}\|_{L^2}^2 +\|\nabla\mathcal{A}\|_{L^3}^2 \| \nabla^2\delta{v}\|_{L^2}^2+\|\delta\mathcal{A}\|_{L^6}^2 \|\nabla{D\bar{\eta}}\|_{L^3}^2 \|\nabla\bar{v}\|_{L^{\infty}}^2\\
  &+(\|b\|_{L^{\infty}}^2+\|\nabla\bar{b}\|_{L^{3}}^2)\|\nabla\delta{b}\|_{L^{2}}^2+\|\nabla\delta\mathcal{A}\|_{L^2}^2  \|\nabla\bar{v}\|_{L^{\infty}}^2\\
  &+ \|b\|_{L^{\infty}}^2\|\delta\mathcal{A}\|_{L^3}^2 \|\nabla\bar{b}\|_{L^{6}}^2
  +\|\delta\mathcal{A}\|_{L^{3}}^2\|(\nabla\bar{q}, \nabla\nabla_{\bar{\mathcal{A}}}\bar{v})\|_{L^6}^2.
     \end{split}
\end{equation}

By
combining \eqref{diffv-j-Lagrangian-MHD-H1-7} with \eqref{diffv-j-Lagrangian-MHD-H1-19}, we find
 \begin{equation}\label{diffv-j-Lagrangian-MHD-H1-20}
\begin{split}
\frac{d}{dt}&\|\nabla_{\mathcal{A}}\delta{v}(t)\|_{L^2}^2+ 2c_{3}\|(\partial_t\delta{v}, \nabla\delta{q}, \nabla^2\delta{v})\|_{L^2}^2\leq C_{2}\Bigl(\|\nabla\mathcal{A}\|_{L^3}^2 \| \nabla^2\delta{v}\|_{L^2}^2\\
&+ \|(\nabla \delta{v}, \nabla\delta{b})\|_{L^2}^2\bigl(\|\partial_t\mathcal{A}\|_{L^{\infty}}+\|b\|_{L^{\infty}}^2+\|(\nabla\bar{b}, \nabla\bar{v}, \partial_t\mathcal{A})\|_{L^3}^2\bigr)+\mathcal{R}_{1}\Bigr)
     \end{split}
\end{equation}
with $\mathcal{R}_{1}$ being given by \eqref {diffvb-Lagr-noj-H1-1aaa}.\\

\no{\bf Step 2.} The energy estimate of $\d b.$

 We first get, by taking the $L^2$ inner product of the equations \eqref{diffb-Lagrangian-MHD-2} with $\partial_t\delta{b},$ that
 \begin{equation}\label{diffb-j-Lagrangian-MHD-H1-1}
\begin{split}
&  \|\partial_t\delta{b}\|_{L^2}^2+\int_{\mathbb{R}^3}\sigma(\rho_0)\, \delta\mathfrak{J}\cdot \nabla_{\mathcal{A}}\wedge\partial_t\delta{b}\,dx=\int_{\mathbb{R}^3}\bigl(\delta{H}+\nabla_{\delta\mathcal{A}}\wedge(\sigma(\rho_0)\nabla_{\bar{\mathcal{A}}}\wedge \bar{b})\bigr) \cdot \partial_t\delta{b}\,dx.
     \end{split}
\end{equation}
Yet observing that
\begin{align*}
&  \int_{\mathbb{R}^3}\sigma(\rho_0)\delta\mathfrak{J}\cdot \nabla_{\mathcal{A}}\wedge\partial_t\delta{b}\,dx \\
&=\int_{\mathbb{R}^3}\sigma(\rho_0) \delta\mathfrak{J}\cdot\partial_t(\nabla_{\mathcal{A}}\wedge\delta{b})\,dx-\int_{\mathbb{R}^3}\sigma(\rho_0) \delta\mathfrak{J}\cdot (\nabla_{\partial_t\mathcal{A}}\wedge \delta{b})\,dx\\
&=\int_{\mathbb{R}^3}\sigma(\rho_0)\, \delta\mathfrak{J}\cdot\partial_t(\delta\mathfrak{J}- \nabla_{\delta\mathcal{A}}\wedge \bar{b})\,dx-\int_{\mathbb{R}^3}\sigma(\rho_0)\, \delta\mathfrak{J}\cdot (\nabla_{\partial_t\mathcal{A}}\wedge \delta{b})\,dx\\
&=\frac{1}{2}\frac{d}{dt}\|\sqrt{\sigma(\rho_0)}\, \delta\mathfrak{J}(t)\|_{L^2}^2-\int_{\mathbb{R}^3}\sigma(\rho_0)\, \delta\mathfrak{J}\cdot( \partial_t(\nabla_{\delta\mathcal{A}}\wedge \bar{b})+\nabla_{\partial_t\mathcal{A}}\wedge \delta{b})\,dx,
     \end{align*}
from which, we infer
 \begin{equation*}\label{diffb-j-Lagrangian-MHD-H1-3}
\begin{split}
&  \int_{\mathbb{R}^3}\sigma(\rho_0)\, \delta\mathfrak{J}\, \nabla_{\mathcal{A}}\wedge \partial_t\delta{b}\,dx \\
&=\frac{1}{2}\frac{d}{dt}\|\sqrt{\sigma(\rho_0)} \delta\mathfrak{J}(t)\|_{L^2}^2-\int_{\mathbb{R}^3}\sigma(\rho_0) \delta\mathfrak{J}\cdot (\nabla_{\partial_t\delta\mathcal{A}}\wedge \bar{b})\,dx\\
&\qquad +\int_{\mathbb{R}^3}\nabla_{\delta\mathcal{A}}\wedge(\sigma(\rho_0) \delta\mathfrak{J})\cdot \partial_t\bar{b}\,dx-\int_{\mathbb{R}^3}\sigma(\rho_0) \delta\mathfrak{J}\cdot (\nabla_{\partial_t\mathcal{A}}\wedge \delta{b})\,dx,
     \end{split}
\end{equation*}
Similarly, one has
 \begin{equation*}\label{diffb-j-Lagr-H1-4}
\begin{split}
&  \int_{\mathbb{R}^3} \nabla_{\delta\mathcal{A}}\wedge(\sigma(\rho_0)\nabla_{\bar{\mathcal{A}}}\wedge \bar{b})  \cdot \partial_t\delta{b}\,dx=\int_{\mathbb{R}^3} \sigma(\rho_0)(\nabla_{\bar{\mathcal{A}}}\wedge \bar{b})   \cdot (\nabla_{\delta\mathcal{A}}\wedge\partial_t\delta{b})\,dx\\
&=\frac{d}{dt}\int_{\mathbb{R}^3} \sigma(\rho_0)(\nabla_{\bar{\mathcal{A}}}\wedge \bar{b} )  \cdot (\nabla_{\delta\mathcal{A}}\wedge \delta{b})(t)\,dx\\
&-\int_{\mathbb{R}^3} \sigma(\rho_0)\partial_t(\nabla_{\bar{\mathcal{A}}}\wedge \bar{b})   \cdot \nabla_{\delta\mathcal{A}}\wedge \delta{b}\,dx-\int_{\mathbb{R}^3} \sigma(\rho_0)(\nabla_{\bar{\mathcal{A}}}\wedge \bar{b})   \cdot \nabla_{\partial_t\delta\mathcal{A}}\wedge \delta{b}\,dx.
     \end{split}
\end{equation*}
By substituting the above equalities into \eqref{diffb-j-Lagrangian-MHD-H1-1},
We obtain
 \begin{equation*}\label{diffb-j-Lagrangian-MHD-H1-5}
\begin{split}
 \frac{d}{dt}&\Bigl(\frac{1}{2}\|\sqrt{\sigma(\rho_0)}\, \delta\mathfrak{J}(t)\|_{L^2}^2-\int_{\mathbb{R}^3} \sigma(\rho_0)(\nabla_{\bar{\mathcal{A}}}\wedge \bar{b} )  \cdot (\nabla_{\delta\mathcal{A}}\wedge \delta{b})(t)\,dx\Bigr)+\|\partial_t\delta{b}\|_{L^2}^2\\
    =&\int_{\mathbb{R}^3}\sigma(\rho_0)\, \delta\mathfrak{J}\cdot\bigl( \nabla_{\partial_t\delta\mathcal{A}}\wedge \bar{b}-\nabla_{\delta\mathcal{A}}\wedge\partial_t\bar{b}+\nabla_{\partial_t\mathcal{A}}\wedge \delta{b}\bigr)\,dx+\int_{\mathbb{R}^3} \delta{H}  \cdot \partial_t\delta{b}\,dx\\
&-\int_{\mathbb{R}^3} \sigma(\rho_0)\partial_t(\nabla_{\bar{\mathcal{A}}}\wedge \bar{b})   \cdot \nabla_{\delta\mathcal{A}}\wedge \delta{b}\,dx-\int_{\mathbb{R}^3} \sigma(\rho_0)(\nabla_{\bar{\mathcal{A}}}\wedge \bar{b})   \cdot \nabla_{\partial_t\delta\mathcal{A}}\wedge \delta{b}\,dx,
     \end{split}
\end{equation*}
which results in
 \begin{equation}\label{diffb-j-Lagrangian-MHD-H1-6}
\begin{split}
\frac{d}{dt}&\Bigl(\|\sqrt{\sigma(\rho_0)}\, \delta\mathfrak{J}(t)\|_{L^2}^2-2\int_{\mathbb{R}^3} \sigma(\rho_0)(\nabla_{\bar{\mathcal{A}}}\wedge \bar{b} )  \cdot (\nabla_{\delta\mathcal{A}}\wedge \delta{b})(t)\,dx\Bigr)+2\|\partial_t\delta{b}\|_{L^2}^2\\
    \lesssim &\|\delta\mathfrak{J}\|_{L^3}(\| \nabla\bar{b}\|_{L^6}\|\partial_t\delta\mathcal{A}\|_{L^2}+\|\delta\mathcal{A}\|_{L^6} \|\nabla\partial_t\bar{b}\|_{L^2})+\|\partial_t\mathcal{A}\|_{L^3} \|\delta\mathfrak{J}\|_{L^3}\|\nabla\delta{b}\|_{L^3}\\
    &+\|\delta{H} \|_{L^2} \| \partial_t\delta{b}\|_{L^2}+(\|\partial_t(\nabla_{\bar{\mathcal{A}}}\wedge \bar{b})\|_{L^2}\|\delta\mathcal{A}\|_{L^6} +\|\nabla\bar{b}\|_{L^6}  \|\partial_t\delta\mathcal{A}\|_{L^2}) \|\nabla\delta{b}\|_{L^3}.
     \end{split}
\end{equation}
It is obvious to observe from \eqref{diffdeltavb-j-MHD-4} that
 \begin{equation*}\label{est-H-L2-1}
\begin{split}
\|\delta{H}\|_{L^2} \lesssim &\|b\|_{L^3} \|\nabla\delta{v}\|_{L^6} +\|b\|_{L^3} \|\delta\mathcal{A}\|_{L^6} \|\nabla\bar{v}\|_{L^{\infty}} +\|\delta{b}\|_{L^6} \|\nabla\bar{v}\|_{L^3}\\
\lesssim &\|b\|_{L^3} \|\nabla^2\delta{v}\|_{L^2} +\|b\|_{L^3} \|\delta\mathcal{A}\|_{L^6} \|\nabla\bar{v}\|_{L^{\infty}} + \|\nabla\delta{b}\|_{L^2}\|\nabla\bar{v}\|_{L^3}.
     \end{split}
\end{equation*}
Then we get, by using Young's inequality to
\eqref{diffb-j-Lagrangian-MHD-H1-6},  that
\begin{equation}\label{diffb-Lagr-MHD-H1-10}
\begin{split}
 \frac{d}{dt}\Bigl(&\|\sqrt{\sigma(\rho_0)}\, \delta\mathfrak{J}(t)\|_{L^2}^2-2\int_{\mathbb{R}^3} \sigma(\rho_0)(\nabla_{\bar{\mathcal{A}}}\wedge \bar{b} )  \cdot (\nabla_{\delta\mathcal{A}}\wedge \delta{b})(t)\,dx\Bigr)+\|\partial_t\delta{b}\|_{L^2}^2\\
\lesssim & \|b\|_{L^3}^2 \|\nabla^2\delta{v}\|_{L^2}^2 +\|b\|_{L^3}^2 \|\delta\mathcal{A}\|_{L^6}^2 \|\nabla\bar{v}\|_{L^{\infty}}^2 + \|\nabla\delta{b}\|_{L^2}^2\|\nabla\bar{v}\|_{L^3}^2\\
    &+(\|\delta\mathfrak{J}\|_{L^3}+\|\nabla\delta{b}\|_{L^3})\,f_1(t)
    +(\|\delta\mathfrak{J}\|_{L^3}^2+\|\nabla\delta{b}\|_{L^3}^2)\|\partial_t\mathcal{A}\|_{L^3}
     \end{split}
\end{equation}
with $f_1(t)$ being given by \eqref{diffvb-Lagr-noj-H1-1aaa}.

Notice that
 \begin{equation*}\label{H1-energy-cont-1}
\begin{split}
|\int_{\mathbb{R}^3} \sigma(\rho_0)(\nabla_{\bar{\mathcal{A}}}\wedge \bar{b} )  \cdot (\nabla_{\delta\mathcal{A}}\wedge \delta{b})\,dx| \lesssim \|\delta\mathcal{A}\otimes\nabla\bar{b}\|_{L^2}\|\nabla\delta{b}|\|_{L^2}
     \end{split}
\end{equation*}
and
 \begin{equation*}\label{J-nabla-b-1}
\begin{split}
&\|\nabla\delta{b}\|_{L^2}\lesssim \|\nabla_{\mathcal{A}}\wedge\delta{b}\|_{L^2}+\|\nabla_{\mathcal{A}}\cdot\delta{b}\|_{L^2}\lesssim \|\delta\mathfrak{J}\|_{L^2}+ \|\delta\mathcal{A}\otimes\nabla\bar{b}\|_{L^2},
     \end{split}
\end{equation*}
we find
 \begin{equation*}\label{H1-energy-cont-2}
\begin{split}
2\bigl|\int_{\mathbb{R}^3} \sigma(\rho_0)(\nabla_{\bar{\mathcal{A}}}\wedge \bar{b} )  \cdot (\nabla_{\delta\mathcal{A}}\wedge \delta{b})\,dx\bigr| & \leq\frac{1}{4}\|\sqrt{\sigma(\rho_0)}\, \delta\mathfrak{J}\|_{L^2}^2+C_5\|\delta\mathcal{A}\otimes\nabla\bar{b}\|_{L^2}^2,
     \end{split}
\end{equation*}
which implies
 \begin{equation}\label{H1-energy-cont-3}
\begin{split}
&\frac{1}{2}\|\sqrt{\sigma(\rho_0)}\, \delta\mathfrak{J}\|_{L^2}^2+2C_5\|\delta\mathcal{A}\otimes\nabla\bar{b}\|_{L^2}^2\\
&\leq \|\sqrt{\sigma(\rho_0)}\, \delta\mathfrak{J}\|_{L^2}^2-2\int_{\mathbb{R}^3} \sigma(\rho_0)(\nabla_{\bar{\mathcal{A}}}\wedge \bar{b} )  \cdot (\nabla_{\delta\mathcal{A}}\wedge \delta{b})\,dx+4C_5\|\delta\mathcal{A}\otimes\nabla\bar{b}\|_{L^2}^2\\
& \leq\frac{3}{2}\|\sqrt{\sigma(\rho_0)}\, \delta\mathfrak{J}\|_{L^2}^2+6C_5\|\delta\mathcal{A}\otimes\nabla\bar{b}\|_{L^2}^2.
     \end{split}
\end{equation}
Thanks to   \eqref{diffb-Lagr-MHD-H1-10}, \eqref{H1-energy-cont-3} and the fact:
\begin{equation*}\label{diffb-j-Lagrangian-MHD-H1-13bbb}
\begin{split}
 \frac{d}{dt}\|\delta\mathcal{A}\otimes\nabla\bar{b}(t)\|_{L^2}^2&\lesssim \int_{\mathbb{R}^3}(|\partial_t\delta\mathcal{A}\otimes\nabla\bar{b}||\delta\mathcal{A}\otimes\nabla\bar{b}|
+|\delta\mathcal{A}\otimes\partial_t\nabla\bar{b}||\delta\mathcal{A}\otimes\nabla\bar{b}|)\,dx\\
&\lesssim \|\nabla\bar{b}\|_{L^6}^2 \|\delta\mathcal{A}\|_{L^6}\|\partial_t\delta\mathcal{A}\|_{L^2}
+\|\delta\mathcal{A}\|_{L^6}^2 \|\nabla\bar{b}\|_{L^6}\|\partial_t\nabla\bar{b}\|_{L^2},
     \end{split}
\end{equation*}
we deduce  that
\begin{equation*}\label{diffb-Lagr-MHD-H1-100}
\begin{split}
\frac{d}{dt}\Bigl(&\|\sqrt{\sigma(\rho_0)}\, \delta\mathfrak{J}(t)\|_{L^2}^2-2\int_{\mathbb{R}^3} \sigma(\rho_0)(\nabla_{\bar{\mathcal{A}}}\wedge \bar{b} )  \cdot (\nabla_{\delta\mathcal{A}}\wedge \delta{b})(t)\,dx\\
&+4C_5\|\delta\mathcal{A}\otimes\nabla\bar{b}(t)\|_{L^2}^2\Bigr)+\|\partial_t\delta{b}\|_{L^2}^2\\
\lesssim &\|b\|_{L^3}^2 \|\nabla^2\delta{v}\|_{L^2}^2 + \|\nabla\delta{b}\|_{L^2}^2\|\nabla\bar{v}\|_{L^3}^2+\bigl(\|\delta\mathfrak{J}\|_{L^3}^2+\|\nabla\delta{b}\|_{L^3}^2\bigr)\|\partial_t\mathcal{A}\|_{L^3}\\
    &+\bigl(\|\delta\mathfrak{J}\|_{L^3}+\|\nabla\delta{b}\|_{L^3}\bigr)\,f_1(t)
    +\mathcal{R}_{2}
     \end{split}
\end{equation*}
with $\mathcal{R}_{2}$ being given by \eqref{diffvb-Lagr-noj-H1-1aaa}.

By summing up \eqref{diffb-Lagr-MHD-H1-10} and \eqref{diffv-j-Lagrangian-MHD-H1-20}, we obtain \eqref{diffvb-Lagr-noj-H1-1}.
This completes the proof of Lemma \ref{S4lem2}. \end{proof}

Now we are in a position to present  the proof of the uniqueness part of Theorem \ref{mainthm-GWP}.

\begin{proof}[Proof of the uniqueness part of Theorem \ref{mainthm-GWP}]
Let us first deal with the estimate of  $\|\delta\mathfrak{J}\|_{L^3}$. Observing from \eqref{diffb-Lagrangian-MHD-2} that
 \begin{equation*}
\begin{split}
&\nabla_{\mathcal{A}}\wedge(\nabla_{\mathcal{A}}\wedge \delta\mathfrak{J})=\nabla_{\mathcal{A}}\wedge(-\nabla_{\mathcal{A}}\wedge((\sigma(\rho_0)-1) \delta\mathfrak{J})-\partial_t \delta{b}+\delta{H}-\nabla_{\delta\mathcal{A}}\wedge(\sigma(\rho_0)\nabla_{\bar{\mathcal{A}}}\wedge\bar{b})),
     \end{split}
\end{equation*}
we write
 \begin{equation*}\label{expre-J-1}
\begin{split}
&-\Delta_{\mathcal{A}}\delta\mathfrak{J}=-\nabla_{\mathcal{A}}(\nabla_{\mathcal{A}}\cdot \delta\mathfrak{J})\\
&\quad+\nabla_{\mathcal{A}}\wedge\bigg(-\nabla_{\mathcal{A}}\wedge((\sigma(\rho_0)-1) \delta\mathfrak{J})-\partial_t \delta{b}+\delta {H}-\nabla_{\delta\mathcal{A}}\wedge(\sigma(\rho_0)\nabla_{\bar{\mathcal{A}}}\wedge\bar{b})\bigg).
     \end{split}
\end{equation*}
While it follows from the second equation of
\eqref{diffb-Lagrangian-MHD-2} that $$\nabla_{\mathcal{A}}\cdot \delta\mathfrak{J}=\nabla_{\mathcal{A}}\cdot (\nabla_{\mathcal{A}}\wedge \delta{b}+ \nabla_{\delta\mathcal{A}}\wedge \bar{b})
=\nabla_{\mathcal{A}}\cdot (\nabla_{\delta\mathcal{A}}\wedge \bar{b}),$$ so that we have
 \begin{equation*}
\begin{split}
&-\Delta_{\mathcal{A}}\delta\mathfrak{J}=-\nabla_{\mathcal{A}}(\nabla_{\mathcal{A}}\cdot (\nabla_{\delta\mathcal{A}}\wedge \bar{b}))\\
&\quad+\nabla_{\mathcal{A}}\wedge\bigg(-\nabla_{\mathcal{A}}\wedge((\sigma(\rho_0)-1) \delta\mathfrak{J})-\partial_t \delta{b}+\delta {H}-\nabla_{\delta\mathcal{A}}\wedge(\sigma(\rho_0)\nabla_{\bar{\mathcal{A}}}\wedge\bar{b})\bigg),
     \end{split}
\end{equation*}
from which, we infer
 \begin{equation}\label{expre-J-3}
\begin{split}
\|\delta\mathfrak{J}\|_{L^3}\lesssim &\|\nabla_{\delta\mathcal{A}}\wedge \bar{b}\|_{L^3}+\|(\sigma(\rho_0)-1) \delta\mathfrak{J}\|_{L^3}+\|(-\Delta_{\mathcal{A}})^{-\frac{1}{2}}\partial_t \delta{b}\|_{L^3}\\
&+\|(-\Delta_{\mathcal{A}})^{-\frac{1}{2}}\delta {H}\|_{L^3}+\|(-\Delta_{\mathcal{A}})^{-\frac{1}{2}}\nabla_{\delta\mathcal{A}}\wedge(\sigma(\rho_0)\nabla_{\bar{\mathcal{A}}}\wedge\bar{b}))\|_{L^3}\\
\lesssim &\|\delta\mathcal{A}\|_{L^6}\|\nabla\wedge \bar{b}\|_{L^6}+\|(\sigma(\rho_0)-1)\|_{L^{\infty}} \|\delta\mathfrak{J}\|_{L^3}+\|(-\Delta_{\mathcal{A}})^{-\frac{1}{2}}\partial_t \delta{b}\|_{L^2}^{\frac{1}{2}}\|\partial_t \delta{b}\|_{L^2}^{\frac{1}{2}}\\
&+\|(-\Delta_{\mathcal{A}})^{-\frac{1}{2}}\delta {H}\|_{L^3}+\|(-\Delta_{\mathcal{A}})^{-\frac{1}{2}}\nabla_{\delta\mathcal{A}}\wedge(\sigma(\rho_0)\nabla_{\bar{\mathcal{A}}}\wedge\bar{b}))\|_{L^3}.
     \end{split}
\end{equation}
 It is easy to observe from \eqref{diffdeltavb-j-MHD-4} and the fact:  $\|(-\Delta_{\mathcal{A}})^{-\frac{1}{2}}\nabla\|_{L^p} \lesssim 1$ ($\forall\,p \in (1, \infty)$), that
 \begin{equation*}
\begin{split}
\|(-\Delta_{\mathcal{A}})^{-\frac{1}{2}}\delta {H}\|_{L^3}&\lesssim \|b\otimes \delta{v}\|_{L^3}+\|\delta b\otimes \bar{v}\|_{L^3}+\|\delta{\mathcal{A}}\otimes ( \bar{b}\otimes \bar{v})\|_{L^3}\\
&\lesssim \|b\|_{L^6} \|\delta{v}\|_{L^6}+\|\delta b\|_{L^6}\|\bar{v}\|_{L^6}+\|\delta{\mathcal{A}}\|_{L^6}\|\bar{b}\|_{L^6}\|\bar{v}\|_{L^{\infty}},
     \end{split}
\end{equation*}
and
 \begin{equation*}
\begin{split}
&\|(-\Delta_{\mathcal{A}})^{-\frac{1}{2}}\nabla_{\delta\mathcal{A}}\wedge(\sigma(\rho_0)\nabla_{\bar{\mathcal{A}}}\wedge\bar{b}))\|_{L^3}\lesssim \|\delta\mathcal{A}\otimes(\sigma(\rho_0)\nabla_{\bar{\mathcal{A}}}\wedge\bar{b})\|_{L^3}\lesssim \|\delta\mathcal{A}\|_{L^6}\|\nabla\bar{b}\|_{L^6}.
     \end{split}
\end{equation*}
 While due to the $\delta{b}$ equations of \eqref{diffdeltavb-MHD-1}, one has
 \begin{align*}
\|(-\Delta_{\mathcal{A}})^{-\frac{1}{2}}\partial_t \delta{b}\|_{L^2} \lesssim & \|(-\Delta_{\mathcal{A}})^{-\frac{1}{2}}\nabla_{\mathcal{A}}\wedge(\sigma(\rho_0) \delta\mathfrak{J})\|_{L^2}\\
&+\|(-\Delta_{\mathcal{A}})^{-\frac{1}{2}}\delta {H}\|_{L^2}+\|(-\Delta_{\mathcal{A}})^{-\frac{1}{2}}\nabla_{\delta\mathcal{A}}\wedge(\sigma(\rho_0)\nabla_{\bar{\mathcal{A}}}\wedge\bar{b})\|_{L^2}\\
\lesssim& \|\delta\mathfrak{J}\|_{L^2}+\|b\otimes \delta{v}\|_{L^2}+\|\delta b\otimes \bar{v}\|_{L^2}+\|\delta{\mathcal{A}}\otimes ( \bar{b}\otimes \bar{v})\|_{L^2}\\
&+\|\delta\mathcal{A}\otimes( \nabla_{\bar{\mathcal{A}}}\wedge\bar{b})\|_{L^2},
\end{align*}
which implies
 \begin{equation*}\label{expre-J-6}
\begin{split}
\|(-\Delta_{\mathcal{A}})^{-\frac{1}{2}}\partial_t \delta{b}\|_{L^2} \lesssim &\|\delta\mathfrak{J}\|_{L^2}+\|b\|_{L^3}\|\delta{v}\|_{L^6}+\|\delta b\|_{L^6} \|\bar{v}\|_{L^3}\\
&+\|\delta{\mathcal{A}}\|_{L^3}(\|\bar{b}\|_{L^6}\|\bar{v}\|_{L^\infty}+\|\nabla\bar{b}\|_{L^6}),
     \end{split}
\end{equation*}
By substituting the above estimates into \eqref{expre-J-3} and using the smallness of $\|(\sigma(\rho_0)-1)\|_{L^{\infty}}$, we obtain
 \begin{equation}\label{J-j-L3-2}
\begin{split}
\|\delta\mathfrak{J}&\|_{L^3}^2\lesssim  (1+\|(\bar{v}, b)\|_{L^3} )\|(\delta\mathfrak{J}, \nabla\delta{v}, \nabla\delta{b})\|_{L^2} \|\partial_t \delta{b}\|_{L^2} +\|(\bar{v}, b)\|_{L^6}^2 \|(\nabla\delta{v}, \nabla\delta b)\|_{L^2}^2\\
&+\|\delta{\mathcal{A}}\|_{L^3} (\|\bar{b}\|_{L^6} \|\bar{v}\|_{L^\infty}
+\|\nabla\bar{b}\|_{L^6} )\|\partial_t \delta{b}\|_{L^2} +\|\delta\mathcal{A}\|_{L^6}^2(\|\nabla\bar{b}\|_{L^6}^2+\|\bar{b}\|_{L^6}^2\|\bar{v}\|_{L^{\infty}}^2).
     \end{split}
\end{equation}

While it follows  from \eqref{diffb-Lagrangian-MHD-2} that
 \begin{equation*}\label{J-j-L3-3}
\begin{split}
\|\nabla\delta{b}\|_{L^3}&\lesssim \|\nabla_{\mathcal{A}}\wedge\delta{b}\|_{L^3}+\|\nabla_{\mathcal{A}}\cdot\delta{b}\|_{L^3}\lesssim \|\delta\mathfrak{J}\|_{L^3}+ \|\nabla_{\delta\mathcal{A}}\wedge\bar{b}\|_{L^3}\\
&\lesssim \|\delta\mathfrak{J}\|_{L^3}+ \|\delta\mathcal{A}\|_{L^6}\|\nabla\bar{b}\|_{L^6}
     \end{split}
\end{equation*}
and
 \begin{equation*}\label{J-j-L3-4}
\begin{split}
&\|\delta\mathfrak{J}\|_{L^3}\lesssim \|\nabla\delta{b}\|_{L^3}+ \|\delta\mathcal{A}\|_{L^6}\|\nabla\bar{b}\|_{L^6},
     \end{split}
\end{equation*}
which together with \eqref{J-j-L3-2} ensures that
 \begin{equation}\label{J-j-L3-6}
\begin{split}
&\|(\delta\mathfrak{J}, \nabla\delta{b})\|_{L^3}^2\lesssim  \|(\bar{v}, b)\|_{L^6}^2 \|(\nabla\delta{v}, \nabla\delta b)\|_{L^2}^2+\|(\nabla\delta{v}, \nabla\delta{b}))\|_{L^2} \|\partial_t \delta{b}\|_{L^2}\\
&\qquad+ (\|\bar{b}\|_{L^6} \|\bar{v}\|_{L^\infty}
+\|\nabla\bar{b}\|_{L^6} )\|\delta{\mathcal{A}}\|_{L^3}\|\partial_t \delta{b}\|_{L^2}+\|\delta\mathcal{A}\|_{L^6}^2(\|\nabla\bar{b}\|_{L^6}^2+\|\bar{b}\|_{L^6}^2\|\bar{v}\|_{L^{\infty}}^2).
     \end{split}
\end{equation}

By inserting \eqref{J-j-L3-6} into \eqref{diffvb-Lagr-noj-H1-1} and using Young's inequality, we achieve
 \begin{equation}\label{diffvb-Lagr-noj-H1-22}
\begin{split}
\frac{d}{dt}E_{1}(t)&+\frac{3}{2}c_{7} D_{1}(t) \leq C_7 \Bigl(\|(b, \nabla\mathcal{A})\|_{L^3}^2\| \nabla^2\delta{v}\|_{L^2}^2+f_3(t)\|(\nabla \delta{v}, \nabla\delta{b})\|_{L^2}^2\\
&+f_{1}(t)\|(\nabla\delta{v}, \nabla\delta{b}))\|_{L^2}\|(\bar{v}, b)\|_{L^6}+f_{1}(t)\|(\nabla\delta{v}, \nabla\delta{b}))\|_{L^2}^{\frac{1}{2}} \|\partial_t \delta{b}\|_{L^2}^{\frac{1}{2}}\\
&+ f_{1}(t)\bigl(\|\bar{b}\|_{L^6}^{\frac{1}{2}} \|\bar{v}\|_{L^\infty}^{\frac{1}{2}}
+\|\nabla\bar{b}\|_{L^6}^{\frac{1}{2}} \bigr)\|\delta{\mathcal{A}}\|_{L^3}^{\frac{1}{2}}\|\partial_t \delta{b}\|_{L^2}^{\frac{1}{2}}+\mathcal{R}_{1}+\mathcal{R}_{2}+\mathcal{R}_{3}\Bigr),
     \end{split}
\end{equation}
where
 \begin{equation*}\label{diffvbv-Lagr-H1-remain-1}
\begin{split}
&f_3(t)\eqdefa \|\partial_t\mathcal{A}\|_{L^{\infty}}+\|b\|_{L^{\infty}}^2+\|(\nabla\bar{b}, \nabla\bar{v}, \partial_t\mathcal{A})\|_{L^3}^2+\|(\bar{v}, b)\|_{L^6}^2\|\partial_t\mathcal{A}\|_{L^3},\\
& \mathcal{R}_{3}\eqdefa f_{1}(t)\|\delta\mathcal{A}\|_{L^6}(\|\nabla\bar{b}\|_{L^6}+\|\bar{b}\|_{L^6}\|\bar{v}\|_{L^{\infty}})\\
&\qquad + \bigl(\|\bar{b}\|_{L^6}^2 \|\bar{v}\|_{L^\infty}^2
+\|\nabla\bar{b}\|_{L^6}^2 \bigr)\bigl(\|\delta{\mathcal{A}}\|_{L^3}^2\|\partial_t\mathcal{A}\|_{L^3}^2+\|\delta\mathcal{A}\|_{L^6}^2\|\partial_t\mathcal{A}\|_{L^3}\bigr).
     \end{split}
\end{equation*}
Since  $C_7 \|b\|_{L^{\infty}_t(L^3)}^2 \leq C_8\|(u_0, B_0)\|_{\dot{H}^{\frac{1}{2}}}^2\leq \frac{1}{2} c_7$
if $\|(u_0, B_0)\|_{\dot{H}^{\frac{1}{2}}}$ small enough, from which and \eqref{diffvb-Lagr-noj-H1-22}, we deduce that
 \begin{equation}\label{diffvb-Lagr-noj-H1-23}
\begin{split}
\frac{d}{dt}E_{1}(t)&+ c_{7} D_{1}(t) \leq C_7 \Bigl(\|\nabla\mathcal{A}\|_{L^3}^2\| \nabla^2\delta{v}\|_{L^2}^2+f_3(t)\|(\nabla \delta{v}, \nabla\delta{b})\|_{L^2}^2\\
&+f_{1}(t)\|(\nabla\delta{v}, \nabla\delta{b}))\|_{L^2}\|(\bar{v}, b)\|_{L^6}+f_{1}(t)\|(\nabla\delta{v}, \nabla\delta{b}))\|_{L^2}^{\frac{1}{2}} \|\partial_t \delta{b}\|_{L^2}^{\frac{1}{2}}\\
&+ f_{1}(t)\bigl(\|\bar{b}\|_{L^6}^{\frac{1}{2}} \|\bar{v}\|_{L^\infty}^{\frac{1}{2}}
+\|\nabla\bar{b}\|_{L^6}^{\frac{1}{2}}\bigr)\|\delta{\mathcal{A}}\|_{L^3}^{\frac{1}{2}}\|\partial_t \delta{b}\|_{L^2}^{\frac{1}{2}}+\mathcal{R}_{1}+\mathcal{R}_{2}+\mathcal{R}_{3}\Bigr).
     \end{split}
\end{equation}
Thanks to \eqref{H1-energy-cont-3}, we know that, there is a positive constant $C$ satisfying
\begin{equation*}\label{equiv-E1-1}
\begin{split}
& C^{-1}\|(\nabla\delta{v}, \delta\mathfrak{J}, \nabla\delta{b}, \delta\mathcal{A}\otimes\nabla\bar{b})(t)\|_{L^2}^2 \leq E_{1}(t) \leq C  \|(\nabla\delta{v}, \delta\mathfrak{J}, \nabla\delta{b}, \delta\mathcal{A}\otimes\nabla\bar{b})(t)\|_{L^2}^2
     \end{split}
\end{equation*}
for any $t \in [0, T^{*}[$.

Let
\begin{equation*}\label{def-diff-energy-total-1}
\begin{split}
& \mathfrak{E}(t)\eqdefa \|(\nabla\delta{v}, \delta\mathfrak{J}, \nabla\delta{b}, \delta\mathcal{A}\otimes\nabla\bar{b})\|_{L^{\infty}_t(L^2)}^2+\|(\partial_t\delta{v}, \nabla\delta{q}, \nabla^2\delta{v}, \partial_t\delta{b})\|_{L^2_t(L^2)}^2,
     \end{split}
\end{equation*}
we then get by integrating \eqref{diffvb-Lagr-noj-H1-23} over $[0,T]$ that
 \begin{equation}\label{est-E-unique-1}
\begin{split}
\mathfrak{E}(T)\leq & C  \Bigl(\|\nabla\mathcal{A}\|_{L^{\infty}_T(L^3)}^2\| \nabla^2\delta{v}\|_{L^2_T(L^2)}^2+\|f_3(t)\|_{L^1_T}\|(\nabla \delta{v}, \nabla\delta{b})\|_{L^{\infty}_T(L^2)}^2\\
&+\sum_{i=1}^3\|\mathcal{R}_{i}\|_{L^1_T}\Bigr)+C \|f_{1}\|_{L^{\frac{4}{3}}_T}\Bigl(\|(\nabla \delta{v}, \nabla\delta{b})\|_{L^{\infty}_T(L^2)}\,\|(\bar{v}, b)\|_{L^4_T(L^6)}\\
&+\|(\nabla \delta{v}, \nabla\delta{b})\|_{L^{\infty}_T(L^2)}^{\frac{1}{2}} \|\partial_t \delta{b}\|_{L^2_T(L^2)}^{\frac{1}{2}} +\bigl(\|t^{\frac{1}{4}}\bar{b}\|_{L^{\infty}_T(L^6)}^{\frac{1}{2}} \|t^{\frac{1}{2}}\bar{v}\|_{L^{\infty}_T(L^\infty)}^{\frac{1}{2}}\\
&+\|t^{\frac{3}{4}}\nabla\bar{b}\|_{L^6}^{\frac{1}{2}} \bigr)\|t^{-\frac{3}{4}} \delta\mathcal{A} \|_{L^{\infty}_T(L^3)} ^{\frac{1}{2}}\|\partial_t \delta{b}\|_{L^2_T(L^2)}^{\frac{1}{2}}\Bigr).
     \end{split}
\end{equation}
Thanks to \eqref{est-Lagr-Euler-1}, we have
 \begin{equation*}\label{est-A-int-1}
\begin{split}
&\|\nabla\mathcal{A}\|_{L^3}\lesssim \int_0^t\|\nabla^2 v(\tau)\|_{L^{3}}\,d\tau,\, \|\delta\mathcal{A}(t)\|_{L^6}\lesssim\|\nabla\delta\mathcal{A}(t)\|_{L^2}\lesssim
\int_0^t\|\nabla^2\delta{v}(\tau)\|_{L^2}\,d\tau,\\
&\|(\delta\mathcal{A}(t),\,\nabla\delta\xi(t))\|_{L^2}\lesssim  \int_0^t\|\nabla \delta{v}(\tau)\|_{L^2}\,d\tau, \quad\|\delta\mathcal{A}\|_{L^3}\lesssim  \int_0^t\|\nabla \delta{v}(\tau)\|_{L^3}\,d\tau,
     \end{split}
\end{equation*}
from which and Hardy's inequality ($\|t^{\alpha-1}\int_0^tf(\tau)d\tau\|_{L^p_T}\lesssim \|t^{\alpha} f(t)\|_{L^p_T}$ ($\forall \,\alpha <1-\frac{1}{p}$, \, $p\in (1, +\infty)$), we infer
 \begin{equation*}\label{est-A-int-2}
\begin{split}
& \|\nabla\mathcal{A}\|_{L^{\infty}_T(L^3)}
\lesssim \int_0^T\|\nabla^2{v}\|_{L^3}\,dt\lesssim\|(u_0, B_0)\|_{\dot{B}^{\frac{1}{2}}_{2, 1}},\, \|t^{-1} (\delta\mathcal{A},\,\nabla\delta\xi)\|_{L^{\infty}_T(L^2)}
\lesssim  \|\nabla \delta{v}(\tau)\|_{L^{\infty}_T(L^2)},\\
&\|t^{-\frac{1}{2}}\delta\mathcal{A}\|_{L^{\infty}_T(L^6)}
\lesssim \|t^{-\frac{1}{2}}\nabla\delta\mathcal{A}\|_{L^{\infty}_T(L^2)}
\lesssim \|t^{-\frac{1}{2}}\int_0^t\|\nabla^2\delta{v}(\tau)\|_{L^2}\,d\tau\|_{L^{\infty}_T}\lesssim
\|\nabla^2\delta{v}\|_{L^2_T(L^2)},\\
&\|t^{-\frac{3}{4}} \delta\mathcal{A} \|_{L^{\infty}_T(L^3)}
\lesssim  \|t^{-\frac{3}{4}}\int_0^t\|\nabla \delta{v}(\tau)\|_{L^3}\,d\tau\|_{L^{\infty}_T}\lesssim \|\nabla \delta{v}\|_{L^{4}_T(L^3)}
\end{split}
\end{equation*}
and
\begin{equation*}\label{est-A-int-3}
\begin{split}
&\|t^{-1}\delta\mathcal{A}\|_{L^2_T(L^6)}
\lesssim
\|\nabla^2\delta{v}\|_{L^2_T(L^2)}, \quad \|t^{-\frac{3}{4}}\delta\mathcal{A}\|_{L^{4}_T(L^6)}\lesssim
\|t^{\frac{1}{4}}\nabla^2\delta{v}\|_{L^4_T(L^2)}.
\end{split}
\end{equation*}
As a result, it comes out
 \begin{equation*}
\begin{split}
&\|t^{-1}\delta\mathcal{A}\|_{L^2_T(L^6)}+\|t^{-\frac{3}{4}}\delta\mathcal{A}\|_{L^{4}_T(L^6)}+
\|t^{-\frac{1}{2}}\delta\mathcal{A}\|_{L^{\infty}_T(L^6)}\\
&\qquad +\|t^{-1} (\delta\mathcal{A},\,\nabla\delta\xi)\|_{L^{\infty}_T(L^2)}+\|t^{-\frac{3}{4}} \delta\mathcal{A} \|_{L^{\infty}_T(L^3)} \lesssim (\mathfrak{E}(T))^{\frac{1}{2}}.
\end{split}
\end{equation*}
Thanks to the fact that
 \begin{equation*}
 \begin{split}
\int_0^T \| \nabla\bar{b}\|_{L^6}^{\frac{4}{3}}\,dt&=\int_0^T t^{-\frac{2}{3}}\,\|t^{\frac{1}{2}}\nabla\bar{b}\|_{L^6}^{\frac{4}{3}}\,dt\lesssim \|t^{-\frac{2}{3}}\|_{L^{\frac{3}{2}, \infty}} \|t^{\frac{1}{2}}\nabla\bar{b}\|_{L^{4, \frac{4}{3}}_T(L^6)}^{\frac{4}{3}}\leq C \|(u_0, B_0)\|_{\dot{B}^{\frac{1}{2}}_{2, 1}},
     \end{split}
 \end{equation*}
 we get
  \begin{equation*}
\begin{split}
\|f_1\|_{L^{\frac{4}{3}}_T} \leq & C\Bigl(\|\nabla\bar{v}\|_{L^2_T(L^{3})}\|t^{\frac{1}{2}}\nabla\bar{b}\|_{L^4_T(L^6)}\|t^{-\frac{1}{2}}\delta\mathcal{A}\|_{L^{\infty}_T(L^6)}+ \| \nabla\bar{b}\|_{L^{\frac{4}{3}}_T(L^6)} \|\nabla\delta{v}\|_{L^{\infty}_T(L^2)}  \\
& + \|t^{\frac{1}{2}}\nabla\bar{b}\|_{L^4_T(L^6)} \|t^{-1}\nabla\delta\xi\|_{L^{\infty}_T(L^2)} \|t^{\frac{1}{2}}\nabla{v}\|_{L^2_T(L^{\infty})}\\
&+\|t^{-\frac{3}{4}}\delta\mathcal{A}\|_{L^{4}_T(L^6)} \|t^{\frac{3}{4}}\nabla\partial_t\bar{b}\|_{L^2_T(L^2)}\Bigr)\leq L(T)\,\mathfrak{E}(T)^{\frac{1}{2}},
     \end{split}
\end{equation*}
here and in what follows, the continuous function $L(T)$ satisfies that $L(T)\rightarrow 0$ as $T\rightarrow 0^{+}$.

Along the same line, we obtain
\begin{equation*}\label{est-f-1-16}
\begin{split}
\|f_3(t)\|_{L^{1}_T} \leq C \bigl(&\|\nabla{v}\|_{L^1_T(L^{\infty})}+\|b\|_{L^2_T(L^\infty)}^2 \\
&+\|(\nabla\bar{b}, \nabla{b}, \nabla\bar{v}, \nabla{v})\|_{L^2_T(L^3)}^2(1+\|(\bar{v}, b)\|_{L^{\infty}_T(L^3)})\bigr)\leq L(T),
     \end{split}
\end{equation*}
and
\begin{align*}
 \|\mathcal{R}_{1}\|_{L^{1}_T} \leq& C   \|t^{\frac{1}{2}}\nabla({v}, \bar{v})\|_{L^2_T(L^{\infty})}^2 \bigl(\|t^{-1}\nabla\delta\xi\|_{L^{\infty}_T(L^2)}^2 \|t^{\frac{1}{2}}\nabla\bar{v}\|_{L^{\infty}_T(L^3)}^2\\
 &\qquad \qquad \qquad \qquad \qquad
 +\|t^{-\frac{1}{2}}\nabla\delta\mathcal{A}\|_{L^{\infty}_T(L^2)}^2 +\|t^{-\frac{1}{2}}\delta\mathcal{A}\|_{L^{\infty}_T(L^6)}^2 \|\nabla^2\bar{\xi}\|_{L^{\infty}_T(L^3)}^2\bigr)\\
  & + C\|t^{-\frac{3}{4}}\delta\mathcal{A}\|_{L^{\infty}_T(L^{3})}^2 \bigl(\|t^{\frac{3}{4}}(\partial_t\bar{v}, \nabla\bar{q}, \nabla\nabla_{\bar{\mathcal{A}}}\bar{v})\|_{L^2_T(L^6)}^2 + \|t^{\frac{1}{2}}b\|_{L^{\infty}_T(L^{\infty})}^2 \|t^{\frac{1}{4}}\nabla\bar{b}\|_{L^2_T(L^{6})}^2\bigr)\\
  \leq&  L(T)\,\mathfrak{E}(T),
     \end{align*}
\begin{align*}
\|\mathcal{R}_{2}\|_{L^{1}_T} \leq &C  \|t^{-\frac{1}{2}}\delta\mathcal{A}\|_{L^{\infty}_T(L^6)}^2 \bigl(\|b\|_{L^{\infty}_T(L^3)}^2  \|t^{\frac{1}{2}}\nabla\bar{v}\|_{L^2_T(L^{\infty})}^2 +\|t^{\frac{1}{4}}\nabla\bar{b}\|_{L^{2}_T(L^6)} \|t^{\frac{3}{4}}\nabla\partial_t\bar{b}\|_{L^{2}_T(L^2)}\bigr) \\
&+C \|t^{\frac{1}{4}}\nabla\bar{b}\|_{L^2_T(L^6)}^2  \|t^{-\frac{1}{2}}\delta\mathcal{A}\|_{L^{\infty}_T(L^6)} \|\partial_t\delta\mathcal{A}\|_{L^{\infty}_T(L^2)}\leq L(T)\,\mathfrak{E}(T),
      \end{align*}
and
 \begin{align*}
\|\mathcal{R}_{3}\|_{L^{1}_T} \leq &C \|f_{1}(t)\|_{L^{\frac{4}{3}}_T}\|t^{-\frac{3}{4}}\delta\mathcal{A}\|_{L^4_T(L^6)}) \bigl(\|t^{\frac{3}{4}}\nabla\bar{b}\|_{L^{\infty}_TL^6)}+\|t^{\frac{1}{4}}\bar{b}\|_{L^{\infty}_T(L^6)}
\|t^{\frac{1}{2}}\bar{v}\|_{L^{\infty}_T(L^{\infty})}\bigr)\\
&+C \bigl(\|t^{\frac{3}{4}}\nabla\bar{b}\|_{L^{\infty}_TL^6)}^2+\|t^{\frac{1}{4}}\bar{b}\|_{L^{\infty}_T(L^6)}^2
\|t^{\frac{1}{2}}\bar{v}\|_{L^{\infty}_T(L^{\infty})}^2\bigr)\\
&\,\times(\|t^{-\frac{3}{4}}\delta{\mathcal{A}}\|_{L^{\infty}_T(L^3)}^2 \|\nabla{v}\|_{L^2_T(L^3)}^2 + \|t^{-\frac{3}{4}}\delta\mathcal{A}\|_{L^4_T(L^6)}^2  \|\nabla{v}\|_{L^2_T(L^3)})
\leq  L(T)\,\mathfrak{E}(T).
    \end{align*}
By substituting the above estimates into
 \eqref{est-E-unique-1}, we obtain
 \begin{equation*}
\begin{split}
&\mathfrak{E}(T) \leq L(T)\,\mathfrak{E}(T).
     \end{split}
\end{equation*}
Taking $T_0>0$ to be so small that $L(T_0)\leq \frac{1}{2}$, we obtain
  \begin{equation*}
\begin{split}
&\mathfrak{E}(T_0) \leq \frac{1}{2}\mathfrak{E}(T_0),
     \end{split}
\end{equation*}
which implies that $\mathfrak{E}(T_0)=0$.

Therefore, we obtain $\delta{v}(t) = \delta{b}(t) =\nabla \delta\Pi(t)\equiv 0$ for any $t \in [0, T_0]$. The uniqueness of such strong solutions on the whole time interval $[0, +\infty)$ then follows by a bootstrap argument.
This ends the proof of Theorem \ref{mainthm-GWP}. \end{proof}

\appendix
\setcounter{equation}{0}
\section{Tool box on Littlewood-Paley theory and Lorentz spaces}\label{appendix-sect-5}

The proof  of the main results in this paper  requires Littlewood-Paley
theory. For the convenience of readers, we briefly explain how it may be built in the
case $x\in\R^3$ (see e.g. \cite{BCD}).

Let  $\varphi(\tau)$ be a smooth function such that
\begin{align*}
&\Supp \varphi \subset \Bigl\{\tau \in \R\,: \, \frac34 <
\tau < \frac83 \Bigr\}\quad\mbox{and}\quad \forall
 \tau>0\,,\ \sum_{j\in\Z}\varphi(2^{-j}\tau)=1.
\end{align*}
we define the homogeneous  dyadic operators as follows:
for $u\in{\mathcal S}_h'$ and any $ j\in\Z,$
\begin{equation}\label{LP-decom-sum-1}
  \begin{aligned}
&\dot\Delta_ju\eqdefa\varphi(2^{-j}|{D}|)u=\cF^{-1}(\varphi(2^{-j}|\xi|)\hat{u})\ \ \hspace{1cm}\mbox{and}
\hspace{1cm}
\dot S_j u\eqdefa\sum_{\ell \leq j-1}\dot\Delta_{\ell}u.
\end{aligned}
\end{equation}
The dyadic operator satisfies the
property of almost orthogonality:
\begin{equation*}
\begin{split}
&\dot\Delta_k\dot\Delta_j u\equiv 0
\quad\mbox{if}\quad\vert k-j\vert\geq 2
\quad\mbox{and}\quad\dot\Delta_k(\dot S_{j-1}u\dot\Delta_j u)
\equiv 0\quad\mbox{if}\quad\vert k-j\vert\geq 5.
\end{split}
\end{equation*}

\begin{defi}[see {\cite[Subsection 2.3]{BCD}}] \label{def2.1}
Let $(p,q)\in[1,+\infty]^2,$ $s\in\R$ and $u\in{\mathcal
S}_h'(\R^3)$, we set
$$
\|u\|_{{\dot{B}}^s_{p,q}}\eqdefa \big\|2^{js}\|{\dot\Delta_j u}\|_{L^{p}}\big\|_{\ell^{q}(\Z)}.
$$

$\bullet$ For $s<\frac{3}{p}$ (or $s=\frac{3}{p}$ if $q=1$), we define $
\dot{B}^s_{p,q}(\R^3)\eqdefa \big\{u\in{\mathcal S}_h'(\R^3)\;\big|\; \|
u\|_{{\dot{B}^s_{p,q}}}<\infty\big\}.$

$\bullet$ If $k\in\N$ and $\frac{3}{p}+k\leq s<\frac{3}{p}+k+1$ (or
$s=\frac{3}{p}+k+1$ if $r=1$), then $\dot{B}^s_{p,q}(\R^3)$ is defined as
the set of distributions $u\in{\mathcal S}_h'(\R^3)$ such that
$\partial^\beta u\in \dot B^{s-k}_{p,q}(\R^3)$ whenever $|\beta|=k.$
\end{defi}

We also recall Bernstein's inequality from  \cite{BCD}:

\begin{lem}\label{lem2.1}
{\sl Let $\mathcal{B}\eqdefa \{
\xi\in\R^3,\ |\xi|\leq\frac{4}{3}\}$ be a ball   and $\mathcal{C}\eqdefa \{
\xi\in\R^3,\frac{3}{4}\leq|\xi|\leq\frac{8}{3}\}$ a ring.
 A constant $C$ exists so that for any positive real number $\lambda,$ any nonnegative
integer $k,$ any smooth homogeneous function $\sigma$ of degree $m$,
any couple of real numbers $(a, \; b)$ with $ b \geq a \geq 1$, and any function $u$ in $L^a$,
there hold
\begin{equation}
\begin{split}
&\Supp \hat{u} \subset \lambda \mathcal{B} \Rightarrow
\sup_{|\alpha|=k} \|\pa^{\alpha} u\|_{L^{b}} \leq  C^{k+1}
\lambda^{k+ d\left(\frac{1}{a}-\frac{1}{b} \right)}\|u\|_{L^{a}},\\
& \Supp \hat{u} \subset \lambda \mathcal{C} \Rightarrow
C^{-1-k}\lambda^{ k}\|u\|_{L^{a}}\leq
\sup_{|\alpha|=k}\|\partial^{\alpha} u\|_{L^{a}}\leq
C^{1+k}\lambda^{ k}\|u\|_{L^{a}},\\
& \Supp \hat{u} \subset \lambda \mathcal{C} \Rightarrow \|\sigma(D)
u\|_{L^{b}}\leq C_{\sigma, m} \lambda^{ m+d\left(\frac{1}{a}-\frac{1}{b}
\right)}\|u\|_{L^{a}}, \end{split}\label{2.1}
\end{equation}}
with $\sigma(D)
u\eqdefa\mathcal{F}^{-1}(\sigma\,\hat{u})$.
\end{lem}

In  order to obtain a better description of the regularizing effect
of the transport-diffusion equation, we shall use Chemin-Lerner type
norm from
\cite{CL}.
\begin{defi}\label{def2.2}
{\sl Let $s\in\R$,
$r,\lambda, p\in [1,+\infty]$ and $T>0,$
 we define
\begin{equation*}
\|u\|_{\widetilde L^\lambda_T(\dot B^s_{p,r})}\eqdefa\Big\|2^{js}\|\dot\Delta_j
u\|_{L^\lambda_T(L^{p})}\Big\|_{\ell ^{r}(\mathbb{Z})}.
\end{equation*}
}
\end{defi}

Before introducing Lorentz space, we begin by recalling the rearrangement  of a function. For a measurable function $f,$ we define its non-increasing rearrangement (see \cite{G14} for instance)  by
 $f^{*}:\R_+\to \R_+$ via
$$
f^{*}(\lambda)\eqdefa \inf\big\{s\geq0;\,
\big|\{x:\ |f(x)|>s\}\ \big|\leq\lambda\big\},
$$
where $\big|\{\ x\in\R^3:\,|f(x)|>s\ \}\big|$ denotes the Lebesgue measure of the set $\{\ x\in\R^3:\,|f(x)|>s\ \}.$

\begin{defi} (Lorentz spaces)\label{espace_lorentz}
Let $f$ a mesurable function and
$1\leq p,q\leq\infty.$
Then $f$ belongs to the Lorentz space $L^{p,q}$ if
\begin{displaymath}
\|f\|_{L^{p,q}}\overset{def}{=}
\begin{cases}
            \Big( \int^\infty_0(t^{1\over p}f^*(t))^q{dt\over t}\Big)
            ^{1\over q}<\infty&
\text{if $q<\infty$}\\
             \displaystyle\sup_{t>0}\bigl(t^{1\over p}f^*(t)\bigr)<\infty &\text
             {if $q=\infty$}.
        \end{cases}
\end{displaymath}
\end{defi}
Alternatively, we can also define the Lorentz spaces by the real interpolation, as the interpolation between the Lebesgue spaces~:
$$
L^{p,q}\eqdefa (L^{p_0},L^{p_1})_{(\theta,q)},
$$
with $1\le p_0<p<p_1\le\infty,$ $0<\theta<1$ satisfying
${1\over p}={1-\theta\over p_0}+{\theta\over p_1}$ and $1\leq q\leq\infty,$ also $f\in L^{p,q}$ if the following quantity
$$
\|f\|_{L^{p,q}}\eqdefa
\Big(\int_0^\infty\big(t^{-\theta}K(t,f)\big)^q
{dt\over t}\Big)^{1\over q}
$$
is finite with
$$
K(f,t)\eqdefa\displaystyle\inf_{f=f_0+f_1}
\big\{\ \|f_0\|_{L^{p_0}}+t\|f_1\|_{L^{p_1}}\;\,\big|
\;f_0\in L^{p_0},\,f_1\in L^{p_1}\ \big\}.
$$

The Lorentz spaces verify the following properties
(see  \cite{lema,ON} for more details)~:

\begin{prop}\label{Neil}
{\sl Let $f\in L^{p_1,q_1},$ $g\in L^{p_2,q_2}$ and
$1\leq p,q,p_j,q_j\leq\infty,$ for $1\leq j\leq2.$
\vspace{0,5cm}

\begin{itemize}
\item[(1)] If $1<p<\infty$ and $1\le q\le\infty,$ then
$$
\|fg\|_{L^{p,q}}
\lesssim
\|f\|_{L^{p,q}}\|g\|_{L^{\infty}}.
$$

\item[(2)]
If ${1\over p}={1\over p_1}+{1\over p_2}$ and
${1\over q}={1\over q_1}+{1\over q_2},$ then
$$
\|fg\|_{L^{p,q}}
\lesssim
\|f\|_{L^{p_1,q_1}}\|g\|_{L^{p_2,q_2}}.
$$

\item[(3)]
For $1\leq p\leq\infty$ and $1\leq q_1\leq q_2\leq\infty,$ we have
$$
L^{p,q_1}\hookrightarrow L^{p,q_2}
\hspace{1cm}\mbox{and}\hspace{1cm}L^{p,p}=L^p.
$$
\end{itemize}}
\end{prop}



\noindent {\bf Acknowledgments.}
 G. Gui is supported in part by National Natural Science Foundation of China under Grants 12371211 and 12126359.
 P. Zhang is partially  supported by National Key R$\&$D Program of China under grant 2021YFA1000800 and by National Natural Science Foundation of China under Grants 12421001, 12494542 and 12288201.

 \vskip 0.4cm

{\bf Data Availability} The authors confirm that this manuscript has no associated data.
 \vskip 0.3cm
{\bf \large Declarations}
 \vskip 0.3cm

{\bf Conflict of interest} The authors state that there is no conflict of interest.


\end{document}